\documentclass[11pt]{ip-journal}

\usepackage{amsmath}
\usepackage{amsfonts}
\usepackage{amssymb}
\usepackage{amsthm}
\usepackage{graphicx}
\usepackage{pb-diagram}
\usepackage{epstopdf}
\usepackage{amscd}
\usepackage{bbm}
\usepackage{multirow}
\usepackage{etex}
\usepackage[mathscr]{eucal}
\usepackage{epsfig,epsf,epic}
\usepackage{pdfsync}
\usepackage{color}

\theoremstyle{plain}
\newtheorem{prop}{Proposition}
\newtheorem{theorem}[prop]{Theorem}

\newtheorem{definition}[prop]{Definition}

\newtheorem{lemma}[prop]{Lemma}
\newtheorem{remark}[prop]{Remark}
\newtheorem{assum}[prop]{Assumption}
\numberwithin{equation}{section}
\newtheorem{notation}[prop]{Notations}
\newtheorem*{conjecture-nonum}{Conjecture}
\newtheorem*{theorem-nonum}{Theorem}

\DeclareMathOperator{\Hom}{Hom}

\DeclareMathOperator{\Spec}{Spec}
\DeclareMathOperator{\supp}{supp}
\DeclareMathOperator{\trace}{tr}
\DeclareMathOperator{\vol}{vol}

\newcommand{\real}{\mathbb{R}}
\newcommand{\comp}{\mathbb{C}}

\newcommand{\inte}{\mathbb{Z}}

\newcommand{\pd}{\partial}

\newcommand{\dd}[1]{\frac{\partial}{\partial #1}}

\newcommand{\lieg}[1]{\mathcal{L}_{\nabla{#1}}}
\newcommand{\half}{\frac{1}{2}}

\newcommand{\ppd}[1]{\frac{\partial^2}{\partial #1^2}}
\newcommand{\Hess}{\nabla^2}

\newcommand{\lam}{\lambda}
\newcommand{\dr}[1]{DR_\lam (#1)}
\newcommand{\adr}[1]{DR_{\lam} (#1)_{sm}}
\newcommand{\morse}{Morse(M)}
\newcommand{\order}{\mathcal{O}}

\newcommand{\tpsi}{\tilde{\psi}}

\newcommand{\mtree}{\tilde{T}}
\newcommand{\ot}{\mathcal{T}}

\newcommand{\dist}{\rho}
\newcommand{\morprod}[1]{m^{Morse}_{#1}}
\newcommand{\deprod}[1]{m_{#1}(\lam)}
\newcommand{\deprodtree}[2]{m_{#1}^{#2}(\lam)}

\begin{document}
\title[Fukaya's conjecture on Witten's twisted $A_{\infty}$-structures]{Fukaya's conjecture on Witten's twisted $A_{\infty}$-structures}
\author[CHAN, LEUNG and MA]{Kaileung Chan, Naichung Conan Leung and Ziming Nikolas Ma}

\address{The Institute of Mathematical Sciences and Department of Mathematics, The
Chinese University of Hong Kong, Hong Kong}
\email{leung@math.cuhk.edu.hk}
\address{Department of Mathematics, The
Chinese University of Hong Kong, Hong Kong}
\email{klchan@math.cuhk.edu.hk}
\address{Department of Mathematics, The
	Chinese University of Hong Kong, Hong Kong}
\email{nikolasming@outlook.com}

\maketitle

% !TEX root = paper.tex
\begin{abstract}
The wedge product on de Rham complex of a Riemannian manifold $M$ can be pulled back to $H^*(M)$ via explicit homotopy constructed by using Green's operator which gives higher product structures. We prove Fukaya's conjecture which suggests that Witten deformation of these higher product structures have semiclassical limits as operators defined by counting gradient flow trees with respect to Morse functions, which generalizes the remarkable Witten deformation of de Rham differential from a statement concerning homology to one concerning real homotopy type of $M$. Various applications of this conjecture to mirror symmetry are also suggested by Fukaya in \cite{fukaya05}.
\end{abstract}
% !TEX root = paper.tex
\section{Introduction}\label{introduction} 
It is known that the differential graded algebra $(\Omega^*(M) ,d, \wedge)$ on a smooth manifold $M$ determines real homotopy type of $M$ (if $\pi_1(M) = 0$), a simplified homotopic classification of manifolds founded by Quillen \cite{quillen1969} and Sullivan \cite{sullivan1977}. If $M$ is a compact oriented Riemannian manifold, Hodge decomposition of the Laplacian $\Delta$ enables us to represent the cohomology of $M$ by the finite dimensional kernel $\Omega^*(M)_0 \subset \Omega^*(M)$ of $\Delta$. The real homotopy type can be captured by the homotopic pullback of the wedge product to $\Omega^*(M)_0$, which gives an $A_\infty$ structure via the homological perturbation lemma in \cite{kontsevich00}.

On the other hand, equipping $M$ with a Morse-Smale function $f$ allows us to study the cohomology of the manifold by the associated Morse complex $CM^*_f$, which is a finite dimensional vector space freely generated by critical points of $f$ equipped with the Morse differential $\delta$ defined by counting gradient flow lines of $f$. Higher product structures can be introduced to enhance the Morse complex to the Morse $A_\infty$ (pre)-category defined as in \cite{a1, fukayamorse}, involving $A_\infty$ products $\{m^{Morse}_k\}_{k\in \inte_+}$ defined by counting gradient trees. 

In Fukaya's paper \cite{fukaya05}, he conjectured that the above two $A_\infty$ product structures can be related to each others via Witten deformation. It is a differential geometric approach suggested in an influential paper \cite{witten82} by Witten to relate Hodge theory to Morse theory by deforming the exterior differential operator
$d$ with\[ d_{f} := e^{-\lam f} d e^{\lam f} = d + \lam df\wedge,\] by a Morse function $f$ with large parameter $\lam \in\real^+$. In this paper, we prove this conjecture by Fukaya. 

This machinery plays an important role in understanding SYZ transformation of open strings datum and provides a geometric explanation for Kontsevich's Homological Mirror Symmetry (Abbrev. HMS) as Fukaya stated in \cite{fukaya05}.\\

More precisely, given a Morse-Smale function $f$, we can define the Witten's twisted Laplace operator by
\begin{equation}\label{Witten_Laplacian} \Delta_{f} := d_{f} d_{f}^* + d_{f}^* d_{f}.\end{equation}
Witten considered the eigenvalues of the operator $\Delta_{f}$ lying inside an interval $[0,1)$, then the sum of the corresponding eigenspaces \newline $\Omega^*(M,\lam)_{sm} \subset \Omega^*(M)$ could be identified with the Morse complex $CM^*_f$ via a linear map (see \eqref{Morse_de Rham_map})
\begin{equation}\label{wittenmap} \phi = \Phi^{-1} : CM^*_{f} \rightarrow \Omega^*(M,\lam)_{sm}. 
\end{equation}
For any critical point $q$ of $f$, $\phi(q)$ will concentrate near $q$ when $\lam$ is large enough. Furthermore, the Witten differential $d_{f}$ is also identified with the Morse differential $m_1^{Morse}$ via $\phi$. The original proof can be found in \cite{HelSj1, HelSj2, HelSj4} while readers may see \cite{zhang} for a detailed introduction.

In order to incorporate the product structure, we have to consider more than one Morse function and the Leibniz rule associated to twisted differential is given by
$$
d_{g+h} (\alpha \wedge \beta) = d_{g}(\alpha) \wedge \beta + (-1)^{|\alpha|} \alpha \wedge d_{h}(\beta).
$$ 
This leads to the notation of the differential graded (dg) category $\dr{M}$, with objects being smooth functions on $M$. The corresponding morphism complex between two objects $f_i$ and $f_j$ is given by the Witten twisted complex $\Omega_{ij}^*(M,\lam) = (\Omega^*(M) , d_{f_{ij}})$, where $f_{ij} = f_j-f_i$. When $f_{ij}$ satisfies the Morse-Smale condition, we can define $\Omega_{ij}^*(M,\lam)_{sm}$ and a homotopy retraction $P_{ij}:\Omega^*_{ij}(M,\lam) \rightarrow \Omega_{ij}^*(M,\lam)_{sm}$ using the explicit homotopy $H_{ij} = d_{f_{ij}}^* G_{ij}$, where $G_{ij}$ is the twisted Green's operator. We can pull back the wedge product via the homotopy, making use of homological perturbation
lemma in \cite{kontsevich00}, to give a Witten's deformed $A_{\infty}$ (pre)\footnote{Roughly speaking, an $A_\infty$ pre-category allows morphisms and $A_\infty$-operations to be defined only on a subcollection of objects, called a generic subcollection, but the $A_\infty$ relation still holds whenever it is defined. Algebraic construction can be done on an $A_{\infty}$ pre-category to obtain an honest $A_\infty$ category which consists of essentially the same amount of information, and so we will restrict ourselves to $A_\infty$ pre-categories.}-category $\adr{M}$ with $A_\infty$ structure $\{m_k(\lam)\}_{k\in \inte_+}$.

For instant, suppose we have smooth functions $f_0$, $f_1$, $f_2$ and $f_3$ such that their pairwise differences are Morse-Smale, and let $\varphi_{ij}\in\Omega^*_{ij}(M,\lam)_{sm}$. Then the higher product 
$$\deprod{3}:\Omega^*_{23}(M,\lam)_{sm}\otimes\Omega^*_{12}(M,\lam)_{sm}\otimes\Omega^*_{01}(M,\lam)_{sm}\rightarrow\Omega^*_{03}(M,\lam)_{sm}$$
 is defined by
\begin{multline}
\deprod{3}(\varphi_{23},\varphi_{12},\varphi_{01}) \\
=P_{03}(H_{13}(\varphi_{23}\wedge\varphi_{12})\wedge\varphi_{01}) + P_{03}( \varphi_{23}\wedge H_{02}(\varphi_{12}\wedge\varphi_{01})).
\end{multline}
In general $m_{k}(\lam)$ will be given by a combinatorial formula involving summation over directed planar trees with $k$ inputs and $1$ output, with wedge product $\wedge$ being applied at vertices and the homotopy operator $H_{ij}$ being applied at internal edges.

Fukaya's conjecture says that the $A_\infty$ structure $\{m_k(\lam)\}_{k \in \inte_+}$, defined using the twisted Green's operators, has leading order given by $\{m_k^{Morse}\}_{k \in \inte_+}$, defined by counting gradient flow trees, via the isomorphism $\phi$.

\begin{conjecture-nonum}[Fukaya \cite{fukaya05}]
For \textit{generic} (see Definition \ref{genericassumption}) sequence of functions $\vec{f} = (f_0,\dots,f_k)$, with corresponding sequence of critical points $\vec{q} = ( q_{01}, q_{12} ,\dots , q_{(k-1)k})$, namely, $q_{ij}$ is a critical point of $f_{ij}$, we have
\begin{equation} 
\Phi (m_k(\lam)(\phi(\vec{q}))) = m^{Morse}_k(\vec{q})+ \mathcal{O}(\lam^{-1/2}). 
\end{equation}
\end{conjecture-nonum}

\begin{theorem-nonum}[Main Theorem]
Fukaya's conjecture is true.
\end{theorem-nonum}

As $A_\infty$ relations of $\{m_k(\lam)\}_{k\in \inte_+}$ are obvious from their algebraic constructions while those of $\{
m_k^{Morse}\}_{k \in \inte_+}$ require studies for boundaries of moduli spaces of gradient flow trees (see e.g. \cite{a1, fukayamorse}), we obtain an alternative proof for $A_\infty$ relations of $\{m_k^{Morse}\}_{k \in \inte_+}$ as an corollary.

The papers \cite{HelSj1, HelSj2, HelSj4, zhang} gives the proof of the main theorem for the case $k=1$, which involves detailed estimate of operator $d_f$ along gradient flow lines of a Morse function $f$ (or $f_{01}$ in our notations). 

%Starting from $k\geq 3$, our theorem involves the semi-classical analysis of the Witten twisted Green operator which is not needed in the $k=1$ case.%

For the case $k=2$, we let $f_0, f_1 , f_2$ be three smooth functions and let $q_{01}, q_{12}, q_{02}$ be critical points of $f_{01} ,
f_{12}, f_{02}$ respectively. By using the analytical techniques in \cite{HelSj1,HelSj4}, it can be proved that the Green's operators $G_{ij}$'s do not appear in the definition of $m_2(\lam)$. If we compute the leading order term in
the matrix coefficients of $m_2(\lam)$, it is essentially the integral 
\begin{equation}\label{m2coefficient} \int_{M}  m_2(\lam)(
\phi(q_{01}), \phi(q_{12})) \wedge \frac{\ast \phi(q_{02})}{\|\phi(q_{02})\|^2}. \end{equation}

Firstly, we perform a global a priori estimate to obtain $\phi(q_{ij}) \sim \mathcal{O}(e^{\lam \rho(q_{ij},\cdot)})$ (lemma \ref{eigenestimate2}), where $\rho$ is the Agmon distance defined in definition \ref{agmondistance}. Therefore, we cut off the integrand to neighborhoods of gradient trees appeared in $m_2^{Morse}$ and compute the leading order contribution from each gradient tree. The WKB approximation (lemma \ref{eigenwkb}) of the $\phi(q_{ij})$'s is used to compute the leading order contribution of \eqref{m2coefficient}.

The technicality for studying the case when $k \geq 3$ is that an WKB approximation of $G_{ij}$ along a gradient flow line of $f_{ij}$ is needed (refer to \S \ref{approximation}). More precisely, for a given form $e^{-\lam  \psi^{}_{\scalebox{.7}{$\scriptscriptstyle S$}}} \nu$, we need to study the local behaviour of the inhomogeneous Witten Laplacian equation of the form
\begin{equation}\label{wittenequation} \Delta_{ij} \zeta_E = d^*_{ij}(e^{-\lam  \psi^{}_{\scalebox{.7}{$\scriptscriptstyle S$}}} \nu) \end{equation}

along a gradient flow line segment of $f_{ij}$ from a starting point $x_S$ to an endpoint $x_E$, and obtain an approximation of $\zeta_E$ of the form $$\zeta_E \sim e^{-\lam \psi^{}_{\scalebox{.7}{$\scriptscriptstyle E$}}}\lam^{1/2}(\omega_{E,0} + \omega_{E,1}\lam^{-1/2} +
\dots).$$ 

The key step in our proof is to determine $\psi_E$ from $\psi_S$ and detailed construction is given in \S \ref{approximation}. A naive guess is $\tilde{\psi}_{E}(x) := \inf_{y} (\psi_S(y) + \dist(y,x))$ which captures the desired behaviors of $\psi_E$ near $x_E$. Unfortunately, $\tilde{\psi}_{E}(x)$ is singular along a hypersurface $U_S$ containing $x_S$ and it prohibits us to solve equation \eqref{wittenequation} iteratively in order of $\lam^{-1}$.

In order to solve \eqref{wittenequation} iteratively, we consider the minimal configuration in variational problem associated to $\inf_{y} (\psi_S(y) + \dist(y,x))$. It turns out that the point $y$ must lie on $U_S$, with a unique geodesic joining $x$ which realizes $\dist(y,x)$, for those $x$ closed enough to $x_E$. This family of geodesics $\{ \gamma_y\}_{y \in U_S}$ gives a foliation of a neighborhood of the flow line segment. Therefore we can use $\psi_E(\gamma_y(t)) = \psi_S(y) + t $ as an extension of $\tilde{\psi}_E$ across $U_S$ and solve the Equation \eqref{wittenequation} iteratively. 

We will prove the main theorem for $k=3$ by using the analysis of $G_{ij}$. The proof of the general case is similar, but more combinatorics involvoed.

%In section \ref{setting} , we introduce the basic notations and give the necessary lemma from \cite{HelSj4}.

The latter of this paper consists of two parts. The first part is the setup in \S \ref{setting} and the proof in \S \ref{proof} modulo technical analysis. The second part is the study of Witten twisted Green operator in \S\ref{approximation}.

% !TEX root = paper.tex
\section{Setting}\label{setting}
%In this section, we introduce the definitions and notations we need and state our main theorem. We begin with recalling the definition of Morse category. 

%-----------------------------Morse category-------------------------------
\subsection{Morse category}\label{sec:Morse_category}
%In this Section \ref{sec:Morse_category}, we will review some basic notations on Morse theory and Morse category from the work of \cite{a1, fukayamorse, fukaya-oh, harvey01, kontsevich00}.

We begin with a review on Morse theory and Morse category, more detail can be found in \cite{a1, fukayamorse, fukaya-oh, harvey01, kontsevich00}. The Morse category $\morse$ has the class of objects being smooth functions $f:M \rightarrow \real$, with the space of morphisms between two objects given by
\[ Hom^*_{\morse}(f_i,f_j) = CM^*(f_{ij}) = \bigoplus_{q\in Crit(f_{ij})} \comp \cdot e_q.\]
which is the Morse complex when $f_{ij}:=f_j-f_i$ satisfies the following Morse-Smale condition:

\begin{definition}\label{morsesmale}
A Morse function $f_{ij}$ is a Morse-Smale function if $V^+_p$ and $V^-_q$ intersecting transversally for any two critical points $p\neq q$ of $f_{ij}$.
\end{definition}

The Morse complex is graded by the Morse index of the corresponding critical point $q$, which is the dimension of unstable submanifold $V^-_q$. The Morse category $\morse$ is an $A_{\infty}$-category equipped with higher products $\morprod{k}$ for every $k\in \inte_+$, or simply denoted by $m_k$, which are given by counting gradient flow trees. To describe that, we first need some terminologies about directed trees.

\subsubsection{Directed trees}\label{treedefinition}

\begin{definition}\label{dtree}
A trivalent directed $k$-leafed tree $T$ is an embedded tree in $\real^2$, together with the following data:
\begin{itemize}
\item[(1)] a finite set of vertices $V(T)$;
\item[(2)] a set of internal edges $E(T)$;
\item[(3)] a set of $k$ semi-infinite incoming edges $E_{in}(T)$;
\item[(4)] a semi-infinite outgoing edge $e_{out}$.
\end{itemize}
Every vertex is required to be trivalent, i.e. it has two incoming edges and one outgoing edge.
\end{definition}

For simplicity, we will call it a $k$-tree. They are identified up to orientation preserving continuous map of $\real^2$ preserving the vertices and edges. Therefore, the topological class for $k$-trees will be finite.

Given a $k$-tree, by fixing the anticlockwise orientation of $\real^2$, we have cyclic ordering of all the semi-infinite edges. We can label connected components of $\real^2 \setminus T$ by integers $0,\dots,k$ in anticlockwise ordering, inducing a labelling on edges such that edge $e$ labelled with $ij$ will be lying between components $i$ and $j$ with the unique normal to $e$ pointing in component $i$. The labelling can be fixed uniquely by requiring the outgoing edge to be labelled by $0k$. For example, there are two different topological types for $3$-trees, with corresponding labelings for their edges as shown in the following figure \ref{3trees}.

\begin{figure}[h]
\includegraphics[scale=0.2]{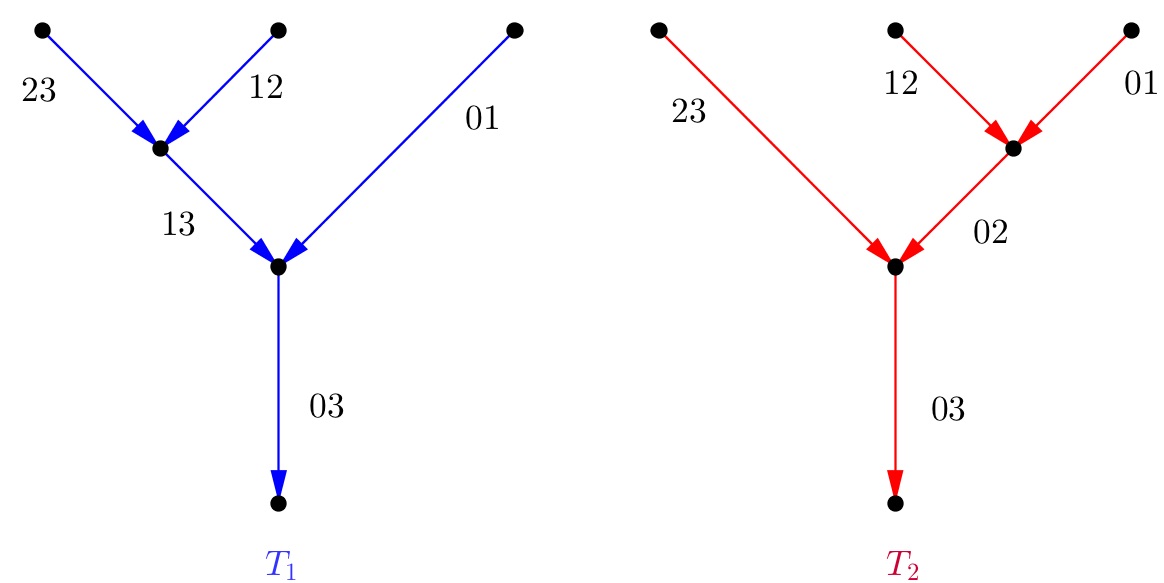}
\caption{Two different types of $3$-trees}
\label{3trees}
\end{figure}

\begin{notation}
A pair $(e,v)$, with $e$ being an edge (either finite or semi-infinite) and $v$ being an adjacent vertex, is called a flag. The unique vertex attached to the outgoing semi-infinite edge is called the root vertex. 
\end{notation}

For the purpose of Morse homology, we need the following notation of metric trees. 

\begin{definition}
A metric $k$-tree $\mtree$ is a $k$-tree together with a length function $l:E(T)\rightarrow(0,+\infty)$.
\end{definition}

Metric $k$-trees are identified up to homeomorphism preserving the length functions. The space of metric $k$-trees has finite number of components, with each component corresponding to a topological type $T$. The component corresponding to $T$, denoted by $\mathcal{S}(T)$, is a copy of $(0,+\infty)^{|E(T)|}$, where $|E(T)|$ is the number of internal edges and equals to $k-2$. The space $\mathcal{S}(T)$ can be partially compactified to a manifold with corners $(0,+\infty]^{|E(T)|}$, by allowing the length of internal edges to be infinity. In particular, it has codimension-$1$ boundary\[
\partial \overline{\mathcal{S}(T)} = \coprod_{T=T{'}\sqcup T^{''}} \mathcal{S}(T^{'})\times \mathcal{S}(T^{''}),\]
where the equation $T=T{'}\sqcup T^{''}$ means splitting the tree $T$ into $T^{'}$ and $T^{''}$ at an internal edge.

\subsubsection{Morse $A_{\infty}$ structure}
We are going to describe the product $m_k$ of the Morse category. First of all, one may notice that the morphisms between two objects $f_i$ and $f_j$ is only defined when $f_{ij}$ is Morse. Given a sequence $\vec{f} = (f_0,\dots,f_k)$ such that all the difference $f_{ij}$'s are Morse, with a sequence of points $\vec{q} = (q_{01},\dots q_{(k-1)k},q_{0k})$ such that $q_{ij}$ is a critical point of $f_{ij}$, we have the following definition of gradient flow tree.

\begin{definition}
A gradient flow tree $\Gamma$ of $\vec{f}$ with endpoints at $\vec{q}$ is a continuous map $\mathbf{f} :\mtree \rightarrow M$ such that it is an upward gradient flow lines of $f_{ij}$ when $\mathbf{f}$ is restricted on the edge labelled $ij$, the incoming edge $i(i+1)$ begins at the critical point $q_{i(i+1)}$ and the outgoing edge $0k$ ends at the critical point $q_{0k}$. 
\end{definition}

We use $\mathcal{M}(\vec{f},\vec{q})$ to denote the moduli space of gradient trees (in the case $k=1$, the moduli of gradient flow line of a single Morse function has an extra $\real$ symmetry given by translation in the domain. We will use this notation for the reduced moduli, that is the one after taking quotient by $\real$). It has a decomposition according to topological types\[
\mathcal{M}(\vec{f},\vec{q}) = \coprod_{T} \mathcal{M}(\vec{f},\vec{q})(T).\]

This space can be endowed with smooth manifold structure if we put generic assumption on the Morse sequence, which will be described as follows. For an incoming critical point $q_{i(i+1)}$, with corresponding stable submanifold $V^+_{q_{i(i+1)}}$, we define a map\[
\mathbf{f}_{T,i(i+1)} : V^+_{q_{i(i+1)}} \times \mathcal{S}(T) \rightarrow M.\]
Fixing a point $x$ in $V^+_{q_{i(i+1)}}$ together with a metric tree $\tilde{T}$, we need to determine a point in $M$. First, suppose $v$ is the vertex connected to the edge labelled $i(i+1)$, there is a unique sequence of internal edges $(e_1,\dots,e_{k-2})$ connecting $v$ to the root vertex $v_r$. To determine the image of $x$ under our function, we apply Morse gradient flow with respect to Morse function associated to $e_{j}$'s for time $l(e_{j})$ to $x$ consecutively according to the sequence $(e_1,\dots,e_{k-2})$.

The maps are then put together to give a map
\begin{equation}
\mathbf{f}_T : V^-_{q_{0k}}\times V^+_{q_{(k-1)k}} \times \cdots \times V^+_{q_{01}} \times \mathcal{S}(T)  \rightarrow \prod_{k+1} M,
\end{equation}
where we use the embedding $\iota:V_{q_{0k}}^- \rightarrow M$ for the first component. 
\begin{definition}\label{genericassumption}
A Morse sequence $\vec{f}$ is said to be generic if the image of $\mathbf{f}_{T}$ intersects transversally with the diagonal submanifold $\Delta \cong M\hookrightarrow M^{k+1}$, for any sequence of critical points $\vec{q}$ and any topological type $T$.
\end{definition}

When the sequence is generic, the moduli space $\mathcal{M}(\vec{f},\vec{q})$ is of dimension\[
\dim_{\real}(\mathcal{M}(\vec{f},\vec{q}) )= \deg(q_{0k}) - \sum_{i=0}^{k-1} \deg(q_{i(i+1)}) + k-2,\]
where $\deg(q_{ij})$ is the Morse index of the critical point. Therefore, we can define $m_k^{Morse}$, or simply denoted by $m_k$, using the signed count $\# \mathcal{M}(\vec{f},\vec{q})$ of points in $\dim_{\real}(\mathcal{M}(\vec{f},\vec{q}) )$ when it is of dimension $0$. In order to have a signed count, we have to fix an orientation of the space $\mathcal{M}(\vec{f},\vec{q})$ which will be discussed later in definition \ref{flowtreeorientation2}.\\

We now give the definition of the higher products in the Morse category.
\begin{definition}
Given a generic Morse sequence $\vec{f}$ with sequence of critical points $\vec{q}$, we define 
$$
m_k : CM_{k(k-1)}^* \otimes \cdots \otimes CM^*_{01} \rightarrow CM^*_{0k}
$$
by
\begin{equation}
\langle m_k(q_{(k-1)k},\dots,q_{01}), q_{0k}\rangle =  \# \mathcal{M}(\vec{f},\vec{q}),
\end{equation}
when\[
\deg(q_{0k}) - \sum_{i=0}^{k-1} \deg(q_{i(i+1)}) + k-2 = 0.\]
Otherwise, the $m_k$ is defined to be zero.
\end{definition}

Since $\morprod{k}$ can only be defined when $\vec{f}$ is a Morse sequence satisfying the generic assumption in definition \ref{genericassumption}, the Morse category is indeed an $A_\infty$ pre-category. Readers may see \cite{a1, fukayamorse} for the detail of algebraic construction to retrive an honest $A_\infty$ category.

%We will not go into detail about the algebraic problem on getting an honest category from this structures. For details about this, readers may see \cite{a1, fukayamorse}.

%-------------------From de Rham to Morse-----------------------------
\subsection{Witten's twisted de Rham category}\label{sec:Witten_twisted_derham}
Given a compact oriented Riemannian manifold $M$, we can construct the de Rham category $\dr{M}$ depending on $\lam$. Objects of this category are again smooth functions, while the space of morphisms between $f_i$ and $f_j$ is\[
\Hom^*_{\dr{M}}(f_i,f_j) = \Omega^*(M),\]
with the twisted differential $ d+ \lam df_{ij}\wedge$, where $f_{ij}:=f_j-f_i$. The composition of morphisms is defined to be the wedge product of differential forms on $M$. This composition is associative and hence the resulted category is a dg category. We denote the complex corresponding to $\Hom^*_{\dr{M}}(f_i,f_j)$ by $\Omega^*_{ij}(M,\lam)$ and the differential $d+\lam df_{ij}\wedge$ by $d_{ij}$.

To relate $\dr{M}$ and $\morse$, we need to apply homological perturbation to $\dr{M}$. Fixing two functions $f_i$ and $f_j$, we consider the Witten Laplacian
$$
\Delta_{ij}:=d_{ij} d^*_{ij} + d^*_{ij} d_{ij},
$$
where $d_{ij}^* =  d^* + \lam \iota_{\nabla f_{ij}}$. We denote the span of eigenspaces with eigenvalues contained in $[0,1)$ by $\Omega^*_{ij}(M,\lam)_{sm}$.

If the function $f_{ij}$ is a Morse-Smale function (see definition \ref{morsesmale}), it is proved in \cite[Appendix: On the Thom-Smale complex]{La92} that the closure $\overline{V^+_q}$ and $\overline{V^-_q}$ have a structure of submanifold with conical singularities. Using this result, one can define the following map as in \cite{zhang, harvey01}\[
\Phi = \Phi_{ij} : \Omega^*_{ij}(M,\lam)_{sm} \rightarrow CM^*(f_{ij})\]
given by 
\begin{equation}\label{Morse_de Rham_map}
\Phi(\alpha) = \sum_{p \in\;Crit(f_{ij})} \left( \int_{V^{-}_p} e^{\lam f_{ij}} \alpha \right) \cdot e_p
\end{equation}
which is an isomorphism identifying $d_{ij}$ with Morse differential $m_1$ when $\lam$ large enough. 

\begin{remark}
This identification gives a connection on the family of vector space $\Omega^*_{ij}(M,\lam)_{sm}$ parametrized by $\lam$ by declaring the basis $e_p$ associated to critical point of $f_{ij}$ to be flat. Equivalently, it is the same as defining
$$
\nabla_{\dd{\lam}} \alpha(\lam) = P_{ij}(e^{-\lam f_{ij}} \dd{\lam} e^{\lam f_{ij}} \alpha(\lam)),
$$
for $\alpha(\lam) \in \Omega^*_{ij}(M,\lam)_{sm}$, where $P_{ij} :\Omega^*_{ij}(M,\lam) \rightarrow \Omega^*_{ij}(M,\lam)_{sm} \hookrightarrow \Omega^*_{ij}(M,\lam)$ is the idempotent associated to the orthogonal projection on $\Omega^*_{ij}(M,\lam)_{sm}$.
\end{remark}

It is natural to ask whether the product structures of two categories are related as $\lam \rightarrow \infty$, and the answer is definite. The first observation is that the Witten's approach indeed produces an $A_{\infty}$ pre-category, denoted by $\adr{M}$, with $A_{\infty}$ structure $\{\deprod{k}\}_{k\in \inte_+}$. It has the same class of objects as $\dr{M}$. However, the space of morphisms between two objects $f_i$, $f_j$ is taken to be $\Omega^*_{ij}(M,\lam)_{sm}$, with $m_1(\lam)$ being the restriction of $d_{ij}$ on the eigenspace $\Omega^*_{ij}(M,\lam)_{sm}$.\\

The natural way to define $\deprod{2}$ for any three objects $f_0$, $f_1$ and $f_2$ is the operation given by\[
\begin{CD} 
\Omega^*_{12}(M,\lam)_{sm} \otimes \Omega^*_{01}(M,\lam)_{sm}	@>\wedge>> \Omega^*_{02}(M,\lam) @>P_{02}>>  \Omega^*_{02}(M,\lam)_{sm}, \\
\end{CD} \]
$P_{ij} :\Omega^*_{ij}(M,\lam) \rightarrow \Omega^*_{ij}(M,\lam)_{sm} \hookrightarrow \Omega^*_{ij}(M,\lam)$ is the idempotent associated to the orthogonal projection to $\Omega^*_{ij}(M,\lam)_{sm}$.

Notice that $\deprod{2}$ is not associative, and we need a $\deprod{3}$ to record the non-associativity. Suppose that  $G_{ij}^0$ is the Green's operator corresponding to Witten Laplacian $\Delta_{ij}$, we let
\begin{equation}\label{smallG}
G_{ij}=(I-P_{ij})G_{ij}^0
\end{equation}
and
\begin{equation}\label{H_ij}
H_{ij}=d_{ij}^*G_{ij}.
\end{equation}
Then $H_{ij}$ is a linear operator from $\Omega^*_{ij}(M,\lam)$ to $\Omega^{*-1}_{ij}(M,\lam)$ such that\[
d_{ij}H_{ij}+H_{ij}d_{ij}=I-P_{ij}.\]
Namely $\Omega^*_{ij}(M,\lam)_{sm}$ is a homotopy retract of $\Omega^*_{ij}(M,\lam)$ with homotopy operator $H_{ij}$. Suppose $f_0$, $f_1$, $f_2$ and $f_3$ are smooth functions on $M$ and $\varphi_{ij}\in\Omega^*_{ij}(M,\lam)_{sm}$, the higher product 
$$\deprod{3}:\Omega^*_{23}(M,\lam)_{sm}\otimes\Omega^*_{12}(M,\lam)_{sm}\otimes\Omega^*_{01}(M,\lam)_{sm}\rightarrow\Omega^*_{03}(M,\lam)_{sm}$$
 is defined by
\begin{multline}\label{m3}
\deprod{3}(\varphi_{23},\varphi_{12},\varphi_{01}) = \\
P_{03}(H_{13}(\varphi_{23}\wedge\varphi_{12})\wedge\varphi_{01}) + P_{03}( \varphi_{23}\wedge H_{02}(\varphi_{12}\wedge\varphi_{01})).
\end{multline}

In general, the construction of $m_k(\lam)$ can be described using $k$-tree. For $k\geq2$, we decompose $\deprod{k}:=\sum_T \deprodtree{k}{T}$, where $T$ runs over all topological types of $k$-trees. 
$$\deprodtree{k}{T} : \Omega_{(k-1)k}^*(M,\lam)_{sm}\otimes \cdots \otimes \Omega_{01}^*(M,\lam)_{sm} \rightarrow \Omega_{0k}^*(M,\lam)_{sm}$$ 
is an operation defined along the directed tree $T$ by 
\begin{itemize}
\item[(1)] applying wedge product $\wedge$ to each interior vertex;
\item[(2)] applying homotopy operator $H_{ij}$ to each internal edge labelled $ij$;
\item[(3)] applying projection $P_{0k}$ to the outgoing semi-infinite edge.
\end{itemize}
The following graph shows the operation associated to the unique $2$-tree.
\begin{figure}[h]
\centering
\includegraphics[scale=0.45]{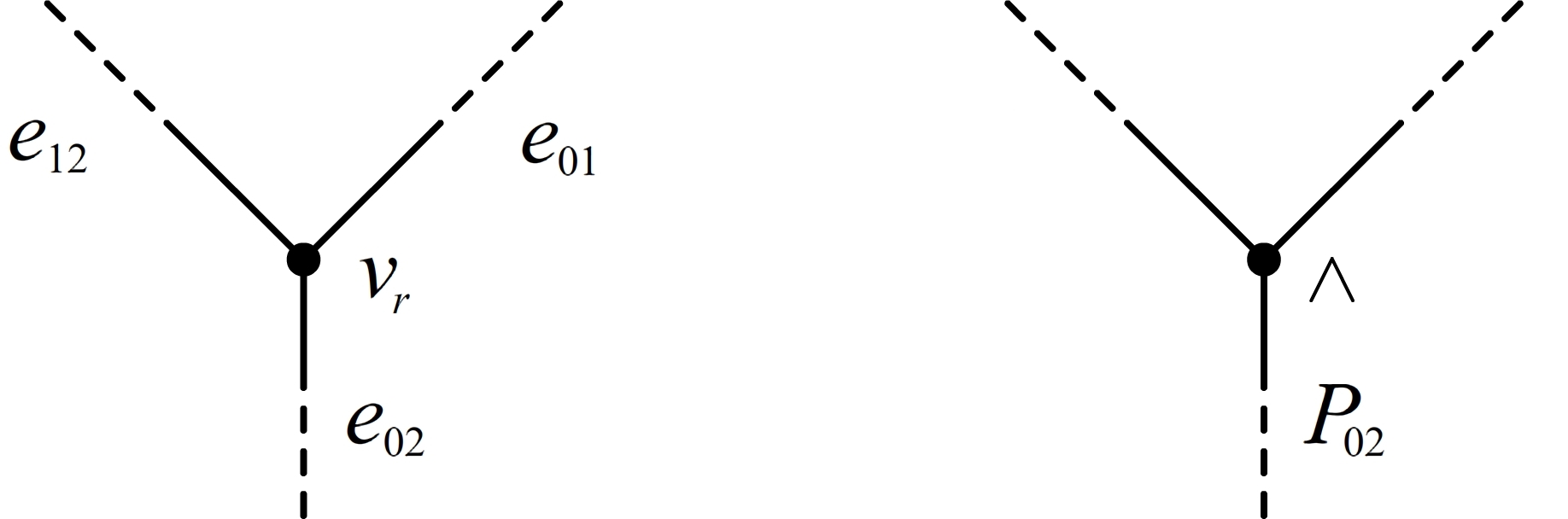}
\caption{The unique 2-tree and the corresponding assignment of operators for defining $\deprod{2}$.}
\label{2tree}
\end{figure}

The higher products $\{m_k(\lam)\}_{k \in \inte_+}$ satisfies the generalized associativity relation, called $A_\infty$ relation. One may treat the $A_\infty$ product as a pullback of the wedge product under the homotopy retract $P_{ij}: \Omega^*_{ij}(M,\lam) \rightarrow \Omega^*_{ij}(M,\lam)_{sm}$. This proceed is called the homological perturbation lemma. For details about this construction, readers may see \cite{kontsevich00}.  As a result, we obtain an $A_{\infty}$ pre-category $\adr{M}$. \\

With the above notations, we restate our main theorem as the following:
\begin{theorem}[Main Theorem]\label{main_theorem}
Given smooth functions $f_0,\ldots,f_k$ satisfying the generic assumption in definition \ref{genericassumption}, with $q_{ij} \in CM^*(f_{ij})$ be corresponding critical points, there exist $\lam_0>0$ and $C_0>0$, such that for all $i\neq j$, $\phi = \Phi^{-1}:CM^*(f_{ij}) \rightarrow \Omega^*_{ij}(M,\lam)_{sm}$ is an isomorphism when $\lam>\lam_0$. Furthermore,  
\begin{equation*}
\Phi( \deprod{k}(\phi(q_{(k-1)k}),\dots,\phi(q_{01}))) = \morprod{k}(q_{(k-1)k},\dots,q_{01})+R(\lam),
\end{equation*}
with $|R(\lam) | \leq C_0 \lam^{-1/2}$.
\end{theorem}

\begin{remark}
The constants $C_0$ and $\lam_0$ depend on the functions $f_0,\dots,f_k$. In general, we cannot choose fixed constants such that the above statement holds for all $\deprod{k}$ and all sequences of functions.
\end{remark}

\begin{remark}
We would like to emphasize the relation between the main theorem and SYZ Mirror Symmetry. Let $T^*M$ be the cotangent bundle of a manifold $M$ which equips the canonical symplectic form $\omega_{can}$, and let $L_i = \Gamma_{df_i}$ be Lagrangian sections. Then a critical point $q_{ij}$ of $f_{ij}$ can be identified with $q_{ij} \in L_i \pitchfork L_j$. Applying the theorem of Fukaya-Oh \cite{fukaya-oh}, the Morse $A_\infty$ operation $m_k^{Morse}$ is equivalent to Floer theoretical $A_\infty$ operations counting holomorphic disks. In the simplest situation concerning $(T^*M,\omega_{can})$, the Witten's twisted de Rham category $\adr{M}$ is related to the Floer theory on $(T^*M,\omega_{can})$ via our main theorem \ref{main_theorem} and Fukaya-Oh's theorem. In more general situation, one expects the correspondence will be one of the ingredients for realizing HMS geometrically. 
\end{remark}

% !TEX root = paper.tex

\section{Proof of Main Theorem}\label{proof}
We fix a generic sequence $\vec{f}$ of $k+1$ functions, with corresponding sequence of critical points $\vec{q}$. First of all, we have \[
\deg(\deprod{k}(\phi(q_{(k-1)k}),\dots,\phi(q_{01})) = \sum_{i=0}^{k-1} \deg(q_{i(i+1)})-k+2,\]
so $\langle \deprod{k}(\phi(q_{(k-1)k}),\dots,\phi(q_{01}), \phi(q_{0k}) \rangle$ is non-trivial only when the equality
\begin{equation}\label{degree_eq}
\sum_{i=0}^{k-1} \deg(q_{i(i+1)})-k+2= \deg(q_{0k})
\end{equation}
holds, which is exactly the condition for $\morprod{k}$ in the Morse category to be non-trivial. We will therefore assume condition \eqref{degree_eq} and consider the integral
\begin{equation*}
\int_{M} \langle \deprod{k}(\phi(q_{(k-1)k}),\dots,\phi(q_{01})), \frac{\phi(q_{0k})}{\|\phi(q_{0k})\|^2} \rangle vol_g.
\end{equation*}
Recall that each directed tree $T$ gives an operation $m_k^T(\lam)$ and $\deprod{k} = \sum_{T} \deprodtree{k}{T}$ which is also the case in Morse category. Therefore, we just have to consider each $\deprodtree{k}{T}$ separately.\\

%----------------------------result for single Morse function-----------------

\subsection{Results for a single Morse function}
We start with stating the results of Witten deformation for a single Morse function $f_{ij}$ which we will assume it to be Morse-Smale as in definition \ref{morsesmale}. These results come from \cite{HelSj1, HelSj2, HelSj4, zhang}, with a few modifications to fit our content. We introduce the Agmon distance $\dist_{ij}$ and lemma \ref{lem:agmon_dist_flow_line} is just \cite[Lemma A2.2]{HelSj4}.

\begin{definition}\label{agmondistance}
	For a Morse function $f_{ij}$, the Agmon distance $\dist_{ij}$, or simply denoted by $\dist$ when no confusion occurs, is the distance function with respect to the degenerated Riemannian metric $\langle \cdot, \cdot \rangle_{f_{ij}} = |df_{ij}|^2 \langle \cdot, \cdot \rangle$, where $\langle \cdot, \cdot \rangle$ is the background metric. 
\end{definition}

\begin{lemma}\label{lem:agmon_dist_flow_line}
	We have $\dist_{ij}(x,y) \geq f_{ij}(x) - f_{ij}(y)$ with equality holds if and only if $y$ is connected to $x$ via a generalized flow line $\gamma : [0,1] \rightarrow M$ with $\gamma(0) = y$ and $\gamma(1) = x$. Here a generalized flow line means that $\gamma$ is continuous, and there is a partition $0=t_0 <t_1< \cdots <t_l = 1$ such that $\gamma|_{(t_r,t_{r+1})}$ is a reparameterization of a gradient flow line of $f_{ij}$ and $\gamma(t_r) \in Crit(f_{ij})$ for $0<r<l$. 
\end{lemma}

Readers may see \cite{HelNi} for more of its basic properties. The Agmon distance is closely related to the Witten's Laplacian, or more preciely the corresponding Green's operator associated to it by the following lemma which is a variant of \cite[Proposition 2.2.5]{HelSj2} in our current situation (readers may also see \cite[Proposition 6.5]{DiSj}).

\begin{lemma}\label{resolventestimate}
	Let $\gamma \subset \comp$ to be a subset whose distance from $Spec(\Delta_{ij})$ is bounded below by a constant. For any $j\in \inte_+$ and $\epsilon>0$, there is $k_j \in \inte_+$ and $\lam_0 = \lam_0(\epsilon)>0$ such that for any two points $x_0, y_0 \in M$, there exist neighborhoods $V$ and $U$ (depending on $\epsilon$) of $x_0$ and $y_0$ respectively, and $C_{j,\epsilon}>0$ such that for any $z \in \gamma$ we have
	\begin{equation}
	\| \nabla^j ((z-\Delta_{ij})^{-1}u)\|_{C^{0}(V)} \leq C_{j,\epsilon} e^{-\lam(\dist_{ij}(x_0,y_0)-\epsilon)} \| u\|_{W^{k_j,2}(U)},
	\end{equation}
	for all $\lam>\lam_0$ and $u \in C^0_c(U)$, where $W^{k,p}$ refers to the Sobolev norm.
\end{lemma}

We will also need modified version of the resolvent estimate for $G_{ij}$, which can be obtained by applying the original resolvent estimate to the the formula
\begin{equation}
G_{ij}(u) = \oint_{\gamma} z^{-1}(z-\Delta_{ij})^{-1} u.
\end{equation}

\begin{lemma}\label{resolventlemma}
	For any $j\in \inte_+$ and $\epsilon>0$, there exist $k_j \in \inte_+$ and $\lam_0 = \lam_0(\epsilon)>0$ such that for any two points $x_0, y_0 \in M$, there exist neighborhoods $V$ and $U$ (depending on $\epsilon$) of $x_0$ and $y_0$ respectively, and $C_{j,\epsilon}>0$ such that
	\begin{equation}
	\| \nabla^j (G_{ij}u)\|_{C^{0}(V)} \leq C_{j,\epsilon} e^{-\lam(\dist_{ij}(x_0,y_0)-\epsilon)} \| u\|_{W^{k_j,2}(U)},
	\end{equation}
	for all $\lam<\lam_0$ and $u \in C^0_c(U)$, where $W^{k,p}$ refers to the Sobolev norm.
\end{lemma}

Under the Morse-Smale condition, one can prove the following spectral gap in the twisted de Rham complex which follows from \cite[Lemma 1.6]{HelSj4} and \cite[Proposition 1.7]{HelSj4}. 
\begin{lemma}\label{lem:spectral_gap}
For each $f_{ij}$, there exist $\lam_0>0$ and constants $c, C>0$ such that \[
\Spec(\Delta_{ij})\cap [ce^{-c\lam}, C\lam^{1}) = \emptyset,\]
for $\lam>\lam_0$.
\end{lemma}
Recall that in Section \ref{sec:Witten_twisted_derham} we have denoted the subspace of $\Omega^*_{ij}(M,\lam)$ with eigenvalues lying in $[0, 1)$ by $\Omega^*_{ij}(M,\lam)_{sm}$, and it is closely related to the Morse complex $CM_{ij}^*$ introduced in Section \ref{sec:Morse_category}. 

Furthermore, we have the following theorem on Witten deformation on the level of chain complexes which is \cite[Theorem 6.9]{zhang} in our current situation.
\begin{theorem}[\cite{HelSj4, zhang}]\label{thm:witten_map_iso}
The map $\Phi = \Phi_{ij} : \Omega^*_{ij}(M,\lam)_{sm} \rightarrow CM_{ij}^*$ in equation \eqref{Morse_de Rham_map} is a chain isomorphism for $\lam$ large enough.
\end{theorem}

\begin{notation}
We will denote the inverse by $\phi = \phi_{ij}$ and write $\phi(q) \in \Omega^*_{ij}(M,\lam)_{sm}$ for a critical point $q$ of $f_{ij}$. 
\end{notation}

Since we are dealing with the case that the background metric which is not flat near critical points of $f_{ij}$, we will need a combination of techniques from \cite{HelSj4, zhang} to prove Theorem \ref{thm:witten_map_iso}, which we will briefly indicate as follows. Readers may take this part for granted, skip the following section \ref{sec:witten_map_iso} and go directly into section \ref{sec:property_of_phi_q}. 

\subsubsection{Sketch of proof for Theorem \ref{thm:witten_map_iso} using results from \cite{HelSj4}} \label{sec:witten_map_iso}
We use $Crit^*(f_{ij})$ to denote the set of critical points of $f_{ij}$ with $*$ being the degree of the critical point.  For each $q\in Crit^l(f_{ij})$, we let\[
M_{q,\eta} = M \setminus \bigcup_{ p \in Crit^l(f_{ij})\setminus \{q\}} B(p,\eta),\]
where $B(p,\eta)$ is the open ball centered at $p$ with radius $\eta$ with respect to the Agmon metric, and $M_{q,\eta}$ is a manifold with boundary when $\eta$ is sufficiently small. 

For each $q \in Crit^l(f_{ij})$, we use $\Omega_{ij}^l(M_{q,\eta},\lam)$ to denote the space of differential $l$-forms with Dirichlet boundary condition, with Witten Lacplacian $\Delta_{ij,q}$ acting on it. The spectral gap Lemma \ref{lem:spectral_gap} holds for $\Delta_{ij,q}$ as well and since there is only one critical point of degree $l$ in $M_{q,\eta}$, the eigenspaces of $\Delta_{ij,q}$ with small eigenvalues is $1$-dimensional.  We have the following decay estimate which is \cite[Theorem 1.4]{HelSj4}.
\begin{lemma}\label{lem:eigenestimate1}
For any $\epsilon$, $\eta>0$ small enough, we have $\lam_0=\lam_0(\epsilon,\eta)>0$ such that when $\lam>\lam_0$, $\Delta_{ij,q}$ has one dimensional eigenspace in $[0,1)$. If we let $\varphi_q\in\Omega_{ij}^l(M_{q,\eta},\lam)$ be the corresponding unit length eigenform, we have
\begin{equation}
\varphi_q  = \mathcal{O}_\epsilon(e^{-\lam (\dist_{ij}(q,x)-\epsilon)}),
\end{equation}
where $\mathcal{O}_{\epsilon}$ stands for $C^{0}$ bound with a constant depending on $\epsilon$. Same estimate holds for any $k$-th derivative $\nabla^k\varphi_q$ as well.
\end{lemma}

We construct $\hat{\varphi}_q \in \Omega^*(M,\lam)_{sm}$, depending on $\lam$ and $\eta$ as follows. For each critical point $p$, we take a cut off function $\theta_p$ such that $\theta_p \equiv 1$ in $\overline{B(p,\eta)}$ and compactly supported in $B(p,2\eta)$. Given a critical point $q \in Crit^l(f_{ij})$, we let \[
\chi_q = 1-\sum_{p \in Crit^l(f_{ij}) \setminus \{q\}}\theta_p.\]
\begin{definition}
For sufficiently small $\eta>0$ and large $\lam$, we define
\begin{equation}
\hat{\varphi}_q := P_{ij} \chi_q \varphi_q,
\end{equation}
where $P_{ij} : \Omega^*_{ij}(M,\lam) \rightarrow \Omega^*_{ij}(M,\lam)_{sm}\hookrightarrow \Omega^*_{ij}(M,\lam)$ is the idempotent associated to the projection to the small eigenspace.
\end{definition}

The difference between $\hat{\varphi}_q$ and $\varphi_q$ is computed in \cite[Lemma 2.1.1]{HelSj2}, which shows that $\hat{\varphi}_q$ satisfies the same estimate in Lemma \ref{lem:eigenestimate1}. Furthermore, \cite[Proposition 1.3]{HelSj4} (reader may also see \cite[Theorem 3.6]{DiSj}) together with \cite[Theorem 5.8]{HelSj1} lead the following WKB approximation of $\hat{\varphi}_q$ (see remark \ref{wkbmeaning}).
\begin{lemma}\label{lem:eigenwkb}
	For $\eta$ small enough and $\lam$ large enough, there is a WKB approximation of $\hat{\varphi}_q$ of the form
	\begin{equation}
	\hat{\varphi}_q \sim \lam^{\frac{\deg(q)}{2}} e^{-\lam \dist_{ij}(q,x)}( \alpha_{q,0}+\alpha_{q,2}\lam^{-1}+\dots + \alpha_{q,2j} \lam^{-j} +\dots),
	\end{equation}
	in a neighborhood $W$ of $V^+_q \cup V^-_q$.
\end{lemma}

Lemma \ref{lem:eigenestimate1}, the WKB approximation in the above Lemma \ref{lem:eigenwkb} combines together with the explicit description of the leading term $\alpha_{q,0}$ in \cite[Theorem 2.5]{HelSj4} and it gives us the explicit computation of $\Phi(\hat{\varphi}_q)$ as follows.

\begin{lemma}\label{lem:approximated_inverse_computation}
	For sufficiently small $\eta$ and large $\lam$, we have $\int_{V^-_q} e^{\lam f_{ij}}\hat{\varphi}_q \neq 0$. Suppose that we renormalize $\hat{\phi}_q:= \frac{\hat{\varphi}_q}{(\int_{V^-_q} e^{\lam f_{ij}}\hat{\varphi}_q )}$, then we have
	$$
	\int_{V_p^-} e^{\lam f_{ij}} \hat{\phi}_q = \delta(p,q) - R(p,q),
	$$
	where $R(p,q) = 0$ if $p=q$ and $R(p,q) = \mathcal{O}_\epsilon(e^{-\lam (c(p,q)-\epsilon)})$ with 
	$$
	c(p,q) := \dist_{ij}(p,q) - (f_{ij}(p) - f_{ij}(q)) >0 
	$$
	from the Morse-Smale condition.
	\end{lemma}

In particular, if we define $\hat{\phi}:CM^*_{ij} \rightarrow \Omega^*(M,\lam)_{sm}$ by $q \mapsto \hat{\phi}_q$, then we have $\Phi \circ \hat{\phi} = id - R$ with $R = \mathcal{O}(e^{-c\lam})$ for some $c>0$. This tells us that $\Phi$ is an isomorphism when is $\lam$ large enough and $\hat{\phi}$ is an approximation of $\phi$.

\subsubsection{Exponential decay of $\phi(q)$} \label{sec:property_of_phi_q}
For a critical point $q\in Crit^*(f_{ij})$, $\phi(q) \in \Omega^*(M,\lam)_{sm}$ has certain exponential decay measured by the Agmon distance from the critical point $q$ as in lemma \ref{eigenestimate2}. It is also a consequence of lemma \ref{lem:eigenestimate1} and lemma \ref{lem:approximated_inverse_computation}. 

\begin{lemma}\label{eigenestimate2}
For any $\epsilon$, there exists $\lam_0=\lam_0(\epsilon)>0$ such that for $\lam>\lam_0$, we have
\begin{equation}
\phi(q) = \mathcal{O}_{\epsilon}(e^{-\lam(\psi_q(x)-\epsilon)}),
\end{equation}
and the same estimate holds for the derivatives of $\phi_{ij}(q)$. Here, $\mathcal{O}_\epsilon$ refers to the dependence of the constant $\epsilon$ and $\psi_q(x) = \dist_{ij}(q,x)+f_{ij}(q)$.
\end{lemma}

\begin{remark}\label{eigenwkbremark}
We write $g^+_{q} = \psi_{q} -f_{ij}$ and $g^-_{q} = \psi_{q} + f_{ij}$ which are nonnegative smooth functions with zero sets $V_{q}^+$ and $V_{q}^-$ respectively, and Bott-Morse in a neighborhood $W$ of $V_{q}^+ \cup V_q^-$.
%We write $g^+_{q} = \psi_{q} -f_{ij}$, which is a nonnegative function with zero set $V_{q}^+$ that is smooth and Bott-Morse in a neighborhood $W$ of $V_{q}^+ \cup V_q^-$. Similarly, we write $g^-_{q} = \psi_{q} + f_{ij}$ which is a nonnegative function with zero set $V_{q}^-$ that is smooth and Bott-Morse in \textcolor{red}{(neighborhood of ?)}$W$.
More properties of the functions $g^{\pm}_q$ can be found right below \cite[equation (2.8)]{HelSj4}

In this case, we write 
\begin{eqnarray*}
e^{\lam f_{ij}}\phi(q) & = &\mathcal{O}_{\epsilon}(e^{-\lam(g^+_{q}-\epsilon)}),\\
e^{-\lam f_{ij}}\ast \phi(q)/\|\phi(q)\|^2 & = &  \mathcal{O}_{\epsilon}(e^{-\lam(g^-_{q}-\epsilon)}).\\
\end{eqnarray*}
\end{remark}

Furthermore, we notice that the normalized basis $\phi(q)/\|\phi(q)\|$'s are almost orthonormal basis as in the following lemma, which is a direct consequence of lemma \ref{eigenestimate2}.
\begin{lemma}\label{lem:orthonormal_basis}
There exist $C,c>0$ and $\lam_0$ such that when $\lam > \lam_0$ such that
$$
\langle \frac{\phi(p)}{\|\phi(p)\|} , \frac{\phi(q)}{\|\phi(q)\|} \rangle = \delta_{pq} + Ce^{-c\lam}.
$$
\end{lemma}

\subsubsection{WKB approximation for $\phi(q)$}\label{wkbeigenform}
Restricting on a sufficiently small neighborhood $W$ containing $V^+_{q} \cup V^-_{q}$, the above decay estimate of $\phi(q)$'s from \cite{HelSj4} can be improved from an error of order $\mathcal{O}_\epsilon(e^{\epsilon \lam})$ to $\mathcal{O}(\lam^{-N})$ for an arbitrary $N \in \inte_+$ which follows from a similar WKB approximation in lemma \ref{lem:eigenwkb}. 
\begin{lemma}\label{eigenwkb}
There is a WKB approximation of $\phi(q)$ of the form
\begin{equation}
\phi(q) \sim \lam^{\frac{\deg(q)}{2}} e^{-\lam \psi_{q}}( \omega_{q,0}+\omega_{q,2}\lam^{-1}+\dots + \omega_{q,2j} \lam^{-j} +\dots),
\end{equation}
in a neighborhood $W$ of $V^+_q \cup V^-_q$. 
\end{lemma}

\begin{remark}\label{wkbmeaning}
	The precise meaning of this WKB approximation is given in section \ref{L2approx}. Roughly speaking, it is a $C^\infty$ approximation in order of $\lam$ on every compact subset of $W$.
\end{remark}

Furthermore, the integral of the leading order term $\omega_{q,0}$ in the normal direction to the stable submanifold $V_q^+$ is computed in \cite[Theorem 2.5]{HelSj4}. 
\begin{lemma}\label{eigenwkbcal}
Fixing any point $x \in V_{q}^+$ and a cutoff function $\chi$ such that $\chi \equiv 1$ around $x$ compactly supported in $W$, we take any closed submanifold (possibly with boundary) $NV_{q,x}^+$ of $W$ intersecting transversally with $V_q^+$ at $x$. Then, we have
\begin{equation*}
\lam^{\frac{\deg(q)}{2}}\int_{NV_{q,x}^+} e^{-\lam g^+_{q}}\chi \omega_{q,0} = 1+\mathcal{O}(\lam^{-1}).
\end{equation*}
Similarily, we have
\begin{equation*}
\frac{\lam^{\frac{\deg(q)}{2}}}{\| \phi(q) \|^2}\int_{NV_{q,x}^-} e^{-\lam g^-_{q}}\chi (\ast \omega_{q,0}) = 1+\mathcal{O}(\lam^{-1}), 
\end{equation*}
for any point $x \in V_{q}^-$, with $NV_{q,x}^-$ intersecting transversally with $V_q^-$ at $x$.
\end{lemma}

So far we have been considering a fixed Morse function $f_{ij}$. From now on, we will consider a fixed generic sequence $\vec{f}$ with corresponding sequence of critical points $\vec{q}$ as in the beginning of section \ref{proof}.
\begin{notation}
We use $q_{ij}$ to denote a fixed critical point of $f_{ij}$. $\phi(q_{ij})$ associated to $q_{ij}$ is abbreviated by $\phi_{ij}$.
\end{notation}

% !TEX root = paper.tex
We will use the result in the previous section to localize the integral
\begin{equation}\label{mkintegral0}
\int_{M}  \deprod{k}(\phi(q_{(k-1)k}),\dots,\phi(q_{01}))\wedge \frac{\ast \phi(q_{0k})}{\|\phi(q_{0k})\|^2} 
\end{equation}
to gradient flow trees, when the degree condition (\ref{degree_eq}) holds. 

\subsection{Proof of $m_2$}\label{sec:m_2_case}

\subsubsection{Apriori estimate for $m_2(\lam)$ case}\label{sec:apriori_m2_case}
We begin with the simplest case $m_2(\lam)$ which does not involve any homotopy operator $H_{ij}$. There is an unique $2$-tree $T$ with a unique vertex $v_r$ as shown in Figure \ref{2tree}. According to the combinatorics of $T$, we define $\vec{\dist}_{T} : M = M^{| V(T)|} \rightarrow \real_+$ which is given by
\begin{equation*}
	\vec{\dist}_{T}(x_{v_r}) =\dist_{01}(x_{v_r},q_{01})+\dist_{12}(x_{v_r},q_{12})+\dist_{02}(x_{v_r},q_{02}).
\end{equation*}
It can be treated as the length of the geodesic tree of type $T$ with unique interior vertices $x_{v_r}$ and end points of semi-infinite edges $e_{ij}$'s laying on $q_{ij}$'s.

By lemma \ref{lem:agmon_dist_flow_line}, we learn that $\dist_{ij}(x,y) \geq f_{ij}(x) - f_{ij}(y)$ and equality holds if and only if $y$ is connected to $x$ through a generalized flow line of $f_{ij}$. Notice that $\vec{\dist}_{T}(x_{v_r}) \geq A$ where
\begin{equation}\label{eqn:sym_area_constant_m2}
A:= f_{02}(q_{02}) - f_{01}(q_{01}) - f_{12}(q_{12}),
\end{equation} and the equality holds if and only if $x_{v_r}$ is one of the interior vertices of a gradient flow tree of the type $T$. We will only consider gradient flow trees instead of \footnote{Here generalized gradient trees refers to continuous map from $T$ to $M$ such that the restriction to each edge being a generalized gradient flow line mentioned in lemma \ref{lem:agmon_dist_flow_line}}generalized gradient trees since we assume the sequence of Morse functions $\vec{f}$ satisfies the generic assumption as in definition \ref{genericassumption}.

From lemma \ref{eigenestimate2}, we notice the integrand 
\begin{equation}
\int_{M}  m_2^{T}(\phi_{12}, \phi_{01})\wedge \frac{\ast \phi_{02}}{\|\phi_{02}\|^2}   = \int_{M}  \phi_{12} \wedge \phi_{01} \wedge \frac{\ast \phi_{02}}{\|\phi_{02}\|^2},
\end{equation}
can be controlled by $e^{-\lam (\dist_{T}(x_{v_r}) - A)}$ in the following sense. 

Fixing $x_{v_r}\in M$ and sufficiently small $\epsilon>0$, we apply cutoff function $\chi_r$ supported in $B(x_{v_r},r_1)$ and obtain
\begin{equation*}
\|\chi_r \wedge \phi_{12} \wedge \phi_{01} \wedge \frac{\ast \phi_{02}}{\|\phi_{02}\|^2}\|_{L^{\infty}(M)}
\leq  C_{\epsilon} e^{-\lam (\vec{\dist}_{T}(x_{v_r})-A - 3r_1 - 3\epsilon)}
\end{equation*}
Here the decay factors $\psi_{q_{01}}(x_v) = \dist_{01}(x_{v_r},q_{01}) + f_{01}(q_{01})$, $\psi_{q_{12}}(x_v) = \dist_{12}(x_{v_r},q_{12}) + f_{12}(q_{12})$ and $\psi_{q_{02}}(x_v)-2f_{02}(q_{02}) = \dist_{02}(x_{v_r},q_{02}) - f_{02}(q_{02})$ come from the a priori estimate in lemma \ref{eigenestimate2} for the input forms $\phi_{01}$,  $\phi_{12}$ and $\frac{\ast \phi_{02}}{\|\phi_{02}\|^2}$ respectively. 

We assume there are gradient trees $\Gamma_1 ,\dots, \Gamma_l$ of the type $T$. For each tree $\Gamma_i$, we take open neighborhoods $D_{\Gamma_i,v_r}$ and $W_{\Gamma_i,v_r}$ of interiors vertices $x_{\Gamma_i,v_r}$ with $\overline{D_{\Gamma_i,v_r}} \subset W_{\Gamma_i,v_r} $ as shown in following Figure \ref{figure:m_2cutoff}. 

\begin{figure}[h]
	\centering
	\includegraphics[scale=0.3]{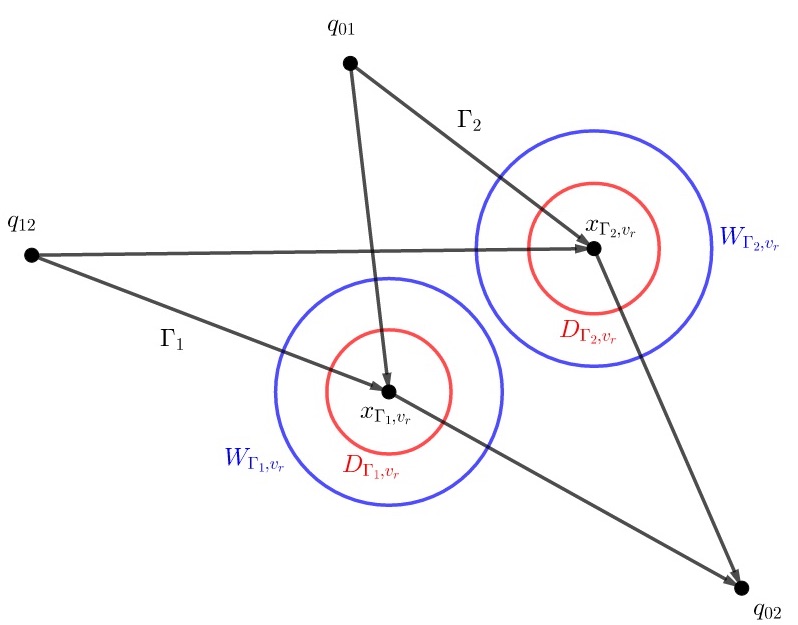}
	\caption{Cut off of integral near gradient trees of type $T$}
	\label{figure:m_2cutoff}
\end{figure}

Since $\vec{\dist}_{T}(x_{v_r})$ is a continuous function in $x_{v_r}$ attending minimum value $A$ exactly at internal vertices $x_{\Gamma_i,v_r}$ of gradient trees $\Gamma_i$'s, we have a constant $C>0$, depends on the size of the neighborhood $D_{\Gamma_i,v_r}$'s, such that $\vec{\dist}_{T} \geq A +C $ in $M  \setminus \cup_i D_{\Gamma_i,v_r}$ by continuity from the discussion above equation \eqref{eqn:sym_area_constant}.

If $B(x_{v_r},r_1)$ is away from the $D_{\Gamma_i,v_r}$'s, we have 
\begin{equation*}
\|\chi_r \wedge \phi_{12} \wedge \phi_{01} \wedge \frac{\ast \phi_{02}}{\|\phi_{02}\|^2}\|_{L^{\infty}(M)}
\leq  C_{\epsilon} e^{-\lam (\frac{C}{2})},
\end{equation*}
and thus contributes exponentially small error terms.

To obtain the leading order term contribution, we take cutoff functions $\chi_{\Gamma_i,v}$,  $\chi_{\Gamma_i,v_r}$ associating to each tree $\Gamma_i$, with supports in $W_{\Gamma_i, v_r}$ and equal to $1$ on $\overline{D_{\Gamma_i, v_r}}$, and get

\begin{align*}
	& \int_{M}  m_2^{T}(\phi_{12}, \phi_{01})\wedge \frac{\ast \phi_{02}}{\|\phi_{02}\|^2} \\
	 = & \sum_i \int_{M}\lbrace \chi_{\Gamma_i,v_r}  \phi_{12} \wedge \phi_{01} \wedge \frac{\ast \phi_{02}}{\|\phi_{02}\|^2}\rbrace + \mathcal{O}(e^{-\lam(\frac{C}{2})}).
\end{align*}
This localizes the integral computing $m_2^{T}$ to gradient trees $\Gamma_i$'s of type $T$. Notice that the neighborhoods $D_{\Gamma_i}$ and $W_{\Gamma_i}$ can be chosen to be arbitrarily small. 

\subsubsection{WKB methods for $m_2$}\label{sec:m_2_WKB}

In this section, we introduce the WKB method which allows us to compute the leading order contribution in $m_2^{T}$ explicitly. We fix a gradient tree $\Gamma$ as in the section \ref{m3estimate}, with interior vertices $x_{v_r}:=x_{\Gamma,v_r}$ (since the gradient tree $\Gamma$ is fixed, we omit the dependence on $\Gamma$ in our notations). We take neighborhoods $W_{v_r}$ of $x_{v_r}$, with cutoff functions $\chi_{v_r}$ supported in $W_{v_r}$ as in section \ref{sec:apriori_m2_case}.

As $x_{v_r} \in  V_{q_{12}}^+ \cap V_{q_{01}}^+ \cap V_{q_{02}}^-$, we can assume that the WKB approximations from lemma \ref{eigenwkb}
\begin{equation*}
\phi_{ij} \sim \lam^{\frac{\deg(q_{ij})}{2}} e^{-\lam \psi_{ij}}( \omega_{ij,0}+\omega_{ij,1}\lam^{-1/2}+\dots),
\end{equation*}
hold in $W_{v_r}$ for $ij = 01, 12, 02$ (by lemma \ref{eigenwkb}, for any $ij = 01, 12, 02$, we have $\omega_{ij,k} = 0$ when $k$ is odd, but we still insist to write the expansions in the above form to unify our notations in the rest of the proof), by taking a smaller $W_{v_r}$ if necessary while using the lemma \ref{eigenestimate2}.

Computing the integral by using the WKB expansions, we have
\begin{multline}\label{eqn:wkbm2}
\int_{M}\lbrace \chi_{v_r}  \phi_{12} \wedge \phi_{01} \wedge \frac{\ast \phi_{02}}{\|\phi_{02}\|^2}\rbrace\\
= \lam^{\frac{deg(q_{12}) +\deg(q_{01})-\deg(q_{02})-1}{2}}\int_{M}  \lbrace\chi_{v_r} (e^{-\lam\psi_{12}}\omega_{12,0}) \wedge  (e^{-\lam\psi_{01}}\omega_{01,0}) \wedge \frac{e^{-\lam \psi_{02}} \ast \omega_{02,0}}{\|\phi_{02}\|^2})\rbrace \\
= \frac{1}{\| \phi_{02} \|^2}\int_M \lbrace \chi_{v_r} (e^{-\lam (\psi_{12}+\psi_{01}+\psi_{02})}\omega_{12,0}\wedge \omega_{01,0}\wedge (\ast \omega_{02,0}))
\end{multline}
modulo terms of order $\mathcal{O}(\lam^{-1})$. We observe that the exponential decay factor of the integrand is $e^{-\lam(\psi_{12}+\psi_{01}+\psi_{02})} = e^{-\lam(g^{+}_{12}+g_{01}^+ + g_{02}^-)}$, where $g^\pm_{ij}$ are introduced in remark \ref{eigenwkbremark}. 

Recall that $g^+_{01}$, $g^+_{12}$ and $g_{02}^-$ are Bott-Morse with absolute minimums on $V^+_{01}$, $V_{12}^+$ and $V_{02}^-$ respectively. The generic assumption (definition \ref{genericassumption}) of the sequence $\vec{f}$ indicates that $V_{12}^+$, $V^+_{01}$ and $V^-_{02}$ intersect transversally at $x_{v_r}$ which means $e^{-\lam (g^{+}_{12}+g_{01}^+ + g_{02}^-)}$ concentrates at $x_{v_r}$. The leading order contribution will be computed in the up coming section. 

\subsubsection{Explicit computations for $m_2$}\label{sec:m2_computation}

We will need the following lemma which will be proven in section \ref{leadingrelation}.

\begin{lemma}\label{stat_phase_exp_NB_1}
	Let $M$ be a $n$-dimensional manifold and $S$ be a $k$-dimensional submanifold in $M$, with a neighborhood $B$ of $S$ which can be identified as the normal bundle $\pi: NS \rightarrow S$. Suppose $\varphi: B \rightarrow\real_{\geq0}$ is a Bott-Morse function with zero set $S$ and $\beta \in\Omega^*(B)$ has a vertically compact support along the fiber of $\pi$, we have\[
	\pi_*(e^{-\lam \varphi(x)}\beta) = (\frac{\lam}{2\pi})^{(n-k)/2} (\iota_{\vol(\Hess\varphi)}\beta)|_V (1+\order(\lam^{-1})),\]
	where $\pi_*$ is the integration along fiber and $\vol(\Hess \varphi)$ is the volume polyvector field defined for the positive symmetric tensor $\Hess \varphi$ along fibers of $\pi$.
\end{lemma}

From lemma \ref{stat_phase_exp_NB_1}, we know that the leading order contribution in the above integral \eqref{eqn:wkbm2} depends only on values of $\omega_{12,0}$, $\omega_{01,0}$ and $\ast \omega_{02,0}$ at the point $x_{v_r}$. We use the normal bundle $NV_{12}^+ \oplus NV_{01}^+ \oplus NV_{02}^-$ at $x_{v_r}$ to parametrize a neighborhood of $x_{v_r}$. Making use of lemma \ref{stat_phase_exp_NB_1}, we can split the integral as follows for computing leading order contribution. We have
\begin{eqnarray*}
	&& \int_M  \chi_{v_r} e^{-\lam (g^{+}_{12}+g^+_{01}+g^-_{02})}\omega_{12,0}\wedge \omega_{01,0}\wedge (\ast \omega_{02,0}) \\
	&=&\pm (\int_{NV^+_{12,x_{v_r}}} e^{-\lam g^{+}_{12}} \chi_{v_r}\omega_{12,0}) (\int_{NV^+_{01,x_{v_r}}} e^{-\lam g_{01}^+} \chi_{v_r}\omega_{01,0})\\
	&&(\int_{NV_{02,x_{v_r}}^-} e^{-\lam g^-_{02}} \chi_{v_r}(\ast \omega_{02,0}))( 1+\mathcal{O}(\lam^{-1})),
\end{eqnarray*}

\noindent where the sign depends on whether the orientations of $ NV_{12}^+ \oplus NV_{01}^+ \oplus NV_{02}^-$ and $TM$ match or not at the point $x_{v_r}$. From lemma \ref{eigenwkbcal}, we obtain equality 
$$\lam^{\frac{\deg(q_{ij})}{2}}\int_{NV^+_{ij,x_{v_r}}} e^{-\lam g_{ij}^+}  \chi_{v_r}\omega_{ij,0} = 1+\mathcal{O}(\lam^{-1}),$$
for $ij = 01,12$ and 
$$
\frac{\lam^{\frac{\deg(q_{02})}{2}}}{\| \phi_{02} \|^2}(\int_{NV_{02,x_{v_r}}^-} e^{-\lam g^-_{02}} \chi_{v_r}(\ast \omega_{02,0}))= 1+\mathcal{O}(\lam^{-1})
$$
from the lemma \ref{eigenwkbcal}. Therefore we conclude that
\begin{equation*}
\int_{M}\lbrace \chi_{v_r}  \phi_{12} \wedge \phi_{01} \wedge \frac{\ast \phi_{02}}{\|\phi_{02}\|^2}\rbrace = \pm (1+\mathcal{O}(\lam^{-1})),
\end{equation*} 
where the sign depends on matching the orientations of $ NV_{12}^+ \oplus NV_{01}^+ \oplus NV_{02}^-$ and $TM$ at the point $x_{v_r}$. 

\begin{remark}
	Notice that we have a stronger estimate with the error term being $\mathcal{O}(\lam^{-1})$ instead of $\mathcal{O}(\lam^{-1/2})$ since the estimate of the homotopy operator (see lemma \ref{homotopywkb}) is not involved in the $m_2$ case.
\end{remark}

%--------------------------------------------A priori estimate-----------------------------------------
\subsection{Proof of $m_3$}\label{sec:m_3_case}

Next we consider the $m_3(\lam)$ case to illustrate the analytic argument needed for handling the homotopy operator $H_{ij}$. 

\subsubsection{Apriori estimate for $m_3(\lam)$ case}\label{aprioriestimate}\label{m3estimate}

There are two $3$-leafed directed trees, which are denoted by $T_1$ and $T_2$. We simply consider $m_3^{T_1}(\lam)$ where $T_1$ is the tree shown in figure \ref{3trees} and relate this operation to counting gradient trees of type $T_1$. $T_1$ has two interior vertices $v$ and $v_r$. According to the combinatorics of $T_1$, we define $\vec{\dist}_{T_1} : M^{| V(T_1)|} \rightarrow \real_+$ by
\begin{eqnarray*}
&&\vec{\dist}_{T_1}(x_v,x_{v_r}) \\
 &=&\dist_{13}(x_v,x_{v_r})
+\dist_{01}(x_{v_r},q_{01})+\dist_{12}(x_v,q_{12})+\dist_{23}(x_v,q_{23}) +\dist_{03}(x_{v_r},q_{03}).
\end{eqnarray*}
It is the length of the geodesic tree of type $T_1$ with interior vertices $x_v,x_{v_r}$ and endpoints of semi-infinite edges $e_{ij}$'s laying on $q_{ij}$'s as shown in the following figure.

\begin{figure}[h]
\centering
\includegraphics[scale=0.2]{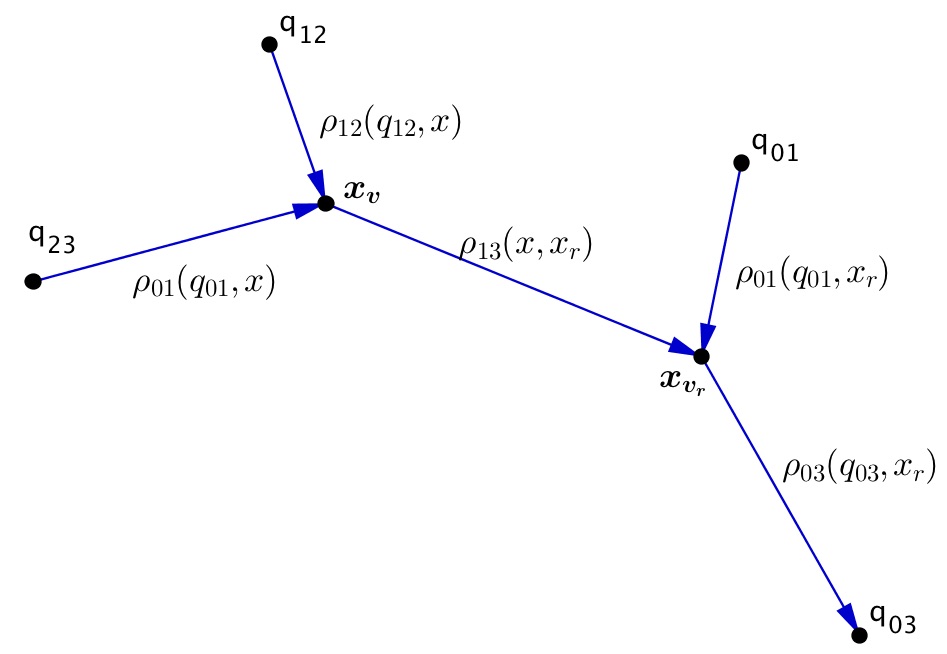}
\end{figure}

Similar to the proof of $m_2(\lam)$ case in section \ref{sec:m_2_case}, we notice that $\vec{\dist}_{T_1}(x_v,x_{v_r}) \geq A$ where
\begin{equation}\label{eqn:sym_area_constant}
A:= f_{03}(q_{03}) - f_{01}(q_{01}) - f_{12}(q_{12}) - f_{23} (q_{23}),
\end{equation} and the equality holds if and only if $(x_v,x_{v_r})$ are interior vertices of a gradient flow tree of the type $T_1$. Once again we only have gradient flow trees instead of generalized gradient trees since we assume the sequence of Morse function $\vec{f}$ satisfyies the generic assumption (see definition \ref{genericassumption}).\\

%Equivalently, we have the term 
%\begin{eqnarray*}
%&&\vec{g}_{T_1}(x_v,x_{v_r}) \\
%&=& \dist_{13}(x_v,x_{v_r}) - (f_{13}(x_{v_r}) - f_{13}(x_v))\\
%&+&g_{q_{01}}^+(x_{v_r})+g_{q_{12}}^+(x_v)+g_{q_{23}}^+(x_v)\\
%\end{eqnarray*}
%being non-negative and attending zero value exactly when $(x_v,x_{v_r})$ being interior vertices of a gradient flow tree of the type $T_1$.

We apply lemma \ref{eigenestimate2} and lemma \ref{resolventlemma} to conclude the integrand of
\begin{equation}
\int_{M}  m_3^{T_1}(\phi_{23}, \phi_{12}, \phi_{01})\wedge \frac{\ast \phi_{03}}{\|\phi_{03}\|^2}   = \int_{M}  H_{13}(\phi_{23}\wedge \phi_{12}) \wedge \phi_{01} \wedge \frac{\ast \phi_{03}}{\|\phi_{03}\|^2},
\end{equation}
is controlled by $e^{-\lam (\dist_{T_1} - A)}$ as follows. 

Fixing two points $x_v,x_{v_r} \in M$ and sufficiently small $\epsilon>0$ such that estimate for $G_{13}$ as well as $H_{13}$ in lemma \ref{resolventlemma} holds for some balls $U=B(x_v, r_1)$ and $V=B(x_{v_r}, r_1)$ (with respect to $\dist_{13}$). %lemma \ref{resolventlemma} holds for operator $G_{13}$ and hence $H_{13}$ with $U$ and $V$ being balls centering at $x_v$ and $x_{v_r}$ (with respect to $\dist_{13}$) of radius $r_1$.
If $\chi$ and $\chi_r$ are cutoff functions supported in $B(x_v,r_1)$ and $B(x_{v_r},r_1)$ respectively, then we have
\begin{equation*}
\| \chi_r H_{13}(\chi \phi_{23}\wedge \phi_{12})\|_{L^{\infty}} \leq C_{\epsilon} e^{-\lam (\psi_{q_{23}}(x_v) + \psi_{q_{12}}(x_v) + \rho_{13}(x_v,x_{v_r})-2r_1-3\epsilon)}
\end{equation*}
for those large enough $\lam$, where lemma \ref{eigenestimate2} gives the decay factors $\psi_{q_{23}}(x_v)$ and $\psi_{q_{12}}(x_v)$ of the input forms $\phi_{23}$ and $\phi_{12}$ respectively, and lemma \ref{resolventlemma} gives the decay factor $\rho_{13}(x_v,x_{v_r})$. Combining with the decay estimates for $\phi_{01}$ and $\frac{\ast \phi_{03}}{\|\phi_{03}\|^2}$ as in section \ref{sec:m_2_case}, we obtain
\begin{equation*}
\|\chi_r H_{13}(\chi \phi_{23}\wedge \phi_{12}) \wedge \phi_{01} \wedge \frac{\ast \phi_{03}}{\|\phi_{03}\|^2}\|_{L^{\infty}(M)}
\leq  C_{\epsilon} e^{-\lam (\vec{\dist}_{T_1}(x_v,x_{v_r})-A - 4r_1 - 5 \epsilon)}
\end{equation*}
where $x_v,x_{v_r}$ are the centers of balls chosen for taking the cutoff functions $\chi,\chi_r$ as above and $A$ is defined in equation \eqref{eqn:sym_area_constant}.

Once again we assume there are gradient trees $\Gamma_1 ,\dots \Gamma_l$ of the type $T_1$. For each tree $\Gamma_i$, we take open neighborhoods $D_{\Gamma_i,v}$ and $W_{\Gamma_i,v}$ of interiors vertices $x_{\Gamma_i,v}$ with $\overline{D_{\Gamma_i,v}} \subset W_{\Gamma_i,v} $, and similarly $D_{\Gamma_i , v_r}$ and $W_{\Gamma_i,v_r}$ for $x_{\Gamma_i,v_r}$, as illustrated in figure \ref{figure:cutoff}.

\begin{figure}[h]
\centering
\includegraphics[scale=0.22]{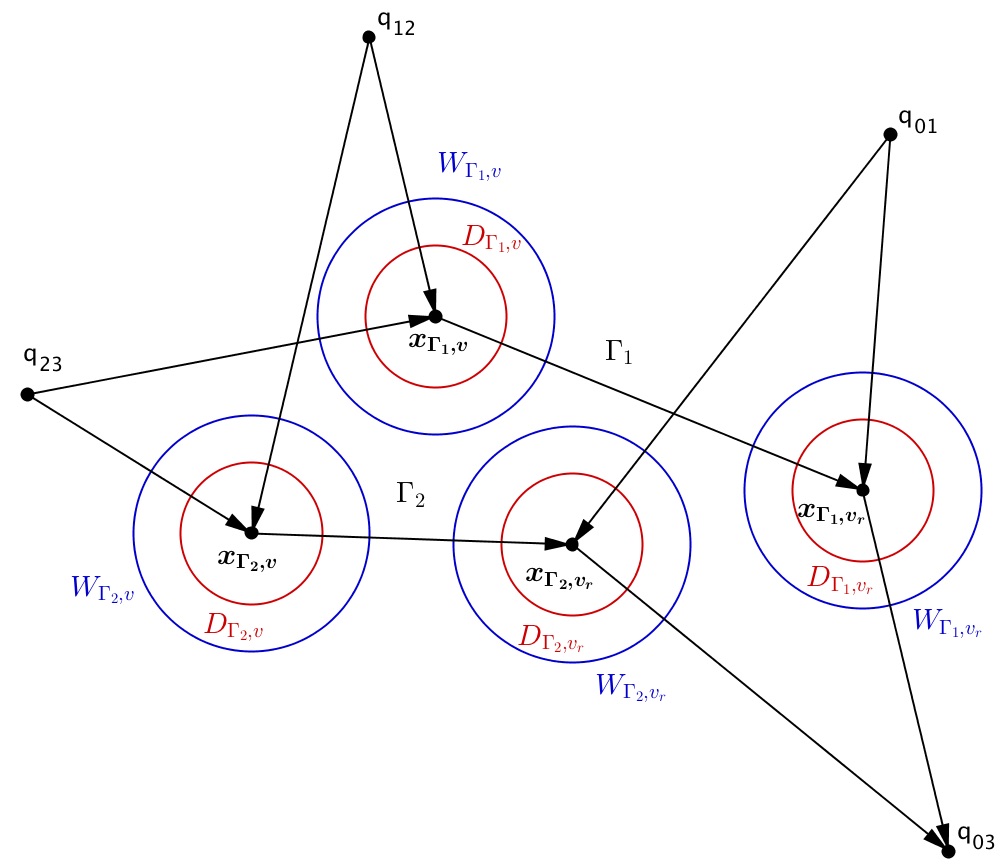}
\caption{Cutoff of integral near gradient trees of type $T_1$}
\label{figure:cutoff}
\end{figure}

Since $\vec{\dist}_{T_1}$ is a continuous function and it attends minimum value $A$ exactly when $(x_v,x_{v_r})=(x_{\Gamma_i,v},x_{\Gamma_i,v_r})$ for some gradient tree $\Gamma_i$, there is a constant $C>0$ (again depending on the size of the neighborhood $D_{\Gamma_i}$'s) such that $\vec{\dist}_{T_1} \geq A +C $ in $M^{|V(T_1)| } \setminus \cup_i D_{\Gamma_i}$ by continuity from the discussion at the beginning of section \ref{m3estimate}, where $D_{\Gamma_i } = D_{\Gamma_i,v} \times D_{\Gamma_i,v_r}$.

If $\vec{B}(\vec{x},r_1) = B(x_v,r_1) \times B(x_{v_r},r_1)$ is away from the $D_{\Gamma_i}$'s, we have 
\begin{equation*}
\|\chi_r H_{13}(\chi \phi_{23}\wedge \phi_{12}) \wedge \phi_{01} \wedge \frac{\ast \phi_{03}}{\|\phi_{03}\|^2}\|_{L^{\infty}(M)}
\leq  C_{\epsilon} e^{-\lam (\frac{C}{2})}.
\end{equation*}

Therefore we can take cutoff functions $\chi_{\Gamma_i,v}$,  $\chi_{\Gamma_i,v_r}$ associating to each tree $\Gamma_i$, with supports in $W_{\Gamma_i, v}$, $W_{\Gamma_i, v_r}$ and equal to $1$ on $\overline{D_{\Gamma_i, v}}$, $\overline{D_{\Gamma_i, v_r}}$ respectively, and obtain

\begin{eqnarray*}
&& \int_{M}  m_3^{T_1}(\phi_{23}, \phi_{12}, \phi_{01})\wedge \frac{\ast \phi_{03}}{\|\phi_{03}\|^2}\\
& = & \sum_i \int_{M}\lbrace \chi_{\Gamma_i,v_r} H_{13}(\chi_{\Gamma_i,v} \phi_{23}\wedge \phi_{12}) \wedge \phi_{01} \wedge \frac{\ast \phi_{03}}{\|\phi_{03}\|^2}\rbrace + \mathcal{O}(e^{-\lam(\frac{C}{2})}).\\
\end{eqnarray*}
This localizes the integral computing $m_3^{T_1}$ to gradient trees of type $T_1$ where the neighborhoods $D_{\Gamma_i}$ and $W_{\Gamma_i}$ can be chosen to be arbitrarily small. 

%--------------------------------------------WKB approximation--------------------------------------
\subsubsection{WKB method for $m_3$}

Similar to the previous section \ref{sec:m_2_WKB}, we only focus on a gradient tree $\Gamma$ of type $T_1$ as in the section \ref{m3estimate}, with interior vertices $x_{\Gamma,v}$ and $x_{\Gamma,v_r}$. Once again, we omit the dependence on $\Gamma$ to simplify our notations. We take neighborhoods $W_{v}$ and $W_{v_r}$ of $x_{v}$ and $x_{v_r}$ respectively, and $\chi_{v}$ and $\chi_{v_r}$ are cutoff functions supported in $W_{v}$ and $W_{v_r}$ respectively as shown in the following figure.

\begin{figure}[h]
\centering
\includegraphics[scale=0.2]{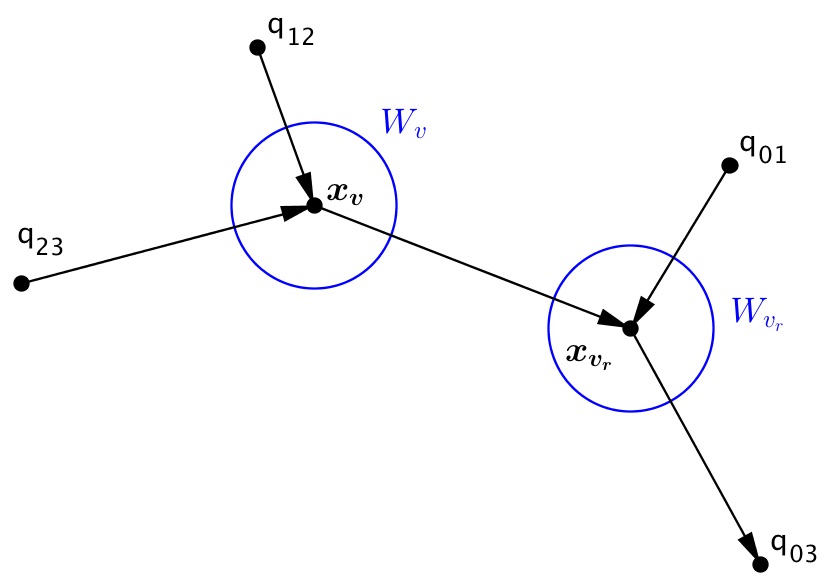}
\end{figure}

As $x_{v} \in V_{q_{12}}^+ \cap V_{q_{23}}^+$, we can assume that the WKB approximations from lemma \ref{eigenwkb}
\begin{equation*}
\phi_{12} \sim \lam^{\frac{\deg(q_{12})}{2}} e^{-\lam \psi_{12}}( \omega_{12,0}+\omega_{12,1}\lam^{-1/2}+\dots),
\end{equation*}
and
\begin{equation*}
\phi_{23} \sim \lam^{\frac{\deg(q_{23})}{2}} e^{-\lam\psi_{23}}( \omega_{23,0}+\omega_{23,1}\lam^{-1/2}+\dots)
\end{equation*}
hold in $W_{v}$ (indeed $\omega_{12,k} = 0$ and $\omega_{23,k} = 0$ when $k$ is odd), by taking a smaller $W_{v}$ if necessary while using the lemma \ref{eigenestimate2}. Then, we need a similar WKB approximation for the term
$$
H_{13} (\chi_{v}\phi_{23} \wedge \phi_{12}),
$$
in the neighborhood $W_{v_r}$. Here we state a WKB lemma for the homotopy operators $H_{ij}$ which appear in the higher products $\deprod{k}$ for $k\geq 3$. The proof will occupy the whole section \ref{approximation}.\\

\paragraph{\textit{WKB for homotopy operator}}\label{wkbhomotopy}

%We give the setup of the lemma.
Let $\gamma(t)$ be a flow line of $\nabla f_{ij}/ | \nabla f_{ij}|_{\rho_{ij}}$ starts at $\gamma(0)=x_{S}$ and ends at $\gamma(T) = x_{E}$ for a fixed $T>0$. We consider an input form $\zeta_S$ defined in a neighborhood $W_S$ of $x_{S}$. Suppose we are given a WKB approximation of $\zeta_S$ in $W_S$, which is an approximation of $\zeta_S$ according to order of $\lam$ of the form
\begin{equation}
\zeta_S \sim e^{-\lam \psi_{\scalebox{.7}{$\scriptscriptstyle S$}}}(\omega_{S,0}+\omega_{S,1} \lam^{-1/2}+ \omega_{S,2} \lam^{-1} + \dots)
\end{equation}
(The precise meaning of this infinite series approximation can be found in section \ref{L2approx}). We further assume that $g_S  = \psi_S - f_{ij}$ is a nonnegative Bott-Morse function in $W_S$ with zero set $V_S$. We consider the equation
\begin{equation}\label{homotopyeq}
\Delta_{ij} \zeta_E = (I-P_{ij})d_{ij}^*(\chi_S \zeta_S),
\end{equation}
where $\chi_S$ is a cutoff function compactly supported in $W_S$, $P_{ij}:\Omega_{ij}^*(M,\lam)\rightarrow\Omega_{ij}^*(M,\lam)_{sm}$ is the projection. We want to have a WKB approximation of the solution $\zeta_E = H_{ij}(\chi_S \zeta_S)$ to the equation \eqref{homotopyeq}.
\begin{lemma}[=Theorem \ref{thm:homotopy_wkb}]\label{homotopywkb}
If $\supp(\chi_S)$ is small enough, there is a WKB approximation of $\zeta_E$ in a small enough neighborhood $W_E$ of $x_E$, of the form
\begin{equation}
\zeta_E \sim e^{-\lam \psi_{\scalebox{.7}{$\scriptscriptstyle E$}}}\lam^{-1/2}(\omega_{E,0}+\omega_{E,1}\lam^{-1/2}+\dots).
\end{equation}
Furthermore, $g_E:=\psi_E-f_{ij}$ is a nonnegative  Bott-Morse function in $W_E$ with zero set $V_E =( \bigcup_{-\infty<t<+\infty} \sigma_t(V_S)) \cap W_E$ which is closed in $W_E$, where $\sigma_t$ is the time $t$ flow of $\nabla f_{ij}/ | \nabla f_{ij}|^2$ (normalized according to $|df_{ij}|^2\langle\cdot,\cdot\rangle$).

%Furthermore, the function $g_E:=\psi_E-f_{ij}$ is a nonnegative function which is Bott-Morse in $W_E$ with zero set $V_E =( \bigcup_{-\infty<t<+\infty} \sigma_t(V_S)) \cap W_E$ which is closed in $W_E$, where $\sigma_t$ is the time $t$ flow of $\nabla f_{ij}/ | \nabla f_{ij}|^2$ (normalized according to $|df_{ij}|^2\langle\cdot,\cdot\rangle$).
\end{lemma}

%------------------------------------------------------M3 case-------------------------------------
\paragraph{\textit{WKB for} $m_3$ (\textit{cont'd})} 

We apply lemma \ref{homotopywkb} with Morse function $f_{13}$, input form $\zeta_S = \phi_{23}\wedge \phi_{12}$, starting vertex $x_S = x_{v}$, ending vertex $x_E = x_{v_r}$, with neighborhood $W_S = W_{v}$ and $W_E = W_{v_r}$ (This can be done by shrinking $W_{v}$ and $W_{v_r}$ if necessary). As a result, we obtain the WKB approximation 
\begin{equation*}
H_{13} (\chi_{v}\phi_{23} \wedge \phi_{12}) \sim \lam^{\frac{\deg(q_{23})+\deg(q_{12})-1}{2}} e^{-\lam\psi_{13}}(\omega_{13,0}+\omega_{13,1}\lam^{-1/2}+\dots),
\end{equation*}
by taking $\psi_{E} = \psi_{13}$ and $\omega_{E,i} = \omega_{13,i}$ in the lemma.\\

In order to compute
$$
\int_{M}  m_3^{T_1}(\lam,\vec{\chi}_{\Gamma}) \wedge  \frac{\ast \phi_{03}}{\|\phi_{03}\|^2} = \int_{M}  \chi_{v_r}H_{13} (\chi_{v}\phi_{23} \wedge \phi_{12}) \wedge \phi_{01} \wedge \frac{\ast \phi_{03}}{\|\phi_{03}\|^2}
$$
up to an error of order $\mathcal{O}(\lam^{-1/2})$, we can simply compute the integral
\begin{multline}\label{wkbm3}
\lam^{\frac{\deg(q_{23})+\deg(q_{12}) +\deg(q_{01})-1}{2}}\int_{M}  \lbrace\chi_{v_r} (e^{-\lam\psi_{13}}\omega_{13,0}) \wedge  (e^{-\lam\psi_{01}}\omega_{01,0}) \wedge \\
\wedge (\lam^{-\frac{\deg(q_{03})}{2}}\frac{e^{-\lam \psi_{03}} (\ast \omega_{03,0})}{\|\phi_{03}\|^2})\rbrace\\
 = \frac{1}{\| \phi_{03} \|^2}\int_M \lbrace \chi_{v_r} (e^{-\lam (\psi_{13}+\psi_{01}+\psi_{03})}\omega_{13,0}\wedge \omega_{01,0}\wedge (\ast \omega_{03,0})).
\end{multline}
We study the exponential decay factor $e^{-\lam(\psi_{13}+\psi_{01}+\psi_{03})}$ of the integrand by defining $g^{}_{13}:=\psi_{13}-f_{13}$. Then, the exponential decay of the integrand can be expressed as
\begin{equation*}
e^{-\lam (g^{}_{13}+g_{01}^+ + g_{03}^-)}.
\end{equation*}

Once again remark \ref{eigenwkbremark} tells us that $g^+_{01}$, $g^+_{12}$, $g_{23}^+$ and $g_{03}^-$ are Bott-Morse with absolute minimums on $V^+_{01}$, $V_{12}^+$, $V_{23}^+$ and $V_{03}^-$ respectively. We also recall from lemma \ref{homotopywkb} that $g^{}_{13}$ is a Bott-Morse function in $W_{v_r}$ with absolute minimum denoted by $V_{13}$ (colored red in the following figure), which is the submanifold $(\bigcup_{-\infty<t<+\infty} \sigma_t(V_{23}^+ \cap V_{12}^+)) \cap W_{v_r}$ flowed out from $V_{23}^+ \cap V_{12}^+$ (colored blue in the following figure), under the flow of $\frac{\nabla f_{13}}{| \nabla f_{13}|^2}$ which is denoted by $\sigma_t$. \\
\begin{figure}[h]
\centering
\includegraphics[scale =0.2]{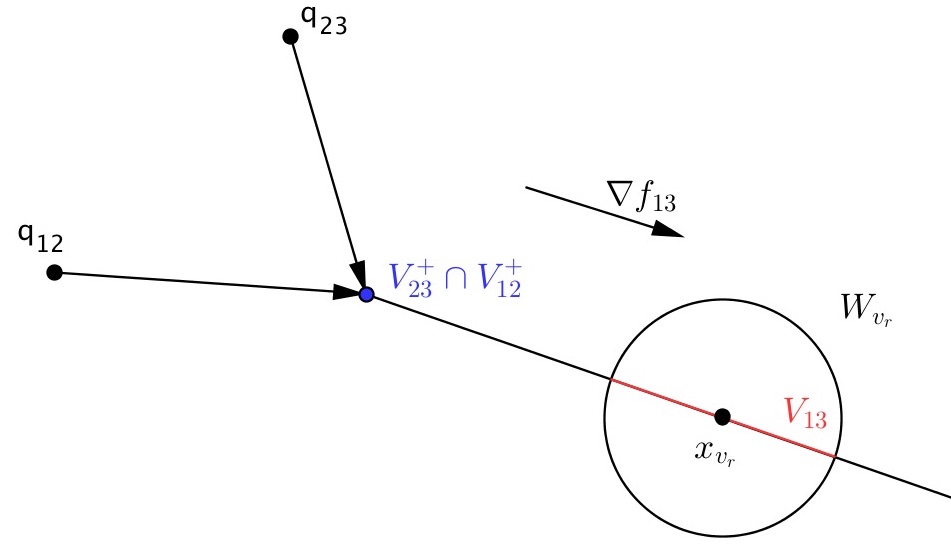}
\end{figure}

The generic assumption of the sequence $\vec{f}$ indicates that $V_{13}$, $V^+_{01}$ and $V^-_{03}$ intersect transversally at $x_{v_r}$ which means $e^{-\lam (g^{}_{13}+g^+_{01}+ g_{03}^-)}$ concentrates at $x_{v_r}$ and hence the leading order contribution will only depend on the value of $\omega_{13,0}\wedge \omega_{01,0} \wedge \ast \omega_{03,0}$ at the point $x_{v_r}$.

%--------------------------------------------Explicit computation-----------------------------------
\subsubsection{Explicit computations for $m_3$}\label{computation}
From lemma \ref{stat_phase_exp_NB_1}, we know that the leading order contribution of the integral \eqref{wkbm3} depends only on values of $\omega_{13,0}$, $\omega_{01,0}$ and $\ast \omega_{03,0}$ at the point $x_{v_r}$ and the integral can be splitted as
%We find from the above Lemma \ref{stat_phase_exp_NB_1} that the leading order contribution in the above integral \eqref{wkbm3} depend only on values of $\omega_{13,0}$, $\omega_{01,0}$ and $\ast \omega_{03,0}$ at the point $x_{v_r}$. Using Lemma \ref{stat_phase_exp_NB_1} as in the previous Section \ref{sec:m2_computation}, we can split the integral as 
\begin{eqnarray*}
&& \int_M  \chi_{v_r} e^{-\lam (g^{}_{13}+g^+_{01}+g^-_{03})}\omega_{13,0}\wedge \omega_{01,0}\wedge (\ast \omega_{03,0}) \\
&=&\pm (\int_{NV_{13,x_{v_r}}} e^{-\lam g^{}_{13}} \chi_{v_r}\omega_{13,0}) (\int_{NV^+_{23,x_{v_r}}} e^{-\lam g_{23}^+} \chi_{v_r}\omega_{23,0})\\
&&(\int_{NV_{03,x_{v_r}}^-} e^{-\lam g^-_{03}} \chi_{v_r}(\ast \omega_{03,0}))( 1+\mathcal{O}(\lam^{-1})),
\end{eqnarray*}

\noindent where the sign depends on whether the orientations of $ NV_{13} \oplus NV_{01}^+ \oplus NV_{03}^-$ and $TM$ match or not at the point $x_{v_r}$. We will compute the above integrals one by one. We obtain equality 
$$\lam^{\frac{\deg(q_{01})}{2}}\int_{NV^+_{01,x_{v_r}}} e^{-\lam g_{01}^+}  \chi_{v_r}\omega_{01,0} = 1+\mathcal{O}(\lam^{-1}),$$
and 
$$
\frac{\lam^{\frac{\deg(q_{03})}{2}}}{\| \phi_{03} \|^2}(\int_{NV_{03,x_{v_r}}^-} e^{-\lam g^-_{03}} \chi_{v_r}(\ast \omega_{03,0}))= 1+\mathcal{O}(\lam^{-1})
$$
from the lemma \ref{eigenwkbcal}. Moreover, we have
$$
\lam^{\frac{\deg(q_{23})+\deg(q_{12}) -1}{2}}\int_{NV_{13,x_{v_r}}} e^{\lam g^{}_{13}} \chi_{v_r}\omega_{13,0} = (1+\mathcal{O}(\lam^{-1})).
$$
This depends on the fact that
\begin{eqnarray*}
&&\lam^{\frac{\deg(q_{23})+\deg(q_{12})}{2}}\int_{N(V_{23}^+ \cap V_{12}^+)_{x_{v}}} e^{-\lam(g_{23}^+ + g_{12}^+) }\chi_{v} \omega_{23,0} \wedge \omega_{12,0} \\
&=&(\lam^{\frac{\deg(q_{23})}{2}}\int_{N(V_{23}^+)_{x_{v}}} e^{-\lam g_{23}^+} \chi_{v}\omega_{23,0}) (\lam^{\frac{\deg(q_{12})}{2}}\int_{N(V_{12}^+)_{x_{v}}} e^{-\lam g_{12}^+}\chi_{v} \omega_{12,0}) (1+\mathcal{O}(\lam^{-1}))\\
&=& 1+\mathcal{O}(\lam^{-1}),\\
\end{eqnarray*}
and the following lemma.
\begin{lemma}[=Lemma \ref{lem:explicit_normal_computation}]\label{homotopywkbcal} Using same notations in lemma \ref{homotopywkb} and suppose $\chi_S$ and $\chi_E$ are cutoff functions supported in $W_S$ and $W_E$ respectively, then we have
\begin{equation}
\lam^{-1/2}\int_{N(V_E)_{v_E}} e^{-\lam g^{}_{\scalebox{.7}{$\scriptscriptstyle E$}}} \chi_{E} \omega_{E,0} = (\int_{N(V_S)_{v_S}} e^{-\lam g^{}_{\scalebox{.7}{$\scriptscriptstyle S$}}}\chi_{S} \omega_{S,0})(1+\mathcal{O}(\lam^{-1})).
\end{equation}
Furthermore, suppose $\omega_{S,0}(x_S)\in \bigwedge^{top} N(V_S)^*_{x_S}$, we have $\omega_{E,0}(x_E) \in  \bigwedge^{top} N(V_E)^*_{x_E}$. Here $\bigwedge^{top} E$ refers to $\bigwedge^r E$ for a rank $r$ vector bundle $E$.
\end{lemma}

Putting the above together, we get the following 
\begin{equation}
\int_M m_3^{T_1}(\lam,\vec{\chi}_\Gamma) \wedge \frac{\ast \phi_{03}}{\|\phi_{03}\|^2}= \pm (1+\mathcal{O}(\lam^{-1/2})),
\end{equation} 
where the sign depends on matching the orientations of $ NV_{13} \oplus NV_{01}^+ \oplus NV_{03}^-$ and $TM$ at the point $x_{v_r}$. The proof for $m_3(\lam)$ is completed and we move on to the $m_k(\lam)$ case for any $k\geq 3$. The proof is essentially the same as the $m_3(\lam)$ case except involving more combinatorics and notations. \\

%----------------------------------------------------------------------------------------------------

% !TEX root = paper.tex
%------------------------------------------------Apriori estimate for mk--------------------
\subsection{Proof of $m_k$}

\subsubsection{A priori estimates for $m_k$}\label{aprioriestimate_mk}
We fix a $k$-leafed tree $T$ and denote the corresponding operation by $m_k^{T}(\lam)$. We try to relate $m_k^{T}(\lam)$ to counting of gradient trees of type $T$. Firstly, we define the function $\vec{\dist}_T : M^{|V(T)|} \rightarrow \real_+$ according to the combinatorics of $T$ by
\begin{multline}\label{length_of_tree}
\vec{\dist}_T(\vec{x}) = \sum_{e_{ij} \in E(T)} \dist_{ij}(x_{S}(e_{ij}),x_{E}(e_{ij}))+ \\
\sum_{i=0}^{k-1}\dist_{i(i+1)}(q_{i(i+1)},x_{E}(e_{i(i+1)}))+\dist_{0k}(q_{0k},x_{S}(e_{0k})).
\end{multline}
Here the variables $\vec{x}$ are labelled by the vertices of $T$. ($x_{S}(e)$ and $x_{E}(e)$ refer to the variables corresponding to vertices which are starting point and endpoint of the edge $e$ respectively.) Recall that $E(T)$ is the set of internal edges of $T$ and each interior edge $e$ has a unique label by two integers as $e_{ij}$, corresponding to the Morse function $f_{ij}= f_j-f_i$. The notation $\dist_{ij}$ refers to the Agmon distance corresponding to the Morse function $f_{ij}$. 

$\vec{\dist}_T(\vec{x})$ is the length function of a geodesic tree (may not be unique) with topological type $T$, with interior vertices $\vec{x}$ and semi-infinite edges ended at critical points $q_{ij}$. Similar to the case of $m_3(\lam)$, we have the following lemma.
\begin{lemma}\label{lem:m_k_symplectic_constant}
The function $\vec{\dist}_T$ is bounded below by $A= f_{01}(q_{01})+\dots+ f_{(k-1)k}(q_{(k-1)k})-f_{0k}(q_{0k})$, and it attains minimum at $\vec{x}$ if and only if $\vec{x}$ is the vector consisting of interior vertices of a gradient flow tree of $\vec{f}$ of type $T$ ended at the corresponding sequence of critical points $\vec{q}$.
\end{lemma}

\begin{proof}
The proof relies on the fact (see \cite{HelSj4}) that we have \[
|f_{ij}(x)-f_{ij}(y)|\leq \dist_{ij}(x,y),\]
if $f_{ij}$ is a Morse function on $M$, and $\dist_{ij}(x,y)$ is the Agmon distance. Furthermore, the equality $f_{ij}(x)-f_{ij}(y)=\dist_{ij}(x,y)$ forces the geodesic from $y$ to $x$ to be a generalized integral curve of $\nabla f_{ij}$ by Lemma \ref{lem:agmon_dist_flow_line}. We apply this fact to each term in (\ref{length_of_tree}) and the result follows.
\end{proof}

Similar to the $m_3(\lam)$ case, every gradient flow tree $\Gamma \in \mathcal{M}(\vec{f},\vec{q})(T)$ is associated with a unique minimum point $\vec{x}_{\Gamma} \in M^{|V(T)|}$ of $\vec{\dist}_T$. For each tree, we take a covering $W_\Gamma$ of $\vec{x}_\Gamma$, given by a product $W_\Gamma = \prod_{v \in V(T)} W_{\Gamma,v}$, where each $W_{\Gamma,v}$ is an open subsets in $M$ containing $x_v$ such that all $W_{\Gamma,v}$'s are disjoint from each other. If we further take $D_\Gamma=\prod_{v \in V(T)} D_{\Gamma,v}$ such that $\overline{D_{\Gamma,v}} \subset W_{\Gamma,v}$, we have a constant $C>0$ depending on size of $D_\Gamma$'s such that $\vec{\dist}_T \geq A+C$ on $ M^{|V(T)|} \setminus D_\Gamma$ (here $A$ is the constant in the lemma \ref{lem:m_k_symplectic_constant}). We are going to localize the integral \eqref{mkintegral0} as follows.

We take a finite covering of $M$ with balls $\{B(x,r)\}_{B(x,r) \in \mathcal{J}}$ of radius $r$ centering at $x$, with a partition of unity $\{\chi_B\}_{B \in \mathcal{J}}$ subordinating to it. We choose a covering $\{B_r(\vec{x})\}_{B \in \mathcal{I}}$ of $M^{|V(T)|}$ given by product $B_r(\vec{x}) = \prod_{v \in V(T)} B(x_v,r)$, where $B(x_v,r) \in \mathcal{J}$. We decompose $\mathcal{I} = \mathcal{I}_1 \cup \mathcal{I}_2$ such that $B\cap\overline{D_\Gamma}$ is empty for all $B\in\mathcal{I}_2$ and $\overline{B} \subset W_{\Gamma}$ for all $B\in\mathcal{I}_1$.
%$B \in \mathcal{I}_2$ are those having empty intersection with $\overline{D_\Gamma}$, and $B \in \mathcal{I}_1$ satisfying $\overline{B} \subset W_{\Gamma}$.
These can be achieved by choosing sufficiently small $r$.

We can take cutoff functions subordinating to the covering $\{B\}_\mathcal{I}$, given by product of functions $\chi_B$ on $M$. We write $\vec{\chi}_{B} = \prod_{v\in V(T)} \chi_{B(x_v,r)}$ which is a function supported in $B$. We will use $\vec{\chi}_{B}$ to cut off the following integral
\begin{equation}\label{mkintegral}
\int_{M}  \deprodtree{k}{T}(\phi(q_{(k-1)k}),\dots,\phi(q_{01}))\wedge \frac{\ast \phi_{0k}}{\|\phi_{0k}\|^2}.
\end{equation}
Recall that the $\deprodtree{k}{T}$ is defined by using wedge product and the homotopy operators $H_{ij}$ and the combinatorics of the tree $T$. We cut off the operation $m_k^{T}(\lam)$ using the function $\chi_{B(x_v,r)}$ whenever taking wedge product at the vertex $v$. We will write $m_k^T(\lam,\vec{\chi})$ for the integral after cutting off by $\vec{\chi}$. Therefore we have
\begin{equation}
m_k^T(\lam)(\phi(\vec{q})) = \sum_{B \in \mathcal{I}_1} m_k^T(\lam,\vec{\chi}_{B})(\phi(\vec{q})) + \sum_{B \in \mathcal{I}_2} m_k^T(\lam,\vec{\chi}_{B})(\phi(\vec{q})),
\end{equation}
where $m^T_k(\lam,\vec{\chi}_{\vec{B}})(\phi(\vec{q}))$ stand for $A_\infty$ operation after cutting off by $\vec{\chi}_{\vec{B}}$. Recall that there is a unique root vertex $v_r$ associated to the direct tree $T$, by applying the resolvent estimate in lemma \ref{resolventlemma} and the estimate in lemma \ref{eigenestimate2}, we obtain the following:

\begin{lemma}
For any $\epsilon>0$, there exist $r(\epsilon), \lam(\epsilon)>0$ such that if we take the covering of radius $r < r(\epsilon)$, we have
\begin{equation}
\|m_k^T(\lam,\vec{\chi}_{B})(\phi(\vec{q}))\wedge \frac{\ast \phi_{0k}}{\|\phi_{0k}\|^2} \|_{L^{\infty}(M)} = \mathcal{O}_{r,\epsilon}(e^{-\lam (\vec{\dist}_T(\vec{x})-A-\epsilon)})
\end{equation}
for any $\lam>\lam(\epsilon)$, where $\vec{x}$ is the center of the ball $B$.
\end{lemma}

The proof is essentially the same as the case for $m_3(\lam)$. Similarly, we have
$$
\sum_{B \in \mathcal{I}_2}\int_{M} m_k^T(\lam,\vec{\chi}_{B})\wedge \frac{\ast \phi_{0k}}{\|\phi_{0k}\|^2} = \mathcal{O}_{r, \epsilon} (e^{-\lam(\frac{C}{2})}),
$$
for sufficiently large $\lam$. It follows from the fact that $\vec{\dist}_T(\vec{x}) \geq A + C$ for those covering in $\mathcal{I}_2$. This result basically says that the integral $\deprodtree{k}{T}$ can be localized to gradient flow tree using the cutoff mentioned above. To summarize, we have the following proposition.
\begin{prop}\label{apriori_cutoff}
For each gradient flow tree $\Gamma$, there is a sequence of cutoff functions $\{\vec{\chi}_\Gamma\}$ which is supported in $W_\Gamma$ and satisfies $\vec{\chi}_\Gamma \equiv 1$ on $\overline{D_\Gamma}$ such that
\begin{eqnarray*}
&& \int_{M} \deprodtree{k}{T}(\phi(\vec{q}))\wedge \frac{\ast \phi_{0k}}{\|\phi_{0k}\|^2}\\
&=& \sum_{\Gamma \in \mathcal{M}(\vec{f},\vec{q})(T)} \int_{M} m_k^T(\lam,\vec{\chi}_{\Gamma}) (\phi(\vec{q}))\wedge \frac{\ast \phi_{0k}}{\|\phi_{0k}\|^2}+ \mathcal{O}(e^{-\lam(\frac{C}{2})}),
\end{eqnarray*}
when $\lam$ is sufficiently large.
\end{prop}

\begin{remark}
In the above argument, the neighborhood $W_\Gamma$ can be chosen to be arbitrary small. We will obtain a smaller constant $C$ if we shrink the neighborhood $W_{\Gamma}$.
\end{remark}

After localizing the integral, we move on to the section concerning WKB approximation which helps to compute of the leading order contribution of $m_k^T(\lam,\vec{\chi}_{\Gamma})$.

%----------------------WKB for Mk case----------------------------------------------------------
\subsubsection{WKB method for $m_k$}\label{wkbinduction}
We consider a gradient tree $\Gamma$ of type $T$, with $k$ semi-infinite incoming edges. Recall in section \ref{treedefinition} that each edge in $T$ is assigned with a label by two integers $i$ and $j$. We will use $ij$ to represent an edge in $T$ and denote the corresponding edge in the gradient tree $\Gamma$ by $e_{ij}$. The vertex in the gradient tree corresponding to $v$ in $T$ will be denoted by $x_{v}$. We again omit the dependence on $\Gamma$ in our notations as it is already fixed. We are going to associate $\phi_{(ij,v)} \in \Omega^*_{ij}(M,\lam)$, together with its WKB approximation 
\[
\phi_{(ij,v)} \sim e^{-\lam \psi_{(ij,v)}}\lam^{r_{(ij,v)}}(\omega_{(ij,v),0}+\omega_{(ij,v),1} \lam^{-1/2}+\dots)\]
in some neighborhood $W_v$ of $x_v$ to each flag $(ij,v)$ as shown in the figure \ref{fig:neighborhoodW}. We also fix cutoff functions $\chi_{v}$ supported in $W_v$ and study the integral $m_k^T(\lam,\vec{\chi})(\vec{q})$ using the arguments in section \ref{aprioriestimate}.

\begin{figure}[h]
\centering
\includegraphics[scale=0.25]{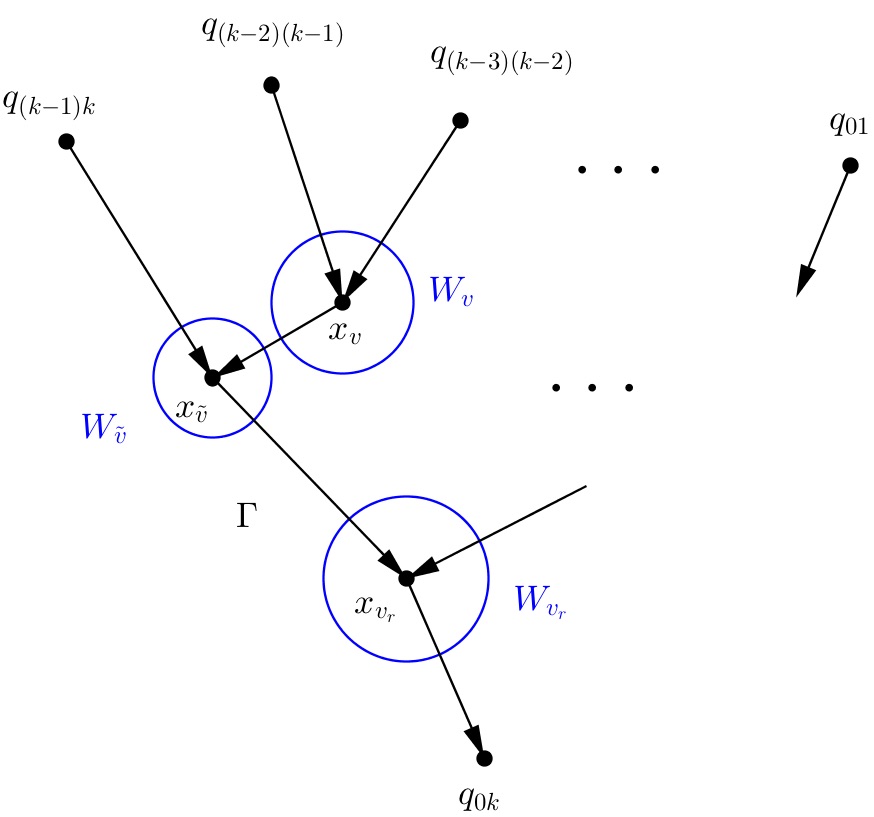}
\caption{}\label{fig:neighborhoodW}
\end{figure}

We define the followings inductively.
\begin{itemize}
\item[(1)] for a semi-infinite incoming edge $i(i+1)$ which ends at vertex $v$, we take $\phi_{(i(i+1),v)}$ to be the input $\phi_{i(i+1)}$, with its WKB approximation in $W_v$ as in lemma \ref{eigenwkb}. We also let $g_{(i(i+1),v)}= \psi_{(i(i+1),v)} - f_{i(i+1)}$. We also choose $W_v$ to be small enough so that the WKB approximations of all input forms associated to edges connected to $v$ holds in $W_{v}$;
\item[(2)] for an internal edge $il$ which starts at vertex $v$, $v$ must be the endpoint of edges $ij$ and $jl$ as shown in figure \ref{fig:interiorvertex},
\begin{figure}[h]\label{interiorvertex}
\centering
\includegraphics[scale=0.22]{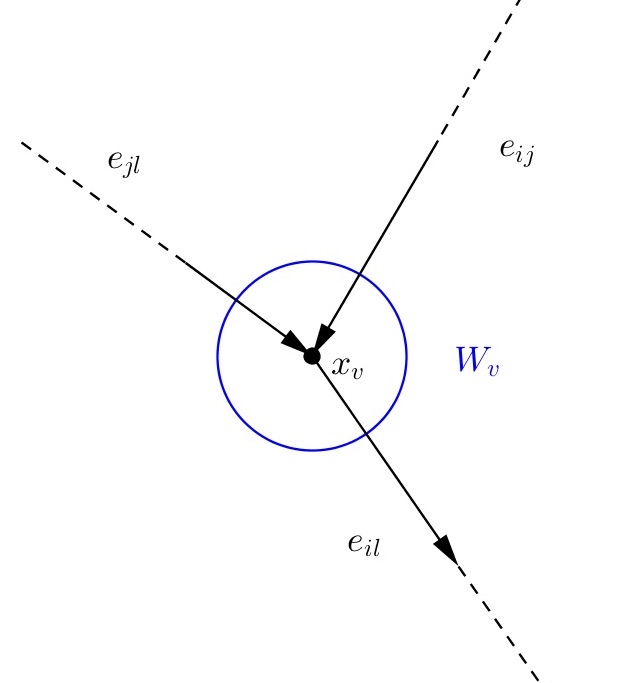}
\caption{}\label{fig:interiorvertex}
\end{figure}
we take $\phi_{(il,v)} = \phi_{(jl,v)} \wedge \phi_{(ij,v)}$. The WKB expression of $\phi_{(il,v)}$ is defined by the following equations:
%comes from the expression of $\phi_{(jl,v)} \wedge \phi_{(ij,v)}$, which means
\begin{eqnarray*}
\psi_{(il,v)} &=& \psi_{(ij,v)}+\psi_{(jl,v)},\\
\omega_{(il,v),n} &=& \sum_{m+m'=n} \omega_{(jl,v),m} \wedge \omega_{(ij,v),m'},\\
r_{(il,v)} &=& r_{(jl,v)} + r_{(ij,v)}.
\end{eqnarray*}
We also let $g_{(il,v)} = g_{(ij,v)}+g_{(jl,v)}$;
\item[(3)] for an internal edge $ij$ with its starting vertex $v_S$ and ending vertex $v_E$ as shown in figure \ref{fig:internaledge},
\begin{figure}[h]
\centering
\includegraphics[scale=0.22]{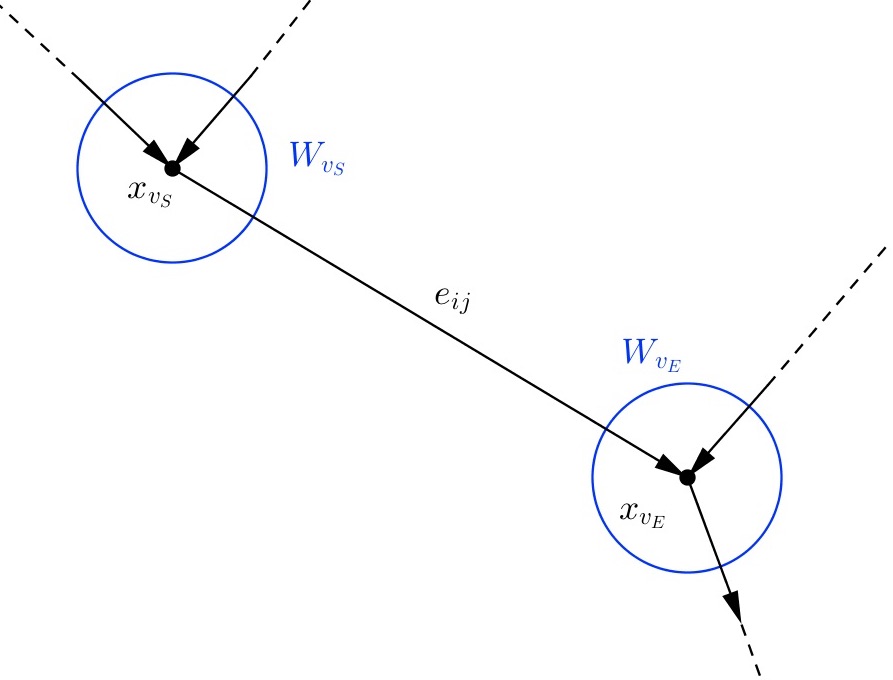}
\caption{}\label{fig:internaledge}
\end{figure} 
we define $\phi_{(ij,v^{}_{\scalebox{.7}{$\scriptscriptstyle E$}})}= H_{ij}(\chi_{v^{}_{\scalebox{.7}{$\scriptscriptstyle S$}}}\phi_{(ij,v^{}_{\scalebox{.7}{$\scriptscriptstyle S$}})})$ in $W_{v^{}_{\scalebox{.7}{$\scriptscriptstyle E$}}}$ and the corresponding WKB approximation can be obtained from lemma \ref{homotopywkb} if $\supp(\chi_{v^{}_{\scalebox{.7}{$\scriptscriptstyle S$}}})$ and $W_{v^{}_{\scalebox{.7}{$\scriptscriptstyle E$}}}$ are chosen to be small enough.
%we take the WKB approximation in lemma \ref{homotopywkb} of $\phi_{(ij,v^{}_{\scalebox{.7}{$\scriptscriptstyle E$}})}= H_{ij}(\chi_{v^{}_{\scalebox{.7}{$\scriptscriptstyle S$}}}\phi_{(ij,v^{}_{\scalebox{.7}{$\scriptscriptstyle S$}})})$ in $W_{v^{}_{\scalebox{.7}{$\scriptscriptstyle E$}}}$ by taking $\supp(\chi_{v^{}_{\scalebox{.7}{$\scriptscriptstyle S$}}})$ and $W_{v^{}_{\scalebox{.7}{$\scriptscriptstyle E$}}}$ small enough for applying the lemma if necessary.
We also define $g_{(ij,v^{}_{\scalebox{.7}{$\scriptscriptstyle E$}})} = \psi_{(ij,v^{}_{\scalebox{.7}{$\scriptscriptstyle E$}})} - f_{ij}$ and $r_{(ij,v^{}_{\scalebox{.7}{$\scriptscriptstyle E$}})} = r_{(ij,v^{}_{\scalebox{.7}{$\scriptscriptstyle S$}})} - \half$.
\item[(4)] for the semi-infinite outgoing edge $0k$ with the root vertex $v_r$, we take $\phi_{(0k,v_r)}$ to be the form $\phi_{0k}$, with WKB approximation from lemma \ref{eigenwkb}. We also define $g_{(0k,v_r)} = \psi_{(0k,v_r)} + f_{0k}$. 
\end{itemize}

\begin{remark}
In section \ref{aprioriestimate}, $\supp(\chi_{\Gamma,v})$ at each internal vertex $v$ has to be chosen to be small enough so that lemma \ref{homotopywkb} can be applied.
%We need to choose the size of cutoff $\supp(\chi_{\Gamma,v})$ appearing in the previous section \ref{aprioriestimate} at each internal vertex $v$ small enough for apply lemma \ref{homotopywkb}.
\end{remark}

From the definition of $m_k^T(\lam,\vec{\chi}_\Gamma)$, we see that 
\begin{equation*}
\int_{M}  m_k^T(\lam,\vec{\chi}_\Gamma)(\phi_{(k-1)k},\dots,\phi_{01})\wedge \frac{\ast \phi_{0k}}{\|\phi_{0k}\|^2}
= \int_{M} \phi_{(jk,v_r)} \wedge \phi_{(0j,v_r)} \wedge \frac{\ast \phi_{(0k,v_r)}}{\| \phi_{(0k,v_r)}\|^2},
\end{equation*}
if three edges $0j$, $jk$ and $0k$ are meeting at the root vertex $v_r$.
Applying lemma \ref{eigenwkb} to input forms $\phi_{i(i+1)}$ and lemma \ref{homotopywkb} to homotopy operators $H_{ij}$ along internal edges $e_{ij}$, we prove that each WKB approximation
$$
\phi_{(ij,v)} \sim e^{-\lam \psi_{(ij,v)}}\lam^{r_{(ij,v)}}(\omega_{(ij,v),0}+\omega_{(ij,v),1}\lam^{-1/2}+\dots)
$$
is an $C^\infty$ approximation with error $e^{-\lam \psi_{(ij,v)}}\mathcal{O}(\lam^{-N})$ for arbitrary $N \in \inte_+$. Therefore, we can replace each $\phi_{(ij,v)}$ by the first term in its WKB approximation for computing the leading order contribution. We obtain
\begin{multline}\label{mkfinalint}
\langle m_k^T(\lam,\vec{\chi}_\Gamma)(\phi_{(k-1)k},\dots,\phi_{01}), \frac{\phi_{0k}}{\| \phi_{0k}\|^2}\rangle\\
=\lbrace \lam^{r_{(jk,v_r)} + r_{(0j,v_r)} +r_{(0k,v_r)}} \int_M e^{-\lam (\psi_{(jk,v_r)}+\psi_{(0j,v_r)}+\psi_{(0k,v_r)})}\\
\chi_{v_r}(\omega_{(jk,v_r),0}\wedge \omega_{(0j,v_r),0}\wedge \frac{\ast \omega_{(0k,v_r),0}}{\| \phi_{0k}\|^2} )\rbrace(1+\mathcal{O}(\lam^{-1/2})).
\end{multline}

%------------------Explicit computation for mk case------------------------------------------------

\subsubsection{Explicit computation for $m_k$}
The argument of the general case is similar to the case $k=3$, with more combinatorics involved. %As in section \ref{wkbinduction}, we fix a gradient tree $\Gamma$ of type $T$.
Similar to the previous section, we may drop the dependence of $\Gamma$ in our notations. We are going to show that 
\begin{equation}
\int_{M}  m_k^T(\lam,\vec{\chi}_{\Gamma})\wedge \frac{\ast \phi_{0k}}{\|\phi_{0k}\|^2} =\pm (1+\mathcal{O}(\lam^{-1/2})),
\end{equation}
where the sign agrees with that associated to the gradient tree $\Gamma$ in Morse category. We begin with some notations associated to $\Gamma$. 

\begin{notation}\label{flowtreeorientation1}
Given a gradient tree $\Gamma$, we inductively associate to each flag $(ij,v)$ an oriented closed submanifold $V_{(ij,v)}\subset W_{v}$ by specifying orientation of its normal bundle. We require:
\begin{itemize}
\item[(1)] for each semi-infinite incoming edge $i(i+1)$ with ending vertex $v$, we let $V_{(i(i+1),v)}:= V_{q_{i(i+1)}}^+ \cap W_{v}$, where $V_{q_{i(i+1)}}^+$ is the stable submanifold of $f_{i(i+1)}$ from the critical point $q_{i(i+1)}$ with the chosen orientation $\nu_{(i(i+1),v)}$ equals to that in the Morse category;
\item[(2)] for an internal edge $il$ with its starting vertex $v$ and assume $ij$ and $jl$ are two incoming edges meeting $e_{il}$ at $v$ as in figure \ref{fig:interiorvertex}. We let $V_{(il,v)} = V_{(ij,v)}\cap V_{(jl,v)}$ (the intersections is transversal from the generic assumption) and $\nu_{(il,v)} = \nu_{(jl,v)} \wedge \nu_{(ij,v)}$, if $\nu_{(ij,v)}$ and $\nu_{(jl,v)}$ are two corresponding orientation forms;
\item[(3)] for an internal edge $ij$ with its starting vertex $v_S$ and ending vertex $v_E$, we define $V_{(ij,v^{}_{\scalebox{.7}{$\scriptscriptstyle E$}})}$ to be $V_E$ obtained from applying lemma \ref{homotopywkb} to the homotopy operator $H_{ij}$. The orientation form $\nu_{(ij,v^{}_{\scalebox{.7}{$\scriptscriptstyle E$}})}$ is chosen such that $[\nu_{(ij,v^{}_{\scalebox{.7}{$\scriptscriptstyle E$}})}] = [df_{ij} \wedge \nu_{(ij,v^{}_{\scalebox{.7}{$\scriptscriptstyle S$}})}]$, under the identification by flow of $\nabla f_{ij}$;
\item[(4)]  for the semi-infinite incoming edge $0k$ with root vertex $v_r$, we let $V_{(0k,v_r)}:= V_{q_{0k}}^- \cap W_{v_r}$, where $V_{q_{0k}}^-$ is the unstable submanifold of $f_{0k}$ from critical point $q_{0k}$ with the chosen orientation $\nu_{(0k,v_r)}$ equal to that in the Morse category.
\end{itemize}
We further choose an isomorphism and projection map for every flag $(ij,v)$
\begin{equation}
\begin{CD}
W_v    @>\cong>>    NV_{(ij,v)}\\
@V\pi_{(e,v)}VV  @V\pi_{NV_{(ij,v)}}VV\\
V_{(ij,v)}    @= V_{(ij,v)},\\
\end{CD}
\end{equation}
by further shrinking $W_v$ suitably.\\
\end{notation}

We can therefore assign a sign to the gradient tree $\Gamma$ in the following way.

\begin{definition}\label{flowtreeorientation2}
For a generic sequence of Morse functions $\vec{f}$ with corresponding critical points $q_{01},\dots, q_{(k-1)k}, q_{0k}$ satisfying the degree condition \eqref{degree_eq}, which gives a gradient tree $\Gamma$, we define
\begin{equation}
sign(\Gamma) = sign(\frac{\nu_{(jk,v_r)} \wedge \nu_{(0j,v_r)}\wedge \nu_{(0k,v_r)}}{vol_g}),
\end{equation}
where $0j, jk$ and $0k$ are edges joining the root vertex $v_r$ as in section \ref{fig:interiorvertex}, $\nu_{(ij,v)}$ is the orientation of normal bundle defined in notation \ref{flowtreeorientation1} and $\nu_{(0k,v_r)}$ is the orientation of chosen for $V_{q_{0k}}^-$. 
\end{definition}

%-----------------------------------------------------------------------------------

We are going to show that\[
\int_{N(V_{(ij,v)})_{x_v}} (e^{-\lam g_{(ij,v)}}\lam^{r_{(ij,v)}} \chi_{v}\omega_{(ij,v),0}) = (1 + \mathcal{O}(\lam^{-1})),\]
for any flag $(ij,v)$ except the outgoing edge $0k$, where $r_{(ij,v)}$ is the number of internal edges before the vertex $v$. This can be seen inductively along the tree $T$. We see that:
\begin{itemize}
\item[(1)] it is true for the semi-infinite incoming edge $i(i+1)$ by lemma \ref{eigenwkbcal};\\
\item[(2)] for an internal edge $il$ with starting vertex $v$ and assume $ij$ and $jl$ are two incoming edges meeting $il$ at $v$, we have
\begin{eqnarray*}
&& \lam^{r_{(il,v)}} \int_{N(V_{(il,v)})_{x_v}} e^{-\lam g_{(il,v)}}\chi_{v} \omega_{(il,v),0}\\
&\equiv&\lam^{r_{(jl,v)} + r_{(ij,v)}} \int_{N(V_{(jl,v)} \cap V_{(ij,v)})_{x_v}} e^{-\lam (g_{(jl,v)} + g_{(ij,v)})}\chi_{v} \omega_{(jl,v),0} \wedge \omega_{(ij,v),0} \\
&\equiv&(\lam^{r_{(jl,v)}} \int_{N(V_{(jl,v)})_{x_{v}}} e^{-\lam g_{(jl,v)}} \chi_{v}\omega_{(jl,v),0}) (\lam^{r_{(ij,v)}}\int_{N(V_{(ij,v)})_{x_{v}}} e^{-\lam g_{(ij,v)}}\chi_{v} \omega_{(ij,v),0})\\
&\equiv& 1,\\
\end{eqnarray*}
modulo an error of order $\mathcal{O}(\lam^{-1})$;\\
\item[(3)] for an internal edge $ij$ with starting vertex $v_S$ and ending vertex $v_E$, we make use of lemma \ref{homotopywkbcal} as before. 
\end{itemize}

We can now calculate the leading contribution from the integral \eqref{mkfinalint}. Recall that we have
\begin{equation}
\psi_{(0j,v_r)}+\psi_{(jk,v_r)} - f_{0k} = g_{(0j,v_r)}+g_{(jk,v_r)}.
\end{equation}
Therefore we obtain
\begin{eqnarray*}
&&\lam^{r_{(0j,v_r)}+r_{(jk,v_r)}+r_{(0k,v_r)}} \lbrace\int_{M} e^{-\lam(\psi_{(0j,v_r)}+\psi_{(jk,v_r)}+\psi_{(0k,v_r)})} 
\\
&&\;\;\;\;\;\;\;\;\;\;\;\;\;\;\;\;\;\;\;\;\;\;\;\;\;\;\;\;\;\;\;\; \chi_{v_r}\cdot(\omega_{(jk,v_r),0}\wedge \omega_{(0j,v_r),0}\wedge \frac{\ast\omega_{(0k,v_r),0}}{\| \phi_{0k}\|^2}) \rbrace \\
&=&\lam^{r_{(0j,v_r)}+r_{(jk,v_r)}+r_{(0k,v_r)}}\lbrace \int_M e^{-\lam (g_{(0j,v_r)}+g_{(jk,v_r)}+g_{(0k,v_r)})}\\
&&\;\;\;\;\;\;\;\;\;\;\;\;\;\;\;\;\;\;\;\;\;\;\;\;\;\;\;\;\;\;\;\;  \chi_{v_r}(\omega_{(jk,v_r),0}\wedge \omega_{(0j,v_r),0} \wedge \frac{\ast\omega_{(0k,v_r),0}}{\| \phi_{0k}\|^2}) \rbrace\\
&=& \pm  (1+ \mathcal{O}(\lam^{-1})),
\end{eqnarray*}
which means
\begin{equation}
\int_M m_k^T(\lam,\vec{\chi}_\Gamma) \wedge \frac{\ast \phi_{0k}}{\|\phi_{0k}\|^2}=\pm (1+\mathcal{O}(\lam^{-1/2})).
\end{equation}
The sign $\pm$ comes from matching the orientation $[\nu_{(jk,v_r)} \wedge \nu_{(0j,v_r)}\wedge \nu_{(0k,v_r)}]$ against that of $vol_g$, which agrees with the sign in Morse category. This completes the proof of our main theorem.
% !TEX root = paper.tex
\section{WKB for Green operator}\label{approximation}
In lemma \ref{resolventlemma}, we have a rough estimate for the twisted Green operator by a Morse function $f$, or the homotopy operator $H_f = d^*_f G_f (I -P_f)$, with an error of order $\mathcal{O}(e^{\lam \epsilon})$. In a neighborhood of the gradient flow line segment of $f$, we are going to improve this results to estimate with error $\mathcal{O}(\lam^{-N})$ for an arbitrary $N \in \inte_+$. This is done by the WKB method for \textit{inhomogeneous} Laplace equation \eqref{homotopyeq}.\\

We study the local behavior of the homotopy operator $H_f$ along a normalized gradient flow line segment
\begin{eqnarray*}
\gamma :[0,T] &\longrightarrow &M,\\
 \frac{d\gamma}{dt} &=&\frac{\nabla f }{| \nabla f |_f},\\
   \gamma(0) = x_S& ,&  \gamma(T)  = x_E,
\end{eqnarray*}
with $\nabla f |_\gamma \neq 0$ along $\gamma$, as shown in figure \ref{fig:flow_line}.\\

Suppose that $\zeta_E = H_f (\chi_S \zeta_S)$ and we have a WKB approximation of $\zeta_S$ in $W_S$ of the form 
\begin{equation}\label{awkbexpansion}
\zeta_S \sim e^{-\lam \psi^{}_{\scalebox{.7}{$\scriptscriptstyle S$}}}(\omega_{S,0}+\omega_{S,1} \lam^{-1/2}+ \omega_{S,2} \lam^{-1} + \dots),
\end{equation}
we aim at establishing a similar expression
\begin{equation}\label{bwkbexpansion1}
\zeta_E \sim  \lam^{-1/2}e^{-\lam \psi^{}_{\scalebox{.7}{$\scriptscriptstyle E$}}}(\omega_{E,0} + \omega_{E,1}\lam^{-1/2} + \dots) 
\end{equation}
of $\zeta_E$ in some open neighborhood $W_E$ of $x_E$. It is possible to propagate the estimate along $\gamma$ since $\nabla f \neq 0$ along $\gamma$.

\begin{figure}[h]
\centering
\includegraphics[scale = 0.2]{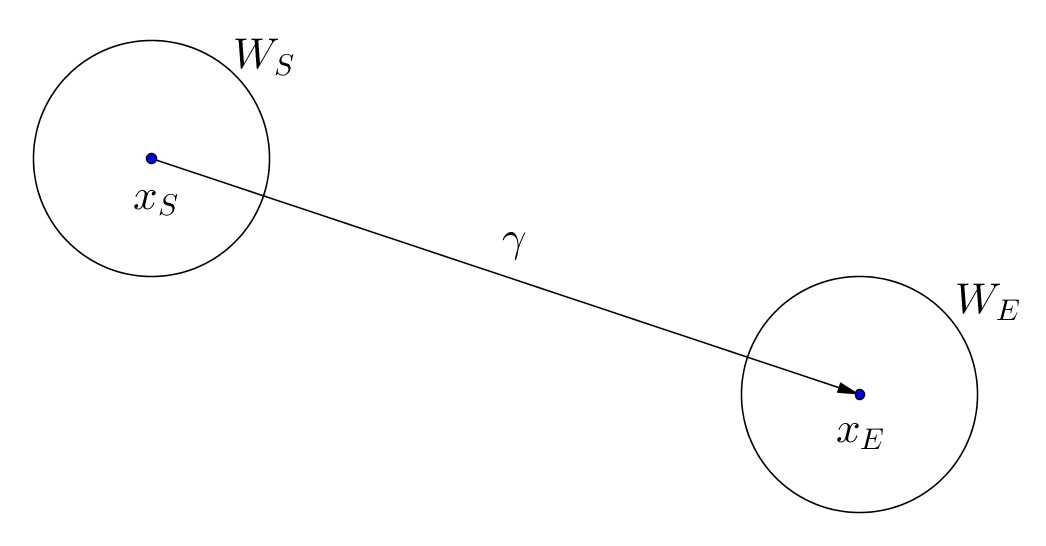}
\caption{}\label{fig:flow_line}
\end{figure}

The key step is to determine $\psi_E$, which is given in the following subsection. As the first trial, we consider the function 
\[\tpsi_E(x):=\inf_{y\in W_S}\{\psi_S(y)+\dist_f(y,x)\},\]
since $e^{-\lam \tpsi^{}_{\scalebox{.7}{$\scriptscriptstyle E$}}}$ is the expected exponential decay suggested by the resolvent estimate in lemma \ref{resolventlemma}.

Unfortunately, $\tpsi_E$ is not the correct choice since it is singular along a hypersurface $U_S$ through $x_S$, and it cannot be used for the iteration process as we keep on differentiating it. 

In the coming section \ref{exp_decay}, we will solve the minimal configuration in variational problem associated to $\inf_{y\in W_S} (\psi_S(y) + \dist_f(y,x))$ and find that the point $y$ must lie on $U_S$, with a unique geodesic joining $x$ which realizes $\dist(y,x)$, for those $x$ closed enough to $x_E$. These family of geodesics $\{ \gamma_y\}_{y \in U_S}$ gives a foliation of a neighborhood of $\gamma$. Therefore we can use $\psi_E(\gamma_y(t)) = \psi_S(y) + t $ as an extension of $\tilde{\psi}_E$ across $U_S$. We then use $\psi_E$ in the iteration similar to classical WKB approximation to obtain the above expansion \eqref{bwkbexpansion1}.

\subsection{The phase function $\psi_E$}\label{exp_decay}
We apply variational method to study the function $\tpsi_E(x)$. Fixing $x\in M$, we define $\alpha(\epsilon,t) := \alpha_\epsilon(t) :(-\epsilon_0,\epsilon_0)\times[0,1]\rightarrow M\setminus Crit(f)$ such that $\alpha_\epsilon(1)\equiv x$ for all $\epsilon$. To minimize the functional\[
L(\epsilon)=\psi_S(\alpha_\epsilon(0))+\int_0^1|\pd_t \alpha_\epsilon|_f dt,\]
we take derivative and get
\begin{lemma} (First variation formula)\label{firstvariation}
\begin{equation}
\frac{dL}{d\epsilon} = \langle \tilde{\nabla}\psi_S(\alpha_\epsilon),\pd_{\epsilon}\alpha_\epsilon\rangle_f|_{t=0} + \int_0^1\frac{1}{|\pd_t \alpha|_f}\langle\tilde{\nabla}_t\pd_{\epsilon}\alpha,\pd_t \alpha \rangle_f dt.
\end{equation}
Here $\tilde{\nabla}$ is the Levi-Civita connection corresponding to the Agmon metric $\langle \cdot , \cdot \rangle_f$ in definition \ref{agmondistance}. 
\end{lemma}
If we assume $\alpha_0$ is a geodesic (with respect to twisted metric $|df|^2 g$) with $|\alpha_0'(t)|_f$ being constant, the Euler-Lagrange equation for $L(\epsilon)$ is\[
\left.\frac{dL}{d\epsilon}\right|_{\epsilon=0}=\langle \tilde{\nabla}\psi_S(\alpha_0)-\frac{\alpha_0'}{|\alpha_0'|_f},\pd_{\epsilon}\alpha \rangle_f\Big\vert_{\substack{t=0\\ \epsilon=0}}=0.\]
Since $\pd_\epsilon \alpha (0,0)$ can be chosen arbitrarily, we have
\begin{equation}\label{ELequation}
\Big(\tilde{\nabla}\psi_S(\alpha_0)-\frac{\alpha_0'}{|\alpha_0'|_f}\Big)\Big|_{t=0} = 0.
\end{equation}
Such an equation restricts the possibility of the starting point $\alpha_0(0)$, namely, we have
$$
| \nabla \psi_S | = | \nabla f |,
$$
at $\alpha_0(0)$, or equivalently, $| \tilde{\nabla} \psi_S |_f = 1$.
\begin{definition}
$$
U_S := \{| \tilde{\nabla} \psi_S |_f =1 \} \cap W_S.
$$
\end{definition}
If $\alpha_0$ is a local extrema of $L$ with $\alpha_0(0) \in W_S$, it forces $\alpha_0(0) \in U_S$. To obtain nice properties of $U_S$, we are going to assume the following throughout the whole section.
\begin{assum}\label{assumption_psi_S}
We define $g_S:W_S\rightarrow\real_{\geq0}$ by $g_S = \psi_S -f$ and assume it to be a smooth Bott-Morse function in $W_S$ with critical point set $V_S$ which contains $x_S$.
%We assume $g_S:W_S\rightarrow\real_{\geq0}$, defined by $g_S = \psi_S -f$, be a smooth Bott-Morse function in $W_S$ with critical point set $V_S$ such that $x_S\in V_S$. 
\end{assum}

\begin{lemma}\label{hypersurface_H}
$U_S$ is a hypersurface containing $V_S$ if $dim(V_S)<dim(M)$ (we shrink $W_S$ if necessary). Otherwise, it is simply $V_S = W_S$.
\end{lemma}

\begin{proof}
Since we have $\nabla g_S \equiv 0$ on $V_S$ and hence $|\nabla\psi_S|=|\nabla f|$ on $V_S$. This gives $V_S \subset U_S$. Moreover, $U_S$ can be defined by the equation\[
\Phi(x)=2\langle\nabla f(x), \nabla g_S(x)\rangle + |\nabla g_S(x)|^2=0.\]
If $v\in T_p M$ where $p\in V_S$, then we have
\begin{eqnarray*}
\nabla_v\Phi(p)&=&2\Hess f(p)(v,\nabla g_S(p)) + 2\Hess g_S(p)(v,\nabla f(p)) + 2\Hess g_S(p)(v,\nabla g_S(p))\\
&=&2\Hess g_S(p)(v,\nabla f(p)),
\end{eqnarray*}
since $\nabla g_S(p)=0$ on $V_S$. As $g_S$ is a Bott-Morse function with critical set $V_S$, $\Hess g_S(p)$ is nondegenerate when it is restricted on the orthogonal complement of $T_pV_S$ in $T_pM$. Therefore, there exists $v$ such that $\nabla_v\Phi(p)\neq 0$.
\end{proof}

We are going to parametrize a neighborhood of $\gamma$ by $U_S \times (-\delta,T+ \delta)$ such that $U_S \times \{0\}\rightarrow M$ is an embedding and $x_S \times [0,T]$ is $\gamma$. $\psi_E$ is defined to be the coordinate function corresponding to the last variable. \\

Motivated from equation \eqref{ELequation}, we define a transversal vector field on $U_S$ which is the initial tangent vector for minimizer of $L$.
\begin{definition}We define a vector field $\nu \in \Gamma(U_S, T_M)$ transversal to $U_S$ (shrinking $W_S$ if necessary) by
\begin{equation}
\nu : = \frac{\nabla \psi_S}{| \nabla \psi_S |_f} = \tilde{\nabla} \psi_S.
\end{equation}
Notice that $\nu = \frac{\nabla f }{|\nabla f|_f} = \tilde{\nabla} f$ on $V_S$.
\end{definition}

It follows from the Euler-Lagrange equation \eqref{ELequation} that any local extrema $\alpha$ of  $L$ will have $\alpha(0) \in U_S$ and $\alpha'(0)  = \nu(\alpha(0))$.
For convenience, we assume that $\gamma$ is extended to gradient flow line defined on $(a,b)$ containing $[0,T]$.
\begin{definition}
We define a map
\begin{equation}\label{sigmamap}
\sigma :W_0 \subset U_S \times (a , b) \rightarrow M, 
\end{equation}
by 
$$\sigma(u, t) = \tilde{\exp}_{u}(t \nu),$$
where $W_0$ is a suitable neighborhood of $\gamma$ where the exponential map $\tilde{\exp}$ with respect to the Agmon Riemannian metric is well defined.
\end{definition}

\begin{lemma}\label{sigmadiff}
Restricting to a small open neighborhood of $\{x_S\} \times [0,b)$, $\sigma$ is a diffeomorphism onto its image containing $\gamma$.
\end{lemma}
This is achieved by showing there is no ``conjugate point'' along $\gamma(t)$ for certain type of geodesic family, and using the fact that $\gamma$ being a global minimizer of functional $L$. Lemma \ref{sigmadiff} enables  us to construct $\psi_E$ needed for WKB approximation in a neighborhood $U_S \times (-\delta,b)$ (take a small enough $\delta$ and shrink $U_S$ if necessary) of $\gamma$ where $\sigma$ is a differeomorphism.
\begin{definition} \label{psidefinition}We define $\psi_E$ on $\sigma(U_S \times (-\delta,b))$ by 
\begin{equation}
\psi_E (\sigma(u,t)) = \psi_S(u)+t,
\end{equation}
for $(u,t) \in U_S \times (-\delta,b)$.
\end{definition}

\subsection{Well-definedness of the phase function $\psi_E$}
We prove lemma \ref{sigmadiff} in this section for ensuring the well-definedness of $\psi_E$. We begin with the second variation formula of $L$. Assume $\alpha:(-\epsilon_0,\epsilon_0)\times[0,l]\rightarrow M$ is a family such that $\alpha_0(t)$ is arc-length parametrized geodesic (with respect to twisted metric $|df|^2 g$) satisfying the condition
$$
\Big(\tilde{\nabla} \psi_S(\alpha) -  \frac{\pd_t\alpha}{ | \pd_t\alpha |_f}\Big)\Big|_{\substack{t=0\\ \epsilon=0}} = 0.$$
From the first variation formula 
\[
\frac{dL}{d\epsilon} = \langle \tilde{\nabla} \psi_S(\alpha_\epsilon(0)), \pd_\epsilon\alpha_\epsilon(0) \rangle_f + \int_0^l \langle \tilde{\nabla}_t \pd_\epsilon\alpha , \frac{\pd_t\alpha}{| \pd_t\alpha |_f} \rangle_f\;dt,\]
we obtain
\begin{lemma}(Second variation formula)
\begin{multline}
\left.\frac{d^2L}{d\epsilon^2}\right|_{\epsilon=0}  = \langle \tilde{\nabla}_{\epsilon}\tilde{\nabla}\psi_S,\pd_\epsilon\alpha\rangle_f|_{t=0} + \langle \tilde{\nabla}\psi_S,\tilde{\nabla}_{\epsilon}\pd_\epsilon\alpha\rangle_f|_{t=0} + \left.\langle \tilde{\nabla}_\epsilon \pd_\epsilon\alpha , \pd_t\alpha \rangle_f\right|_0^l \\
+\int_0^l \langle \tilde{\nabla}_t \pd_\epsilon\alpha,\tilde{\nabla}_t \pd_\epsilon\alpha \rangle_f + \langle \tilde{R}(\pd_\epsilon\alpha, \pd_t\alpha) \pd_\epsilon\alpha,\pd_t\alpha \rangle_f - \langle \tilde{\nabla}_t \pd_\epsilon\alpha,\pd_t\alpha \rangle_f^2\;dt,
\end{multline}
where the right hand side is evaluated at $\epsilon=0$. Here $\tilde{R}$ is the curvature tensor with respect to $\langle \cdot,\cdot \rangle_f$. 
\end{lemma}

If we further impose the condition that $\pd_\epsilon\alpha(\epsilon, l) \equiv 0$ for all $\epsilon$, we have
\begin{multline}
\left.\frac{d^2L}{d\epsilon^2}\right|_{\epsilon=0}= \langle \tilde{\nabla}_\epsilon \tilde{\nabla} \psi_S , \pd_\epsilon\alpha \rangle_f|_{t=0} \\
+\int_0^l \langle \tilde{\nabla}_t \pd_\epsilon\alpha,\tilde{\nabla}_t \pd_\epsilon\alpha \rangle_f + \langle \tilde{R}(\pd_\epsilon\alpha, \pd_t\alpha) \pd_\epsilon\alpha,\pd_t\alpha \rangle_f - \langle \tilde{\nabla}_t \pd_\epsilon\alpha,\pd_t\alpha \rangle_f^2\;ds.
\end{multline}
Therefore we consider the bilinear form $I$ associated to the above quadratic form.
\begin{definition}
\begin{multline}
I(X,Y) =\tilde{\nabla}^2 \psi_S(X,Y)(0)+\int_0^l \langle \tilde{R}(X, \pd_t\alpha) Y,\pd_t\alpha \rangle_f dt \\
+ \int_0^l \langle \tilde{\nabla}_t X -\langle \tilde{\nabla}_t X,\pd_t\alpha \rangle_f \pd_t\alpha ,\tilde{\nabla}_t Y - \langle \tilde{\nabla}_t Y,\pd_t\alpha \rangle_f \pd_t\alpha\rangle_f dt,
\end{multline}
for vector fields $X,Y$ on $\alpha_0$, $X(l)=0=Y(l)$. 
\end{definition}
For any such vector field $X$, we can find a family of curves $\alpha_\epsilon$ satisfying the assumption $\pd_\epsilon\alpha(\epsilon, l) \equiv 0$ with $\pd_\epsilon \alpha = X$. The same holds for piecewise smooth vector field with the same initial condition.\\

\begin{proof}[Proof of lemma \ref{sigmadiff}]
The proof depends on the fact that $\gamma$ is an absolute minimum of $L$ among the set of paths $\alpha$ in $M \setminus Crit(f)$ with $\alpha(0) \in W_S$, and contradiction will occur if the differential of $\sigma$ is singular along $\{x_S\} \times [0,b)$. This argument is a modification of the standard one of geodesic beyond conjugate point is never length minimizing.

First, we notice that $d\sigma_{(x_S,t_0)} (0, \dd{t}) = \gamma'(t)$ for a fixed $t_0 \in [0,b)$. We have to compute $d\sigma_{(x_S,t_0)}(v,0)$ for arbitrary $(v,0) \in T_{(x_S,t_0)}(W_0)$. We claim that $\pd_\epsilon\alpha(0,t_0)$ will never be parallel to $\pd_t\alpha(0,t_0)$ for $v\neq0$.\\

Taking a curve $\beta(\epsilon)$ in $U_S$ with $\beta(0)=x_S$ and $\beta'(0)=v$, we can construct a family of arc-length parametrized geodesics $\alpha_\epsilon$ by taking exponential map\[ 
\alpha(\epsilon,t) = \tilde{\exp}_{\beta(\epsilon)}(t\nu).\]
We have $\pd_\epsilon\alpha(0,t) = d\sigma_{(x_S,t)}(v,0)$ with $\pd_\epsilon\alpha$ being a Jacobi field on $\alpha_0$. Suppose the contrary that $\pd_\epsilon\alpha(0,t_0) = c \pd_t \alpha(0,t_0)$ for some constant $c$, then we must have $\tilde{\nabla}_t \pd_\epsilon\alpha(0,t_0) \neq 0$, otherwise we must have $\pd_\epsilon\alpha \equiv c \pd_t\alpha$ which contradicts $v \neq 0 $. \\

We claim that there is a path from $U_S$ to the point $\sigma(v_S,t_0+\delta)$ which gives a smaller value of $L$ comparing to the gradient flow line $\gamma$ from $v_S$ to the point $\sigma(v_S,t_0+\delta)$.  We will denote $l = t_0+\delta$ to fit our previous discussion.\\

%Notice that we have $\langle \tilde{\nabla}_t \pd_\epsilon\alpha, \pd_t\alpha \rangle_f \equiv 0 $, from the fact $\langle \pd_t\alpha ,\pd_t\alpha \rangle \equiv 1$. Therefore we have $\langle \pd_\epsilon\alpha,\pd_t\alpha \rangle_f \equiv c(\epsilon)$ which is a function depends on $\epsilon$ only.  

We construct the path by defining a variational vector field $Y_\eta$ on $\gamma$, depending on a small $\eta>0$ to be fixed. We take a vector field $Z(t)$ such that $Z(0)=0$, $Z(l) = 0$, $\langle Z, \pd_t \rangle_f \equiv 0$ on $[t_0,l]$ and $ Z(t_0) = - \tilde{\nabla}_t \pd_\epsilon(0,t_0)$. We define a piecewise smooth vector field\[
Y_\eta(t) := \left\{ \begin{array}{cc}
\pd_\epsilon\alpha + \eta Z & \textrm{if }t \in [0,t_0], \\ \chi \langle \pd_\epsilon\alpha ,\pd_t\alpha \rangle_f \pd_t\alpha + \eta Z & \textrm{if }t \in [t_0,l],
\end{array} \right.\]
where $\chi$ is a cutoff function on $[t_0,l]$ with $\chi(t_0) = 1$ and $\chi=0$ in a neighborhood of $l$. Notice that $\tilde{\nabla}_t \langle \pd_\epsilon \alpha , \pd_t \alpha \rangle_f  =0 $ from the fact that $| \pd_t \alpha|_f \equiv 1$. A direct computation shows
\[
I(Y_\eta,Y_\eta) = -2\eta | \tilde{\nabla}_t \pd_\epsilon\alpha(0,t_0) |^2_f + 2\eta^2 I(Z,Z).
\]
We have $I(Y_\eta,Y_\eta) <0$ for $\eta$ small enough. \\

By taking the family of curves $\beta_\epsilon$ corresponding to $Y_\eta$, we obtain\[
\left.\frac{d^2L_\beta}{d\epsilon^2}\right|_{\epsilon = 0} < 0,\]
where $L_\beta(\epsilon)=L(\beta(\epsilon))$. For small enough $\epsilon$, $\beta_\epsilon(t)$ will be a curve from $U_S$ to $\sigma(0,l)$ which gives a smaller value of $L$ comparing to $\beta_0= \gamma$. This is impossible because we have\[
L_\beta(\epsilon) \geq f(\sigma(0,l))\]
and the lower bound is attained at $\gamma$.\\

As a conclusion, we can show that $\sigma$ gives a local diffeomorphism onto its image by shrinking $W_0$ if necessary. Therefore it is injective in a contractible neighborhood of the gradient flow line $\gamma$. 
\end{proof}

Under the identification $\sigma$, we use the coordinate $u_1,\dots,u_{n-1}$ for $U_S$ and use $(u_1,\dots,u_{n-1},t)$, or simply $(u,t)$, as coordinates for image of $W_0$ under $\sigma$. By shrinking $W_0$ if necessary, we assume that $W_0$ is a coordinate chart through the map $\sigma$. This justifies definition \ref{psidefinition} of $\psi_E$ being a smooth function on $\sigma(W_0)\subset M$.

%We further required $W_0$ satisfying the property that $\{u\} \times [-t_0,t] \subset W_0$ for any point $(u,t) \in W_0$, with some $t_0$ small enough. This means every point $(u,t)$ in $W_0$ can flow backward for time $t+t_0$ by the vector field $\tilde{\nabla} \psi_E$.

\subsection{Properties of $\psi_E$}
We are going to study the first and second derivatives of $\psi_E$ which is necessary for WKB approximation in the equation (\ref{homotopyeq}). We define\[
V_E:=\sigma((V_S \times (-\delta,b))\cap W_0) \subset \sigma(W_0)\]
as shown in the following picture.\\

%We will use same notations as in previous subsection and define the time $s$ flow $\sigma_s $ given by $\sigma_s (\sigma(u,t)) = \sigma(u,t+s)$. We can assume $\psi_E \geq \tilde{\psi}$ by shrinking $W_0$ if necessary.\\

\begin{figure}[h]
\centering
\includegraphics[scale=0.4]{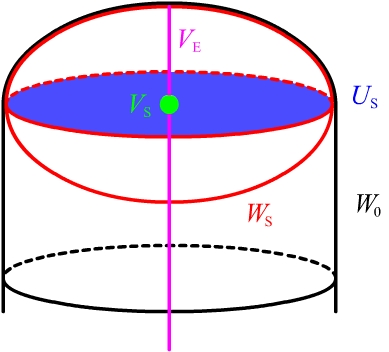}
\end{figure}

\begin{lemma}\label{properties_of_psi}
In $W_0$, we have \[
\tilde{\nabla} \psi_E = d\sigma_* \dd{t}.\]
In particular, we have $\nabla \psi_E = \nabla f$ on $V_E$ and $| \nabla \psi_E| =  | \nabla f|$. 
\end{lemma}
\begin{proof}
We first consider the subset $t \in [0,b)$ in $W_0$. Let $\beta(\epsilon)$ be a curve in $U_S$ such that $\beta(0)=u$ and\[
\alpha(\epsilon,t)=\exp_{\beta(\epsilon)}(t\nu) = \sigma(\beta(\epsilon),t).\]
Notice that we have $\psi_E(\alpha_\epsilon(t)) = L(\alpha_\epsilon|_{[0,t]})$. Applying the first variation formula, we have
\begin{eqnarray*}
\langle \tilde{\nabla}\psi_E( \alpha_\epsilon(t)),\pd_\epsilon \alpha_\epsilon(t)\rangle_f\big|_{\epsilon=0} &=& \left.\frac{dL}{d\epsilon}\right|_{\epsilon=0}\\
&=& \langle \pd_t \alpha(0,t),\pd_\epsilon \alpha(0,t)\rangle_f.
\end{eqnarray*}
As $\pd_\epsilon \alpha(0,t)$ can be chosen arbitrarily, we get\[
\tilde{\nabla}\psi_E(u,t) = \pd_t\alpha(0,t) = d\sigma_{(u,t)}\dd{t}.\]
The same argument works for $t\in (-\delta,0]$ by taking 
\[
L(\alpha_\epsilon|_{[t,0]})=\psi_S(\alpha_\epsilon(0))+\int_t^0|\pd_t \alpha_\epsilon|_f dt.\]

Furthermore, we have $|\tilde{\nabla}\psi_E(u,t)|^2_f=|d\sigma_{(u,t)}\dd{t}|^2_f=1$ which gives $|\nabla\psi_E(u,t)|=|\nabla f|$. Finally, as we know $\nabla \psi_S = \nabla f$ on $V_S$ and flow lines of $\nabla f$ are geodesics after reparametrizations, we get $\nabla \psi_E = \nabla f$ on $V_E$.
\end{proof}

We now consider the second derivatives of $g_E = \psi_E-f$.
\begin{lemma}\label{psi_f}
By choosing a small enough $W_0$, we have
\begin{enumerate}
\item $g_E \geq 0$ and
\item $g_E$ is a Bott-Morse function with critical set $V_E = \{ g_E = 0\}$.
\end{enumerate}
\end{lemma}

\begin{proof}
The previous lemma implies that $\nabla g_E = 0$ on $V_E$. We are going to show $\nabla^2 g_E$ is positive definite in the normal bundle of $V_E$. Fixing any $t \in [0,b)$, we consider the submanifold $U_t = \sigma (U_S \times \{t\} \cap W_0)$. There is an isomorphism between the normal bundle of $V_t = \sigma(V_S \times \{t\} \cap W_0)$ in $U_t$ and the normal bundle of $V_E$ in $W_0$. Therefore we restrict $g_E$ on $U_t$ and consider its Hessian. 

We abuse the notations and write $ u : W_0 \rightarrow U_S$ as the projection map. We take $h = g_E - g_S \circ u$, then $h \geq 0 $ on $U_t$ by definition of $\psi_E$ and $\nabla h = 0 = h$ on $V_t$. Therefore we have $h$ is positive semi-definite on the normal bundle of $V_t$ in $U_t$. Moreover, we have 
$\nabla^2 (g_S \circ u) =( \nabla^2 g_S ) \circ u$ on $V_S$
being positive definite in the normal bundle. 

By choosing sufficiently small $\delta$, we can assume that $\nabla^2 g_E >0$ along $V_E$ and hence the result follows.
\end{proof}

Next, we consider the second order derivatives for $\Psi= \psi_E - \psi_S = g_E - g_S$ defined on $W_S$.

\begin{lemma}\label{psi_f_g}
By choosing small enough neighborhood $W_S$ of $v_S$ if necessary, we have
\begin{enumerate}
\item $\Psi \leq 0$ on $W_S$ and
\item $\Psi$ is a Bott-Morse function with critical set $U_S = \{ \Psi = 0\}\subset W_S$.
\end{enumerate}
\end{lemma}

\begin{proof}
We first have $\nabla \Psi = 0$ on $U_S$ because $\nabla \psi_E = \nabla \psi_S$ on $U_S$. If we consider $\nabla^2 \Psi (\dd{t},\dd{t})$ on $V_S$, then we have $\nabla^2 g_E (\dd{t},\dd{t}) =0$ and $\nabla^2 g_S ( \dd{t},\dd{t}) >0$. Therefore, there exists an neighborhood $U$ of $V_S$ in $U_S$ so that\[
\Hess\Psi(x)(\dd{t},\dd{t}) < 0\]
for all $x\in U$. Choosing $W_S$ small enough will achieve the desired result.
\end{proof}
\begin{remark}
We can extend the function $\Psi$ from $W_S$ to $W_0$ to be a non-negative function with critical set $U_S$ which is also an absolute maximum. This is for our convenience in later arguments. 
\end{remark}

% !TEX root = paper.tex

\subsection{The WKB iteration}\label{WKB_method}
After knowing these properties of $\psi_E$, we will describe the iteration procedure to define $\omega_{E,i}$ inductively. 

First, by lemma \ref{properties_of_psi}, we have $|df|^2=|d\psi_E|^2$ and hence the expansion
\begin{eqnarray*}
e^{\lam \psi^{}_{\scalebox{.7}{$\scriptscriptstyle E$}}} \Delta_f e^{-\lam \psi^{}_{\scalebox{.7}{$\scriptscriptstyle E$}}} &=& \Delta + \lam M_f+\lam(\lieg{\psi^{}_{\scalebox{.7}{$\scriptscriptstyle E$}}}-\lieg{\psi^{}_{\scalebox{.7}{$\scriptscriptstyle E$}}}^*)\\
&=& \Delta + \lam(2\lieg{\psi^{}_{\scalebox{.7}{$\scriptscriptstyle E$}}}-M_{g^{}_{\scalebox{.7}{$\scriptscriptstyle E$}}}),
\end{eqnarray*}
where $M_{g^{}_{\scalebox{.7}{$\scriptscriptstyle E$}}}=\lieg{g^{}_{\scalebox{.7}{$\scriptscriptstyle E$}}}+\lieg{g^{}_{\scalebox{.7}{$\scriptscriptstyle E$}}}^*$. Following \cite{HelSj4}, we let
\begin{equation*}
\ot=2\lieg{\psi^{}_{\scalebox{.7}{$\scriptscriptstyle E$}}}-M_{g^{}_{\scalebox{.7}{$\scriptscriptstyle E$}}},
\end{equation*}
and consider the following equation
\begin{equation*}
(\Delta + \ot \lam )(\mu_0(\lam)+\mu_1(\lam)+\cdots) = e^{\lam\Psi}\nu,
\end{equation*}
order by order in $\lam$ where $\mu_i(\lam)$ is a function (depending on $\lam$). We often write $\mu_i$ to simplify our notations. The first equation to be solved is
\begin{equation}\label{1st_eq}
\lam \ot \mu_0(\lam) = e^{\lam \Psi}\nu.
\end{equation}

In order to solve the above equation involving $\mathcal{L}_{\nabla \psi^{}_{\scalebox{.7}{$\scriptscriptstyle E$}}}$, we need a map $\tau$ describing the flow of $\nabla \psi_E$. It is given by renormalising $\sigma$ such that $d\tau_*(\dd{t}) = \nabla \psi_E$ and is of the form
\begin{equation}\label{taumap}
\tau : W \subset U_S \times (-\infty,+\infty)\rightarrow M,
\end{equation}
with the same image as $\sigma$. We can also assume that $W \cap \{u\}\times \real$ is a connected open interval.

\begin{notation}\label{localcoordinates}
We use $(u_1,\dots,u_{n-1},t)$ as coordinates of $\tau(W)$ from now on. For simplicity, we also let $u_n = t$ and $\grave{u} = (u_1,\dots,u_{n-1})$.
\end{notation}

For the iteration process, we focus on
\begin{equation*}
\Omega_0^*(W) = \{ \beta \in \Omega^*(W)|\;\;\overline{supp(\beta) }\cap (U_S \times (-\infty ,t_0])\;compact\;for\;all\;t_0 \},
\end{equation*}
for the definition of the following integral operator.
\begin{definition} We let
$I:\Omega^*_{0}(W)\rightarrow\Omega^*_{0}(W)$ given by 
\begin{equation}\label{Ioperator}
I(\phi):= \int_{-\infty}^0 e^{\int_s^0 \frac{1}{2}\tau_{\epsilon}^*(M_{g_E})\; d\epsilon}\; \tau_s^*(\phi) ds,
\end{equation}
where $\tau_s(u,t) = \tau(u,t+s)$ is the flow of $\nabla \psi_E$ for time $s$.
\end{definition}

To solve \eqref{1st_eq}, we put
\begin{equation}\label{omega_0}
\mu_0= \frac{1}{2\lam} I(e^{\lam \Psi}\nu).
\end{equation}
Then it can be checked that $\mu_0$ is the solution to \eqref{1st_eq}. The second equation to be solved is 
\begin{equation}
\lam \ot\mu_1 = - \Delta\mu_0.
\end{equation}
Again, we put\[
\mu_1= -\frac{1}{2\lam} I(\Delta\mu_0).\]
In general, we have the transport equation  for $l\geq0$
\begin{equation}\label{trans_eq}
\ot\mu_{l+1} = -\lam^{-1}\Delta\mu_l.
\end{equation}
This gives
\begin{equation}\label{wkbiteration}
\mu_{l+1}= -\frac{1}{2\lam} I(\Delta\mu_l).
\end{equation}
as solutions in $W$.

\subsection{Estimate of the WKB iteration}\label{estimate}
In this section, we are going to obtain norm estimates for $\mu_l$'s.  We consider terms appearing in the iteration which are essentially of the form
\begin{equation}\label{gen_term}
I^j\Big(e^{\lam \Psi}(\prod_{\alpha}\nabla_{\alpha}\Psi)\beta\Big)
\end{equation}
with $j\geq 0$ and $\beta\in\Omega^*_0(W)$, where $I^j$ is the composition of $I$ for $j$ times. Here each $\alpha=(\alpha_{1},\ldots,\alpha_{n})$ is a multi-index such that\[
\nabla_{\alpha}\Psi=\nabla_{\dd{u_1}}^{\alpha_1} \cdots \nabla_{\dd{u_{n-1}}}^{\alpha_{n-1}}\nabla_{\dd{u_n}}^{\alpha_n} \Psi.\]
With
\[m(\alpha):=\max \{ 0,2-\alpha_n \},\]
we have
\begin{equation}\label{order_at_zero}
\nabla^j (\prod_\alpha \nabla_{\alpha}\Psi )|_{U_S} \equiv 0,
\end{equation}
for $j < \sum_\alpha m(\alpha)$ from lemma \ref{psi_f_g}. 

\begin{remark}
Different choices of order of taking differentiation in definition of $\nabla_{\alpha}$ will result in a difference involving the curvature of $(M,g)$, however, the order of vanishing in equation \eqref{order_at_zero} remains unchanged and hence the following estimates hold for any such choice.
\end{remark}

The counting of vanishing order along $U_S$ is needed for applying the following semi-classical approximation lemma \ref{stat_phase_exp}, appearing in \cite{DiSj}.
\begin{lemma}\label{stat_phase_exp}
Let $U\subset \real^n$ be an open neighborhood of $0$ with coordinates $x_1,\ldots,x_n$. Let $\varphi:U\rightarrow\real_{\geq0}$ be a Morse function with unique minimum $\varphi(0)=0$ in $U$. Let $\tilde{x}_1,\ldots,\tilde{x}_n$ be a Morse coordinates near $0$ such that\[
\varphi(x)=\frac{1}{2}(\tilde{x}_1^2+\cdots+\tilde{x}_n^2).\]
For every compact subset $K\subset U$, there exists a constant $C=C_{K,N}$ such that for every $u\in C^{\infty}(U)$ with $\supp(u) \subset K$, we have
\begin{eqnarray}\label{stat_ineq}
&& \left|(\int_K e^{-\lam \varphi(x)} u ) - (\frac{\lam}{2\pi})^{n/2}\Big(\sum_{k=0}^{N-1} \frac{ \lam^{-k}}{2^kk!} \tilde{\Delta}^k(\frac{u}{\Im})(0) \Big)\right|\nonumber\\
&\leq& C \lam^{-n/2-N} \sum_{|\alpha|\leq 2N+n+1} \sup|\pd^{\alpha}u|,
\end{eqnarray}
where\[
\tilde{\Delta}=\sum\ppd{\tilde{x}_j},\qquad\Im=\pm\det(\frac{d\tilde{x}}{dx}),\]
and $\Im(0)=(\det\Hess\varphi(0))^{1/2}$.

In particular, if $u$ vanishes at $0$ up to order $L$, then we can take $N=\lceil L/2 \rceil$ and get\[
 \left| \int_K e^{-\lam \varphi(x)} u \right| \leq C \lam^{-n/2-\lceil L/2 \rceil}.\]
\end{lemma}
From the above, we obtain the following lemma.

\begin{lemma}\label{gmnormestimate}
Let $L_{\grave{u}}$ be the line interval along $t$ direction with fixed $\grave{u}$ coordinates, we have the norm estimate
\begin{equation*}
\left( \int_{L_{\grave{u}}} | \nabla_\alpha  (e^{\lam \Psi})|^{2^k} \right)^\frac{1}{2^k} \leq  C_{\alpha,k} \lam^{\frac{\alpha_n}{2} - \frac{1}{2^{k+1}}}
\end{equation*}
for any multi-index $\alpha$ and $k\in \inte_{\geq 0 }$.
\end{lemma} 

Motivated by the above lemma, we consider a filtration
$$
\dots \subset F^{-s}\subset \dots F^{-1}\subset F^0 \subset F^1 \subset F^2\subset \dots \subset F^s \subset \dots \subset  \Omega^*_0(W)
$$
of the space of differential forms on $\Omega^*_0(W)$ which is defined as follows. 

\begin{definition}\label{filtrationdefinition}
$\phi \in \Omega^*_0(W)$ is in $F^s$ if for any compact subset $K \subset W$ and integers $j,k\in \inte_+$, we have
$$
\| \nabla_\alpha \phi \|_{L^{2^k}(K\cap L)} \leq C_{\alpha,k,K} \lam^{\frac{\alpha_n+s}{2} - \frac{1}{2^{k+1}}},
$$
for any line $L=L_{\grave{u}}$.
\end{definition}

The Lemma \ref{gmnormestimate} simply means $e^{\lam \Psi} \in F^0$.

\begin{prop}
We have $\nabla F^s \subset F^{s+1}$ and $F^s \cdot F^r \subset F^{r+s}$, where $\cdot$ denotes the wedge product of forms.
\end{prop}
\begin{proof}
The first property is trivial. For the relation $F^s \cdot F^r \subset F^{r+s}$, we fix $j \in \inte_+$ and a compact subset $K$. For $\phi \in F^r$ and $\psi \in F^s$, we first observe that
$$
\nabla_\alpha (\phi \wedge \psi) = \sum_{\beta + \theta = \alpha }(\nabla_{\beta} \phi) \wedge (\nabla_{\theta} \psi).
$$
Then the H\"older inequality implies that
\begin{eqnarray*}
\| (\nabla
_{\beta} \phi) \wedge (\nabla_\theta \psi) \|_{L^{2^k}(K \cap L)} &\leq& C \| \nabla_\beta \phi \|_{L^{2^{k+1}}(K \cap L)} \| \nabla_\theta \psi \|_{L^{2^{k+1}}(K \cap L)}\\
& \leq & C\lam^{\frac{\beta_n+s}{2} -\frac{1}{2^{k+2}}} \cdot \lam^{\frac{\theta_n+r}{2} - \frac{1}{2^{k+2}}}\\
& \leq & C \lam^{\frac{\alpha_n+r+s}{2} - \frac{1}{2^{k+1}}}
\end{eqnarray*}
and the result follows.
\end{proof}

\begin{lemma}\label{filtrationlemma}
For $\phi \in F^{s}$, we have
\begin{eqnarray*}
I(\phi) & \in & F^s,\\
 \Delta I(\phi)& \in &F^{s+1}.
\end{eqnarray*}
\end{lemma}

\begin{proof}
To simplify the notations, we only prove the statement for functions as we can fix a basis (independent of $\lam$) for differential forms in $W$, and estimate the coefficient functions. The Christoffel symbols appearing in differentiating the basis will be independent of $\lam$ and not affecting the following estimates. For the same reason, let us simply pick a flat metric in $u_i$'s coordinates for simplicity. In that case, we can write $\Delta = \sum_i \nabla_i^2$.

We first consider the operator $\nabla_n^2$, and we will have $2\nabla_n I(\phi) = M_{g^{}_{\scalebox{.7}{$\scriptscriptstyle E$}}}\phi$ where $M_{g^{}_{\scalebox{.7}{$\scriptscriptstyle E$}}}$ is acting as scalar multiplication by function. Therefore we have 
$$
\|\nabla_{\alpha} (\nabla_n^2 I(\phi))\|_{L^{2^k}(K \cap L)} = \| \nabla_\alpha \nabla_n(M_{g^{}_{\scalebox{.7}{$\scriptscriptstyle E$}}} \phi) \|_{L^{2^k}(K \cap L)} \leq C_{\alpha,k,K} \lam^{\frac{\alpha_n+s+1}{2}-\frac{1}{2^{k+1}}}.
$$
This implies $(\nabla_n^2 I(\phi)) \in F^{s+1}$. 

Next, we consider the operator $\nabla_i^2$ for $i<n$. Fixing a multi-index $\alpha$ and using the result $I(\phi) \in F^s$, we have
$$
\|\nabla_\alpha \nabla_i^2 (I \phi)\|_{L^{2^k}(K \cap L)} \leq C_{\alpha,k,K} \lam^{\frac{\alpha_n+s}{2}-\frac{1}{2^{k+1}}},
$$
which gives $\nabla_i^2 (I \phi) \in F^s \subset F^{s+1}$.\\

It remains to show that $I(\phi) \in F^s$ which requires estimates of the term $\nabla_\alpha I(\phi)$. There are two cases to be considered, $\alpha_n \neq 0$ and $\alpha_n = 0$. If $\alpha_n \neq 0$, we can cancel the integral operator with one of the $\nabla_n$, which gives
$$
\|\nabla_\alpha I(\phi)\|_{L^{2^k}(K \cap L)} = \half \|\nabla_{\hat{\alpha}} (M_{g^{}_{\scalebox{.7}{$\scriptscriptstyle E$}}} \phi)\|_{L^{2^k}(K \cap L)} \leq  C_{\alpha,k,K} \lam^{\frac{\alpha_n+s-1}{2}-\frac{1}{2^{k+1}}},
$$
where $\hat{\alpha}$ refers to the multi-index by letting $\hat{\alpha}_n = \alpha_n -1$.

If $\alpha_n = 0$, we can commute all the $\nabla_{\alpha}$ with the integral operator $I$. We let $Q(\grave{u},t,s)=e^{\int_s^0 \frac{1}{2}\tau_{\epsilon}^*(M_{g^{}_{\scalebox{.7}{$\scriptscriptstyle E$}}})\; d\epsilon}$ as a function and write $I(\phi)(\grave{u},t)= \int_{-\infty}^0 Q(\grave{u},t,s) \phi(\grave{u},t+s) ds$. Therefore we have 
$$
\nabla_\alpha (I(\phi)) = \sum_{\beta + \theta = \alpha} \int_{-\infty}^0 \nabla_{\beta}(Q(\grave{u},t,s)) \nabla_{\theta}\phi(\grave{u},t+s) ds,
$$
and 
$$
\|\nabla_\alpha (I(\phi))\|_{L^{2^k}(K \cap L)} \leq C_{\alpha,k,K} \sum_{\theta \subset \alpha}|\int_{-\infty}^0 \nabla_{\theta}\phi(\grave{u},t+s) ds| \leq C_{\alpha,k,K} \lam^{\frac{\alpha_n+s}{2}-\half}.
$$
Combining the two cases will give $I(\phi) \in F^s$.
\end{proof}

\begin{remark}\label{filtrationremark}
Using the above lemma, we can show that the $\mu_l(\lam)$'s appearing in the iteration equation \eqref{wkbiteration} will satisfy $\mu_l(\lam) \in F^{-l-2}$. In particular, we can get an explicit estimate as
$$
\| \nabla^j \mu_l(\lam)\|_{L^2(K)} \leq C_{j,K} \lam^{\frac{j-l-2}{2}-\frac{1}{4}} ,
$$
for all $j$ and compact subset $K \subset W$.
\end{remark}

\subsection{A priori estimate}\label{L2approx}
We make use of the WKB iteration to construct the WKB expansion and prove that it does give a desired approximate as in theorem \ref{thm:homotopy_wkb} to the solution in section \ref{L2approx} and section \ref{sec:WKB_approximation}. This is a standard technique which is taken from \cite{HelSj1} (readers may also see \cite[Chapter 4]{helffer2006semi}), with slight modification in the current case. To begin with, we obtain an a priori estimate for the solution in this subsection.\\

We consider the equation
\begin{equation}\label{wkbequation}
\Delta_f \zeta_E = (I-P_f)d_f^*(\chi_S \zeta_S)
\end{equation}
 in $W$, where $\zeta_S \in \Omega^*(W_S)$ is the input form depending on $\lam$ and $\chi_S \in C^{\infty}_c(W_S)$ is some cutoff function to be chosen later. We assume $\zeta_S$ has a WKB approximation on $W_S$ of the form
\begin{equation}
\zeta_S \sim e^{-\lam \psi^{}_{\scalebox{.7}{$\scriptscriptstyle S$}}}(\omega_{S,0}+\omega_{S,1} \lam^{-1/2}+ \omega_{S,2} \lam^{-1} + \dots),
\end{equation}
where $\omega_{S,i} \in \Omega^*(W_S)$ and $\psi_S = f+g_S$. It is an approximation in the sense that
\begin{equation}
\|e^{\lam \psi^{}_{\scalebox{.7}{$\scriptscriptstyle S$}}} \zeta_S - (\sum_{i=0}^{N} \omega_{S,i} \lam^{-i/2}) \|^2_{L^\infty(W_S)} \leq C_N \lam^{-N-1}
\end{equation}
for $N$ large enough, where $C_N$ is a constant depending on $N$. We also require similar norm estimates for its derivatives
\begin{equation}
\|  e^{\lam\psi^{}_{\scalebox{.7}{$\scriptscriptstyle S$}}}\nabla^j ( \zeta_S - e^{-\lam \psi^{}_{\scalebox{.7}{$\scriptscriptstyle S$}}} (\sum_{i=0}^{N} \omega_{S,i} \lam^{-i/2})) \|^2_{L^\infty(W_S)} \leq C_{j,N} \lam^{-N-1+2j},
\end{equation}
with $C_{j,N}$ depending on $j, N$.\\

We want to get a similar expansion for $\zeta_E$, using the iteration defined in the section \ref{WKB_method}. We consider any small enough compact neighborhood $K\subset W$ of the flow line $\gamma$ with $\chi \equiv 1$ on $K$. $\chi_S$ is chosen so that $\supp(\chi_S) \subset K$. The following figure illustrates the situation.\\
\begin{figure}[h]
\centering
\includegraphics[scale=0.5]{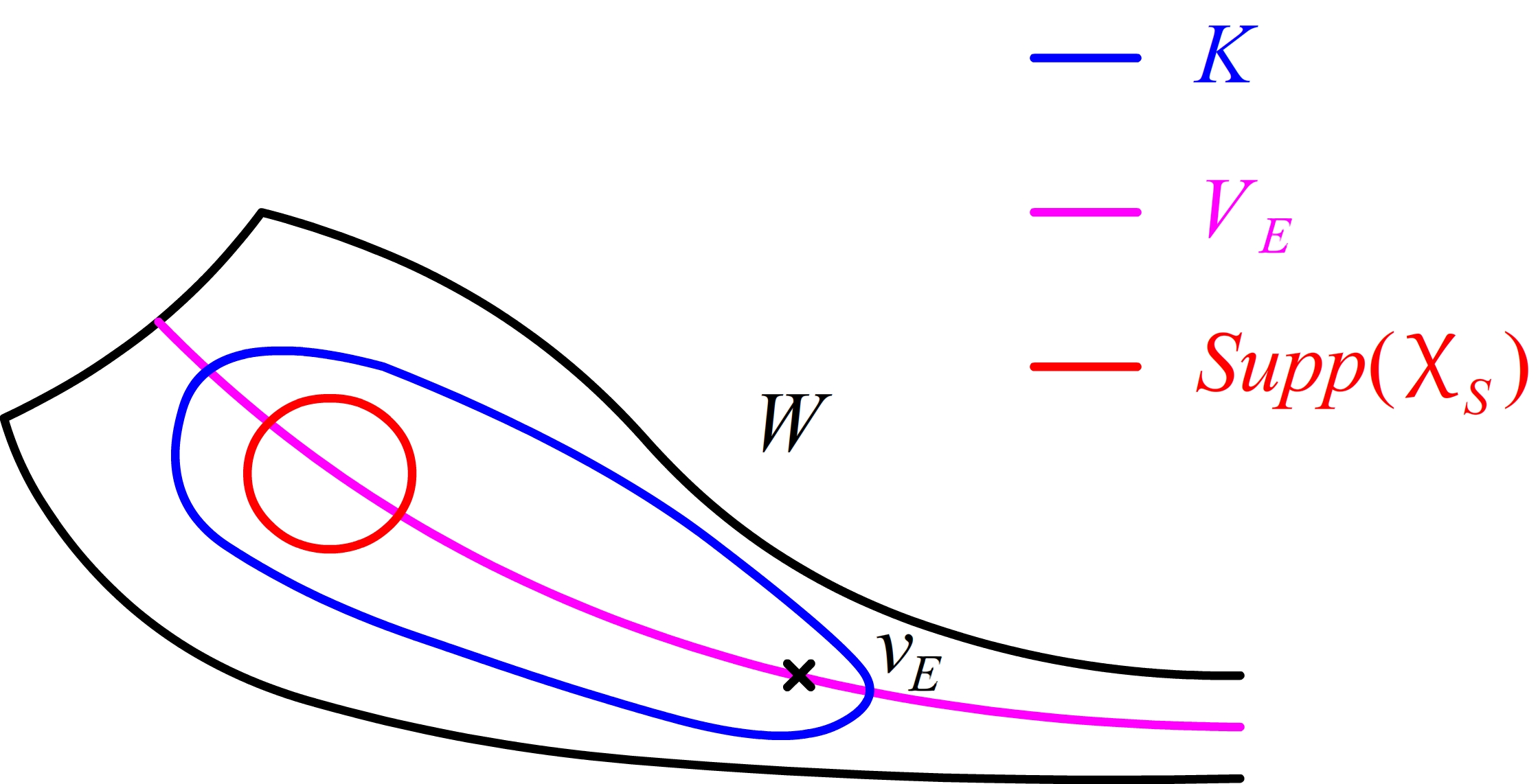}
\end{figure}

\noindent If $K$ is small enough, we have an a priori estimate of $\zeta_E$ in $K$ as lemma \ref{wkbapriori}, which is essentially the result of \cite[Proposition 5.5]{HelSj1} with modification to suit our current situation. 
\begin{lemma} \label{wkbapriori}
For small enough $\supp (\chi_S)$ and $K$, and any $j \in \inte_+$, there exists $\lam_{j,0}>0$ such that for any $\lam > \lam_{j,0}$, we have
\begin{equation}
\| e^{\lam \psi^{}_{\scalebox{.7}{$\scriptscriptstyle E$}}} \nabla^j \zeta_E\|^2_{L^\infty(K)} \leq C_{j} \lam^{N_j},
\end{equation}
where $N_j$ is an positive integer depending on $j$.
\end{lemma}

In order to prove the above lemma, we need to know certain properties of $\chi$ and the chosen compact set $K$. Let $\tilde{\psi}:= \inf_{y \in \supp(\chi_S)} \{ \psi_S + \dist_f(y,x)\}$, we have the following lemma playing the role of \cite[Lemma 5.7]{HelSj1}.

\begin{lemma}\label{convexlemma}
There exists $\epsilon>0$ such that for all sufficiently small $K$, we have 
\begin{equation}
\tilde{\psi}(x) + \dist(y,x) \geq \psi_E(y) +\epsilon,
\end{equation}
for all $y\in K$ and $x \in \supp(\nabla \chi)$.
\end{lemma}
\begin{proof}
Using the fact that $\psi_E = f$ on $V_E$ and choosing $K$ small enough such that $| \psi_E - f | \leq \epsilon$ on $K$, we can simply prove 
$$
\tilde{\psi}(x) +\dist(y,x) \geq f(y) + \epsilon ,
$$
by choosing small enough $K$ and $\epsilon$. From the properties of Agmon distance $\dist$, we have 
\begin{equation*}
\tilde{\psi}(x) \geq \min_{z \in \supp(\chi_S)}( f(z)  + f(x)- f(z)) = f(x),
\end{equation*}
with equality holds only if $z \in V_S$ and there is a generalized gradient line joining $z$ to $x$. Therefore, we have 
\begin{equation*}
\tilde{\psi}(x) +\dist(y,x) \geq f(x) + f(y) -f(x) = f(y),
\end{equation*}
with equality holds only if there is a generalized gradient line joining a point $z \in V_S$ to $x \in \supp(\chi)$ and then to $y \in K$. This is impossible by for our choices of $\chi$ and $K$. Hence we always have strict inequality and therefore we can find small $\epsilon$ by compactness argument. 
\end{proof}

We consider a closed neighborhood $\tilde{W}$ of $\supp(\chi)$ in $W$ with smooth boundary. We let $\tilde{G}$ to be the twisted Green's operator on $\tilde{W}$ using Dirichlet boundary condition. We first argue that $\zeta_E$ can be replaced by $\tilde{\zeta}_E = d^*_f\tilde{G} \chi_S \zeta_S$.
\begin{lemma}\label{locality}
There exists $\delta>0$ such that
\begin{equation*}
\| e^{\lam \psi^{}_{\scalebox{.7}{$\scriptscriptstyle E$}}}\nabla^j (\chi \zeta_E - \tilde{\zeta}_E )\|_{L^{\infty}(K)} \leq C_j e^{-\lam \delta},
\end{equation*}
for all $j \in \inte_+$ whenever $\supp(\chi_S)$ and $K$ are chosen to be small enough. 
\end{lemma}
\begin{proof}
We let $r_\lam = \chi \zeta_E - \tilde{\zeta}_E $. First, $r_\lam$ satisfies the equation
\begin{equation}
\tilde{\Delta}_f r_\lam =  [\Delta,\chi] \zeta_E  - \chi P_f d_f^*(\chi_S \zeta_S) .
\end{equation}
Therefore we have $r_\lam = (\tilde{G} [\Delta,\chi] G  -  \tilde{G}\chi P_f )d_f^*(\chi_S \zeta_S)$. We consider it term by term to get estimate of $r_\lam$. Making use of lemma \ref{resolventlemma} and a similar statement for $\tilde{G}$, we have for any $\epsilon >0$,
$$
\tilde{G} [\Delta,\chi] G \sim \mathcal{O}_\epsilon\Big(\exp(-\lam(\min_{z\in \supp(\nabla \chi)}(\dist(x,z)+\dist(z,y)-\epsilon)))\Big).
$$
Using lemma \ref{convexlemma}, we can show there exists $\delta_0>0$ such that
$$
\tilde{G} [\Delta,\chi] G  d_f^*(\chi_S \zeta_S) \sim \mathcal{O}(e^{-\lam(\psi^{}_{\scalebox{.7}{$\scriptscriptstyle E$}}+\delta_0)})
$$
in $K$ when $\lam$ is small enough.

For the term $\tilde{G}\chi P_f $, we have 
$$
\tilde{G}\chi P_f  \sim \mathcal{O}_\epsilon\Big(\sum_{q\in C_f^{l}}\exp(-\lam(\dist(x,q)+\dist(q,y)-\epsilon))\Big)
$$
follows from lemma \ref{eigenestimate2} and modified version of lemma \ref{resolventlemma} for $\tilde{G}$, where $l=deg(\zeta_S)$. Again, we can find a constant $\delta_1>0$ such that 
$$
\min_{x \in \supp(\chi_S)}(\psi_S(x)+\dist(x,q)+\dist(q,y)) \geq \psi_E(y)+2\delta_1,
$$
for $y \in K$.
Similarly we have 
$$
\tilde{G}\chi P_f d_f^*(\chi_S \zeta_S) \sim \mathcal{O}(e^{-\lam(\psi^{}_{\scalebox{.7}{$\scriptscriptstyle E$}}+\delta_1)})
$$
in $K$ when $\lam$ is large enough. Notice that the constant $\delta = \min\{\delta_0,\delta_1\}$ can chosen to be the same if we shrink $\supp(\chi_S)$ and $K$ and keep $\tilde{W}$ and $\chi$ fixed.
\end{proof}
Next, we obtain estimates for $\tilde{\zeta}_E$ similar to those in lemma \ref{wkbapriori} for $\zeta_E$ using the argument as in \cite[Proposition 5.5]{HelSj1}. 
\begin{lemma}\label{tildewkbapriori}
For any $j \in \inte_+$, there exists $\lam_{j,0}>0$ such that if $\lam > \lam_{j,0}$, we have
\begin{equation}
\| e^{\lam\psi^{}_{\scalebox{.7}{$\scriptscriptstyle E$}}} \nabla^j \tilde{\zeta}_E\|^2_{L^\infty(\tilde{W})} \leq C_{j} \lam^{N_j},
\end{equation}
where $N_j$ is an positive integer depending on $j$.
\end{lemma}
\begin{proof}
We consider the equation
\begin{equation}\label{tildewkbequation}
\Delta_f \tilde{\zeta}_E  =d_f^*( \chi_S \zeta_S)
\end{equation}
with Dirichlet boundary condition in $\tilde{W}$, and divide the proof into three steps:

\textbf{Step 1:}
Without loss of generality, we assume there is a constant $C_0>0$ such that $C_0^{-1} \leq \psi_E \leq C_0$ and $C_0^{-1} \leq | df|^2=|d\psi_E|^2 \leq C_0$ on $\tilde{W}$. We define the function 

\begin{equation}\label{bigphi_definition}
\Phi = 
\psi_E -\frac{C}{\lam} \log(\lam\psi_E),
\end{equation}
with $C>0$ to be chosen. Therefore we have 
$$
|df|^2- |d \Phi|^2\geq \frac{C |df|^2}{\lam \psi_E}\geq \frac{C}{C_0^2 \lam}.
$$
Using the equation \eqref{tildewkbequation} we get
\begin{eqnarray*}
Re(\langle e^{2\lam\Phi} d_f^*(\chi_S \zeta_S),\tilde{\zeta}_E   \rangle)& =&  ( \| d(e^{\lam\Phi} \tilde{\zeta}_E )\|^2+\| d^*(e^{\lam\Phi}\tilde{\zeta}_E  )\|^2)\\
 &&+ \langle (\lam^2(|df|^2 - |d \Phi|^2) + \lam M_f) e^{\lam \Phi} \tilde{\zeta}_E ,e^{\lam\Phi} \tilde{\zeta}_E  \rangle
\end{eqnarray*}
and if we choose a large $C>0$ to absorb the term $\langle \lam M_f e^{\lam\Phi}\tilde{\zeta}_E, e^{\lam\Phi }\tilde{\zeta}_E \rangle $, we have
\begin{eqnarray*}
&& ( \| d(e^{\lam\Phi} \tilde{\zeta}_E)\|^2+\| d^*(e^{\lam\Phi} \tilde{\zeta}_E)\|^2)+ \frac{C\lam}{2C_0^2} \| e^{\lam\Phi} \tilde{\zeta}_E \|^2\\
&\leq&  C_1 \| e^{\lam\Phi} d_f^*(\chi_S \zeta_S)\|^2 \leq C_1 (\frac{C_0}{\lam})^{2C} \| e^{\lam\psi^{}_{\scalebox{.7}{$\scriptscriptstyle E$}}}d_f^*( \chi_S \zeta_S )\|^2 \\
&\leq &C_2 (\frac{C_0}{\lam})^{2C}\| e^{\lam\psi^{}_{\scalebox{.7}{$\scriptscriptstyle S$}}}d_f^*( \chi_S \zeta_S )\|^2 \leq C_3 \lam^{2-2C} .
\end{eqnarray*}
Therefore we get 
$$
 ( \| d (e^{\lam\psi^{}_{\scalebox{.7}{$\scriptscriptstyle E$}}} \tilde{\zeta}_E)\|^2+\| d^*(e^{\lam\psi^{}_{\scalebox{.7}{$\scriptscriptstyle E$}}} \tilde{\zeta}_E)\|^2) + \lam \| e^{\lam\psi^{}_{\scalebox{.7}{$\scriptscriptstyle E$}}} \tilde{\zeta}_E\|^2 \leq C_3,
$$
and so $\|e^{\lam\psi^{}_{\scalebox{.7}{$\scriptscriptstyle E$}}} \tilde{\zeta}_E\|^2_{L^2(K)} \leq C_4 \lam^{-1}$, for $\lam < \lam_0$.

\textbf{Step 2:} We prove the $L^2$ estimate for derivatives of $\tilde{\zeta}_E$.
We apply $d_f$ and $d^*_f$ to both sides of equation \eqref{tildewkbequation}. We obtain
\begin{equation}
\Delta_f (d_f \tilde{\zeta}_E) = d_fd_f^* (\chi_S \zeta_S).
\end{equation}
Applying the result in step $1$ to $d_f \tilde{\zeta}_E$, we have 
$$
\| e^{\lam\psi^{}_{\scalebox{.7}{$\scriptscriptstyle E$}}} d_f \tilde{\zeta}_E \|^2_{L^2(K)} \leq C_4 \lam^{-1}. 
$$
Since $d_f =  d +  \lam df \wedge$, we have 
$$
\| e^{\lam \psi^{}_{\scalebox{.7}{$\scriptscriptstyle E$}}} d \tilde{\zeta}_E \| ^2_{L^2(K)}\leq C_5 \lam^{1}.
$$

Corresponding result for $d^* \tilde{\zeta}_E$ can be obtained by a similar argument. These combine together to obtain an estimate for $\nabla \tilde{\zeta}_E$. By applying $\nabla$ successively, we obtain all higher derivatives' estimates in a similar fashion. 

\textbf{Step 3:} Finally, we improve the estimate to $L^{\infty}$ norm. Since we have $L^2$ norm estimate for all the derivatives of $\tilde{\zeta}_E$. We use the Sobolev embedding on $\tilde{W}$ to obtain the $L^{\infty}$ norm estimate. Details are left to readers.
\end{proof}
Lemma \ref{wkbapriori} follows from lemma \ref{locality} and lemma \ref{tildewkbapriori} directly.

% !TEX root = paper.tex
\subsection{WKB approximation}\label{sec:WKB_approximation}
Next, we consider the WKB approximation of $\zeta_E$. From the WKB approximation \eqref{awkbexpansion} of $\zeta_S$, we can take $d_f^*$ on both side and obtain a WKB approximation of $d_f^* (\chi_S\zeta_S)$
\begin{equation}\label{dstarawkbexpansion}
d_f^* (\chi_S\zeta_S) \sim e^{-\lam \psi^{}_{\scalebox{.7}{$\scriptscriptstyle S$}}} ( d^* + \lam (\iota_{\nabla f}+\iota_{\nabla \psi_S}))(\chi_S\omega_{S,0} + \chi_S\omega_{S,1} \lam^{-1/2} + \dots),
\end{equation}
after grouping terms according to their orders of $\lam$. We apply the iteration in the previous subsection \ref{WKB_method} terms by terms to the above series and then group the terms according to orders of $\lam$ of their $L^2$ norms. As a result, we obtain a WKB expansion
\begin{equation}\label{bwkbexpansion2}
\zeta_E \sim e^{-\lam \psi^{}_{\scalebox{.7}{$\scriptscriptstyle E$}}}(\omega_{E,0}(\lam) + \omega_{E,1}(\lam) + \dots)
\end{equation}
in $W$, where $\omega_{E,i}(\lam)$'s are functions also depending on $\lam$. Using lemma \ref{filtrationlemma} and remark \ref{filtrationremark}, we know that for every $l$ and any compact subset $\tilde{K} \subset W$,
$$
\| \omega_{E,l}(\lam) \|^2_{L^2(\tilde{K})}\leq C_{l,\tilde{K}} \lam^{ -l-1/2}
$$
for those $\lam < \lam_{l,0}$, and also 
$$
\|e^{\lam\psi^{}_{\scalebox{.7}{$\scriptscriptstyle E$}}}( \Delta_f (e^{-\lam\psi^{}_{\scalebox{.7}{$\scriptscriptstyle E$}}} \sum_{i=0}^N \omega_{E,i}(\lam)) - d_f^*( e^{-\lam\psi^{}_{\scalebox{.7}{$\scriptscriptstyle S$}}} \sum_{i=0}^N \omega_{S,i}\lam^{-i/2}))\|^2_{L^2(\tilde{K})} \leq C_{N,\tilde{K}} \lam^{-N-1/2},
$$
for $\lam > \lam_{N,0}$. After establishing the estimate of WKB iteration in section \ref{estimate}, we need to show that it is a good approximation as stated in theorem \ref{thm:homotopy_wkb}. The proof is actually a slight modification of \cite[Theorem 5.8]{HelSj1}. 
\begin{theorem}\label{thm:homotopy_wkb}
For any $\supp(\chi_S)$ and $K$ small enough, and $N$ large enough, there exists $\lam_{j,N,0}>0$ such that for $\lam > \lam_{j,N,0}$ we have
\begin{equation}
\|e^{\lam \psi^{}_{\scalebox{.7}{$\scriptscriptstyle E$}}}\nabla^j\lbrace \zeta_E - e^{-\lam\psi^{}_{\scalebox{.7}{$\scriptscriptstyle E$}}} (\sum_{i=0}^{N} \omega_{E,i}(\lam) )\rbrace \|^2_{L^2(K)} \leq C_{j,N} \lam^{-N+2j}.
\end{equation} 
\end{theorem} 
\begin{proof}
Making use of lemma \ref{locality}, we again consider the equation \ref{tildewkbequation}. It suffices to show that the approximation works for $\tilde{\zeta}_E$ on some small enough pre-compact neighborhood $K$ of the flow line $\gamma$. We divide the proof into several steps. 

\textbf{Step 1:} As $\omega_{E,i}(\lam)$'s do not vanish on boundary of $\tilde{W}$, we first need to cut them off suitably for applying integration by part. $\omega_{E,i}(\lam)$'s, being defined by integrating along flow of $\tau$, have support as shown in the following figure \ref{fig:supportomega}. 

\begin{figure}[h]
\centering
\includegraphics[scale=0.2]{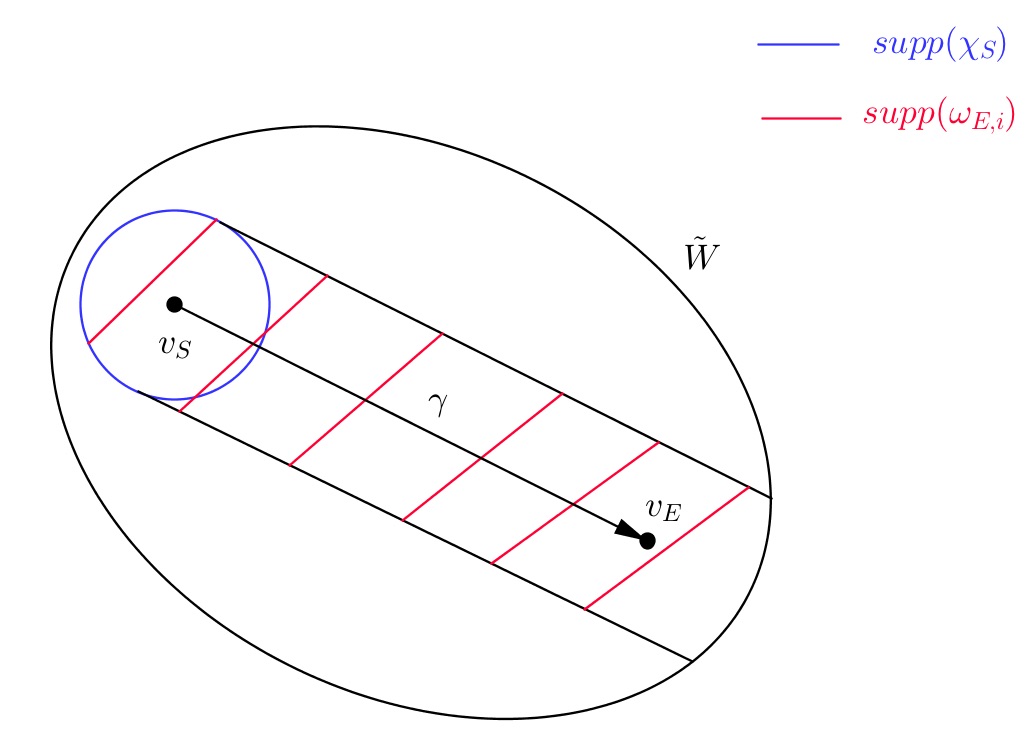}
\caption{Support of $\omega_{E,i}$'s}
\label{fig:supportomega}
\end{figure}

Suppose we have $\tau_{\tilde{T}}(v_S) = v_E$, then we can choose $\tilde{\chi}$ which only depends on variable $t$ (using coordinate defined by $\tau$) such that $\tilde{\chi} \equiv 1 $ for $t \leq \tilde{T}$. The support of $\nabla \tilde{\chi}$ is shown in the following figure \ref{fig:supportchi}. \\

\begin{figure}[h]
\centering
\includegraphics[scale=0.2]{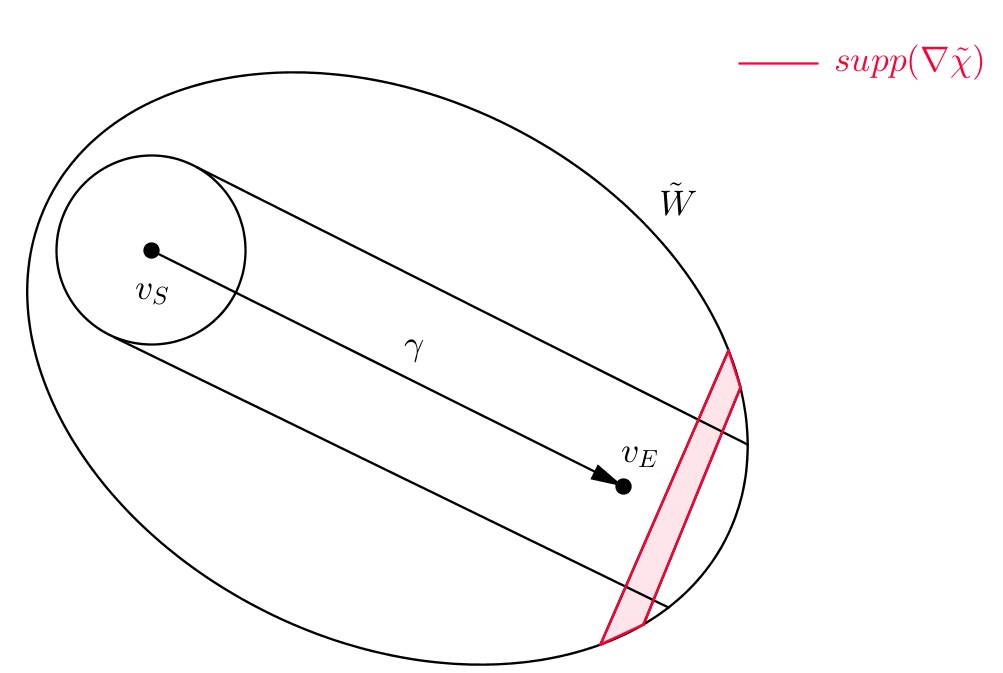}
\caption{Support of $\nabla \tilde{\chi}$}
\label{fig:supportchi}
\end{figure}

By shrinking $K$ and $\supp(\chi_S)$ if necessary, we obtain some $\epsilon>0$ such that 
\begin{equation}
\psi_E(y) + \dist(y,x) \geq \psi_E(x) + \epsilon
\end{equation}
for $x \in K$ and $y \in \supp(\nabla\tilde{\chi})$. We define the function 
\begin{equation}
\Phi_N = \min \{ \Phi + N\lam^{-1} \log(\lam), \min_{y \in \supp(\nabla \tilde{\chi})}(\Phi(y) + (1-\epsilon) \dist(x,y)) \},
\end{equation}
where $\Phi:=\psi_E -\frac{C}{\lam} \log(\lam\psi_E)$ is defined in \eqref{bigphi_definition}, and $\epsilon$ is chosen as in lemma \ref{convexlemma}. We have 
$$
|df|^2- |d \Phi_N|^2\geq \frac{C |df|^2}{\lam\psi_E}\geq \frac{C}{C_0^2\lam},
$$
for $\lam$ large enough. Notice that we have $\Phi_N = \Phi+ N\lam^{-1} \log(\lam)$ in $K$ for $\lam$ large enough, and $\Phi_N = \Phi$ in $\supp(\nabla \tilde{\chi})$. 

\textbf{Step 2:} Writing the reminder term as $ r_k = \tilde{\chi} (\tilde{\zeta}_E- e^{-\lam\psi^{}_{\scalebox{.7}{$\scriptscriptstyle E$}}} (\sum_{i=0}^{k-1} \omega_{E,i}(\lam) ))$, we get 
\begin{eqnarray*}
&& ( \| d(e^{\lam \Phi_N}r_k)\|^2_{L^2(K)}+\| d^*(e^{\lam\Phi_N}r_k)\|^2_{L^2(K)})+ \frac{C\lam^{1}}{2C_0^2} \| e^{\lam \Phi_N}  r_k \|^2_{L^2(K)}\\
&\leq & D \| e^{\lam\Phi_N} d_f^*(\chi_S \zeta_S- e^{-\lam\psi^{}_{\scalebox{.7}{$\scriptscriptstyle S$}}} \sum_{i=0}^{k-1}\chi_S\omega_{S,i}\lam^{-i/2}) \|^2_{L^2(\tilde{W})}\\
&+ &D \| e^{\lam \Phi_N}(d_f^*(e^{-\lam \psi^{}_{\scalebox{.7}{$\scriptscriptstyle S$}}} \sum_{i=0}^{k-1}\chi_S\omega_{S,i}\lam^{-i/2}) -\Delta_f( e^{-\lam\psi^{}_{\scalebox{.7}{$\scriptscriptstyle E$}}} \sum_{i=0}^{k-1}\omega_{E,i}(\lam)) )\|^2_{L^2(\tilde{W})}\\
&+ &D ( \| e^{\lam\Phi} [\Delta,\tilde{\chi}] \tilde{\zeta}_E\|^2_{L^2(\tilde{W})}+\| e^{\lam\Phi} [\Delta,\tilde{\chi}] (e^{-\lam\psi^{}_{\scalebox{.7}{$\scriptscriptstyle E$}}} \sum_{i=0}^{k-1}\omega_{E,i}(\lam)) \|^2_{L^2(\tilde{W})}).
\end{eqnarray*}
We handle the right hand side term by term. First, we have
$$
 \| e^{\lam\Phi_N} d_f^*(\chi_S \zeta_S - e^{-\lam\psi^{}_{\scalebox{.7}{$\scriptscriptstyle S$}}} \sum_{i=0}^{k-1}\chi_S\omega_{S,i} \lam^{-i/2})\|^2 \leq C_{k}\lam^{-2C+2N-k+2}.  
$$
Second, we have 
$$
\| e^{\lam \Phi_N}(d_f^*(e^{-\lam \psi^{}_{\scalebox{.7}{$\scriptscriptstyle S$}}} \sum_{i=0}^{k-1}\chi_S\omega_{S,i} \lam^{-i/2})-\Delta_f( e^{-\lam\psi^{}_{\scalebox{.7}{$\scriptscriptstyle E$}}} \sum_{i=0}^{k-1}\omega_{E,i}(\lam)) )\|^2 \leq C_k \lam^{-2C+2N-k+1}.
$$
Third, we have
$$
\| e^{\lam \Phi} [\Delta,\tilde{\chi}] \tilde{\zeta}_E \|^2 \leq D_1 \lam^{-2C+N_0},
$$
where $N_0$ is the integer in lemma \ref{wkbapriori}. Finally, we have
$$
\| e^{\lam \Phi} [\Delta,\tilde{\chi}] (e^{-\lam \psi^{}_{\scalebox{.7}{$\scriptscriptstyle E$}}} \sum_{i=0}^{k-1}\omega_{E,i}(\lam)) \|^2 \leq C_k \lam^{-2C+N_0},
$$
by choosing a larger $N_0$ independent of $k$, if necessary. Combining the above, by choosing $N = N_0 + k$, we have
$$
( \| d(e^{\lam \psi^{}_{\scalebox{.7}{$\scriptscriptstyle E$}}}r_k)\|^2_{L^2(K)}+\| d^*(e^{\lam \psi^{}_{\scalebox{.7}{$\scriptscriptstyle E$}}}r_k)\|_{L^2(K)}^2) + \lam \| e^{\lam \psi^{}_{\scalebox{.7}{$\scriptscriptstyle E$}}} r_k \|^2_{L^2(K)} \leq C_k \lam^{-k+2},
$$
which gives $ \| e^{\lam \psi^{}_{\scalebox{.7}{$\scriptscriptstyle E$}}} r_k \|^2_{L^2(K)} \leq C_k \lam^{-k+1}$, for those $\lam < \lam_{k,0}$. 

\textbf{Step 3:} We obtain $L^2$ estimate for all derivatives of $r_k$. We repeat the above argument for $d_f r_k$ and $d^*_f r_k$. For any $j, N \in  \inte_+$, we can find a $k_{j,N}$ large enough such that for any $k > k_{j,N}$, we have
$$
\| e^{\lam \psi^{}_{\scalebox{.7}{$\scriptscriptstyle E$}}} \nabla^j r_k \|^2_{L^2(K)}\leq C_{j,K,N} \lam^{-N}, 
$$
for $\lam > \lam_{j,k,N,0}$. 

\textbf{Step 4:} We apply interior Sobolev embedding to improve the statement in step $3$ into $L^{\infty}$ norm, by further shrinking $K$ if necessary. As a result, we have for $N$ large enough, there exists $\lam_{j,N,0}>0$ and $M_N$ such that we have
\begin{equation}
\|e^{\lam \psi^{}_{\scalebox{.7}{$\scriptscriptstyle E$}}}\nabla^j\lbrace \tilde{\zeta}_E - e^{-\lam \psi^{}_{\scalebox{.7}{$\scriptscriptstyle E$}}} (\sum_{i=0}^{M_N} \omega_{E,i}(\lam) )\rbrace \|^2_{L^\infty(K)} \leq C_{j,N} \lam^{-N+2j}
\end{equation}
for $\lam < \lam_{j,N,0}$. Finally, we observe that $\|  \nabla^j \omega_{E,i}(\lam)\|^2_{L^\infty(K)}\leq C_{i,j} \lam^{-i+j+\frac{1}{2}}$ and hence obtain the result by dropping redundant terms in the approximation series.

\end{proof}

Finally, we restrict on a sufficiently small neighborhood $W_E$ of $v_E$. Since the operator $I$ is given by an integral with an exponential decay $e^{\lam \Psi}$ along flow line, we can apply lemma \ref{stat_phase_exp} to obtain an expansion
$$
\omega_{E,i}(\lam) = \lam^{-\frac{1}{2}}(\omega_{E,i,0}+ \omega_{E,i,1}\lam^{-1} + \omega_{E,i,2}\lam^{-2}+\dots).
$$
By regrouping terms according to their orders of $\lam$, we obtain an expansion of the form given in equation \eqref{bwkbexpansion1}.

% !TEX root = paper.tex
%-------------------------------------relation between first order term------------------------------------

\subsection{Relation between $\omega_{S,0}$ and $\omega_{E,0}$}\label{leadingrelation}
From section \ref{WKB_method}, we constructed a WKB approximation in $W_E$\[
\zeta_E=e^{-\lam \psi^{}_{\scalebox{.7}{$\scriptscriptstyle E$}}}(\omega_{E,0}(\lam)+\omega_{E,1}(\lam)+\cdots).\]
In particular, $\omega_{E,0}(\lam)$ is given by
\begin{equation}\label{firstterm}
\omega_{E,0}(\lam)
= \frac{1}{2}(\int_{-\infty}^0 e^{\int_s^0 \frac{1}{2}\tau_{\epsilon}^*(M_{g^{}_{\scalebox{.6}{$\scriptscriptstyle E$}}})d\epsilon}\; \tau^*_s( e^{\lam \Psi} (\iota_{2\nabla f}+\iota_{\nabla g^{}_{\scalebox{.7}{$\scriptscriptstyle S$}}}) \chi_S \omega_{S,0} ) ds).
\end{equation}
In this section, we study the relation between integrals of $\omega_{S,0}$ and $\omega_{E,0}$ which is used in lemma \ref{homotopywkbcal}. We begin by recalling lemma \ref{stat_phase_exp_NB_1}. Let $M$ be a $n$-dimensional manifold and $S$ be a $k$-dimensional submanifold in $M$, with a neighborhood $B$ of $S$ which can be identified as the normal bundle $\pi: NS \rightarrow S$. Suppose $\varphi: B \rightarrow\real_{\geq0}$ is a Bott-Morse function with zero set $S$, we have

\begin{lemma}\label{stat_phase_exp_NB_2}
Let $\beta \in\Omega^*(B)$ which is vertically compact support along the fiber of $\pi$. Then, we have\[
\pi_*(e^{-\lam \varphi(x)}\beta) = (\frac{2\pi}{\lam})^{(n-k)/2} (\iota_{\vol(\Hess\varphi)}\beta)|_V (1+\order(\lam^{-1})),\]
where $\pi_*$ is the integration along fiber and $\vol(\Hess \varphi)$ is the volume polyvector field defined for the positive symmetric tensor $\Hess \varphi$ along fibers of $\pi$.
\end{lemma}

We use the notations in section \ref{exp_decay} and assume there is an identification of $W_S$ and $W_E$ with the normal bundle $NV_S$ and $NV_E$ of $V_S$ and $V_E$ respectively. We use $\pi_S$ and $\pi_E$ to stand for the bundle maps respectively. We have the following lemma which relates the integration of $\omega_{E,0}$ and $\omega_{S,0}$ along the fibers of $\pi_E$ and $\pi_S$ respectively.
\begin{lemma}\label{lem:explicit_normal_computation}
	Assume $\omega_{S,0} \in \wedge^{top} NV_S^*$ on $V_S$, then\[
\pi_{E*}(e^{-\lam g^{}_{\scalebox{.7}{$\scriptscriptstyle E$}}}\omega_{E,0}) =  \varrho^*\pi_{S*}(e^{-\lam g^{}_{\scalebox{.7}{$\scriptscriptstyle S$}}}\omega_{S,0})(1+\mathcal{O}(\lam^{-1})),\]
where $\varrho:V_E\rightarrow V_S$ is the projection map using the identification $V_E \equiv (V_S \times \real) \cap W_E$ given by $\tau$ (flow of $\nabla \psi_E$). Furthermore, we have $\omega_{E,0} \in \wedge^{top} NV_E^*$ on $V_E$.
\end{lemma}

\begin{proof}
We use the coordinates $u_1,\dots,u_{n-1},t$ for $W$, where $u_1,\dots ,u_{n-1}$ are coordinates of $U_S$. We further assume that $\{ u_{s+1} = 0, \dots ,u_{n-1} = 0\}=V_S$. From lemma \ref{psi_f_g}, $\Psi\leq0$ is a Bott-Morse function with zero set $U_S$. Applying lemma \ref{stat_phase_exp_NB_2} to the equation \eqref{firstterm}, we have
\begin{eqnarray*}
& &\omega_{E,0}(u,t) \\
&\equiv&(\frac{\pi}{2\lam})^{1/2}(\ppd{t} (-\Psi)|_{t=0})^{-1/2} \Big( e^{\int_{-t}^0 \frac{1}{2}\tau^*_\epsilon (M_{g^{}_{\scalebox{.6}{$\scriptscriptstyle E$}}})\;d\epsilon}\; \tau^*_{-t}((\iota_{2\nabla f}  +\iota_{\nabla g^{}_{\scalebox{.7}{$\scriptscriptstyle S$}}}) \chi_S\omega_{S,0})\Big) ,
\nonumber
\end{eqnarray*}
modulo terms of $\mathcal{O}(\lam^{-1})$. From lemma \ref{psi_f}, $g_E\geq0$ is a Bott-Morse function with zero set $V_E$. Applying lemma \ref{stat_phase_exp_NB_2} again, we get, modulo terms of $\mathcal{O}(\lam^{-1})$, 
\begin{eqnarray*}
&&\pi_{E*}(e^{-\lam g^{}_{\scalebox{.7}{$\scriptscriptstyle E$}}}\omega_{E,0})(u,t)\\
&\equiv&(\frac{2\pi}{\lam})^{(n-s-1)/2} \iota_{\vol(\nabla^2 g^{}_{\scalebox{.7}{$\scriptscriptstyle E$}})}(\omega_{E,0})\\
&\equiv&\pi \Big( (\frac{2\pi}{\lam})^{(n-s)/2} \iota_{\vol(\nabla^2 g^{}_{\scalebox{.7}{$\scriptscriptstyle E$}})}  (\ppd{t}(-\Psi)|_{t=0})^{-1/2}
 ( e^{\int_{-t}^0 \frac{1}{2}\tau^*_\epsilon (M_{g^{}_{\scalebox{.6}{$\scriptscriptstyle E$}}})\;d\epsilon}\; \tau^*_{-t}(\iota_{2\nabla f}   \omega_{S,0}))\Big),\\
\end{eqnarray*}
for those $(u,t) \in V_E$. The term involving $\iota_{\nabla g_S}$ is dropped as $\tau^*_{-t}(d g_S)$ vanishes for $(u,t) \in V_E$. To make further simplifications, we need the following lemma.
\begin{lemma}\label{simplification}
Fixing a point $(u,t) \in V_E$, we have\[
e^{\int_{-t}^0 \frac{1}{2}\tau^*_\epsilon (M_{g^{}_{\scalebox{.6}{$\scriptscriptstyle E$}}})d\epsilon} = \Big(\frac{\det( \Hess g^{}_E)(u,t)}{\det(\Hess g^{}_E)(u,0)}\Big)^{1/2}\]
as operators on $\bigwedge^{top} NV_E^*$, where the right hand side acts as multiplication. Here $\nabla^2 g_E$ is treated as an operator acting on $NV_E$ using the metric tensor.
\vspace{0.5cm}
\end{lemma}
\noindent From the fact that $ \omega_{S,0} \in \bigwedge^{top} NV_S^*$ upon restricting to $V_S$, we have $\tau^*_{-t} (\iota_{\nabla f} \omega_{S,0}) \in \bigwedge^{top} NV_E^*$ for those $(u,t) \in V_E$ and
$$
\begin{array}{rl}
&\pi_{E*}(e^{-\lam g^{}_{\scalebox{.7}{$\scriptscriptstyle E$}}}\omega_{E,0})(u,t) \\
= &2\pi (\frac{2\pi}{\lam})^{(n-s)/2}(\ppd{t}(-\Psi)|_{t=0})^{-1/2}\Big( (\frac{\det( \Hess g^{}_{\scalebox{.7}{$\scriptscriptstyle E$}})(u,t)}{\det(\Hess g^{}_{\scalebox{.7}{$\scriptscriptstyle E$}})(u,0)})^{1/2}  \iota_{ \nabla f \wedge\vol( \nabla^2 g^{}_{\scalebox{.7}{$\scriptscriptstyle E$}})} \tau^*_{-t} (\omega_{S,0}) \Big).
\end{array}
$$
Notice that $\nabla f = \dd{t}$ on $V_E$, therefore we have $$(\ppd{t} (-\Psi)|_{t=0})^{1/2} \nabla f = \vol( \nabla^2_t(-\Psi)|_{t=0}),$$
where we view $W$ as a $\real$-bundle over $U_S$ and consider $\vol( \nabla^2_t(-\Psi)|_{t=0})$ as the volume vector field along its fibers.
Furthermore, we have the relation
$$
d\tau^*_{-t}((\frac{\det( \Hess g^{}_E)(u,t)}{\det(\Hess g^{}_E)(u,0)})^{1/2} \vol(\nabla^2 g^{}_E)(u,t)) = \vol(\nabla^2 g^{}_E)(u,0).
$$

\noindent Combining the above, we have
$$
\begin{array}{rl}
&\pi_{E*}(e^{-\lam g^{}_{\scalebox{.7}{$\scriptscriptstyle E$}}}\omega_{E,0})(u,t)\\
=&(2\pi)^{(n-s)/2} \lam^{(-n+s)/2} \Big( \tau^*_{-t}( \iota_{\vol(\nabla^2_t(-\Psi)|_{t=0}) \wedge \vol(\nabla^2 g^{}_{\scalebox{.7}{$\scriptscriptstyle E$}})|_{t=0}} \omega_{S,0}) \Big).
\end{array}
$$
Finally, from the relation $\Psi = g_E - g_S$, we get 
$$\vol(\nabla^2_t(-\Psi)) \wedge \vol(\nabla^2 g^{}_E) = \vol(\nabla^2 g^{}_S)$$
on $V_S$, where $\vol(\nabla^2 g^{}_S)$ is the volume polyvector field along the fibers of $\pi_S$.
Therefore, we have\[
\pi_{E*}(e^{-\lam g^{}_{\scalebox{.7}{$\scriptscriptstyle E$}}}\omega_{E,0})(u,t) \equiv \tau^*_{-t}(\pi_{S*}(e^{-\lam g^{}_{\scalebox{.7}{$\scriptscriptstyle S$}}}\omega_{S,0})(u,0))\]
modulo terms of $\mathcal{O}(\lam^{-1})$, for those $(u,t) \in V_E$.
\end{proof}

\begin{proof}[Proof of Lemma \ref{simplification}]
First of all, we have the equality\[
\frac{1}{2}M_{g^{}_{\scalebox{.7}{$\scriptscriptstyle E$}}}=\Hess g^{}_E-\frac{1}{2}\trace(\Hess g^{}_E),\]
on the set $\{ \nabla g^{}_E = 0 \}$. We can treat $\Hess g^{}_E$ as an operator acting on $NV_E^*$ as $g^{}_E$ is Morse along $V_S$. Restricting to $\bigwedge^{top}NV_E^*$, it is just $\trace(\Hess g^{}_E)$. Therefore we have\[
\frac{1}{2}M_{g^{}_{\scalebox{.7}{$\scriptscriptstyle E$}}} = \frac{1}{2}\trace(\Hess g^{}_E),\]
acting on $\bigwedge^{top} NV_E^*$.\\
On $V_E$, we have
\begin{eqnarray}
&& \nabla_t \Big( \int_{0}^t \frac{1}{2}\trace(\Hess g^{}_E)(u,\epsilon)\;d\epsilon) - \frac{1}{2}\log(\det(\Hess_u g^{}_E)(u,t))\;\Big)\label{differentingg1}\\
\nonumber&=& \frac{1}{2} \trace(\Hess g^{}_E)(u,t) - \frac{1}{2}\trace((\Hess g^{}_E(u,t))^{-1} \nabla_t(\Hess g^{}_E(u,t))).
\end{eqnarray}
We will show that the above expression vanish.\\

Restricting on the set $\{ \nabla g^{}_E = 0 \}$, for any vector fields $X,Y \in TW$, we have
\begin{eqnarray*}
\nabla_t (\nabla^2_u g^{}_E)(X,Y) &=& \nabla_t (\nabla^2 g^{}_E(X,Y)) - \nabla^2 g^{}_E (\nabla_t X,Y) - \nabla^2 g^{}_E (X,\nabla_t Y)\\
&=& \nabla_t\langle X, \nabla_Y\nabla g^{}_E\rangle - \langle \nabla_t X,\nabla_Y\nabla g^{}_E \rangle - \langle \nabla_X\nabla g^{}_E,\nabla_t Y \rangle\\
&=& \langle X, \nabla_t\nabla_Y\nabla g^{}_E \rangle + \langle \nabla_X\nabla g^{}_E, [\pd_t,Y] \rangle + \langle \nabla_X\nabla g^{}_E, \nabla_Y\pd_t \rangle\\
&=& \langle X, \nabla_Y\nabla_t\nabla g^{}_E \rangle + \langle (\nabla^2t\nabla^2g^{}_E) X, Y\rangle,
\end{eqnarray*}
and
\begin{eqnarray*}
\nabla^2(\nabla_t g^{}_E)(X,Y) &=& \langle \nabla_Y\nabla(\pd_t g^{}_E), X \rangle\\
&=& Y \langle \nabla(\pd_t g^{}_E), X \rangle - \langle \nabla(\pd_t g^{}_E), \nabla_Y X \rangle\\
&=& Y \langle \nabla_X \nabla g^{}_E, \pd_t \rangle + Y \langle \nabla g^{}_E, \nabla_X\pd_t \rangle - \langle \nabla_{\nabla_Y X} \nabla g^{}_E, \pd_t \rangle\\
&=& Y \langle X, \nabla_t \nabla g^{}_E \rangle + Y \langle \nabla g^{}_E, \nabla_X \pd_t \rangle - \langle \nabla_Y X, \nabla_t \nabla g^{}_E \rangle\\
&=& \langle X, \nabla_Y\nabla_t\nabla g^{}_E \rangle + (\nabla^2 g^{}_E\nabla^2t) X, Y \rangle.
\end{eqnarray*}
Therefore, we have
$$
\nabla_t (\nabla^2 g^{}_E) - \nabla^2(\nabla_t g^{}_E)
= [\nabla^2 t , \nabla^2 g^{}_E],
$$
where the Hessians are treated as endomorphisms of $TM$. Restricting the above equation to the subspace $NV_E$ and multipling by $(\nabla^2 g^{}_E)^{-1}$, we have
$$
\trace ( (\nabla^2g^{}_E)^{-1} (\nabla_t(\nabla^2 g^{}_E)) = \trace ( (\nabla^2 g^{}_E)^{-1}  \nabla^2 (\nabla_t g^{}_E)). 
$$
Finally, from the equation $|\nabla \psi_E|^2 = | \nabla f |^2$, we obtain 
$$
\nabla_t g^{}_E = \half | \nabla g^{}_E |^2.
$$
Applying $\nabla^2$ to both sides and restricting to $V_E$, it gives
$$
\nabla^2 (\nabla_t g^{}_E)(X,Y) = \langle \nabla^2 g^{}_E(X) ,\nabla^2 g^{}_E(Y)\rangle,
$$
or simply 
$$
\nabla^2 (\nabla_t g^{}_E) =( \nabla^2 g^{}_E)^2
$$
if we treat both sides as operators on $TM$.\\

Substituting it back into equation \eqref{differentingg1}, we find that the derivative in equation \eqref{differentingg1} vanish. Therefore we have
\begin{eqnarray*}
&&(\int_{0}^t \frac{1}{2}\trace(\Hess g^{}_E)(u,\epsilon)\;d\epsilon)\\ 
&= &\frac{1}{2}\log(\det(\Hess g^{}_E)(u,t)) - \frac{1}{2}\log(\det(\Hess(g^{}_E))(u,0)),
\end{eqnarray*}
which is the equation we needed.
\end{proof}

Therefore, we complete the proof of lemma \ref{homotopywkb} and \ref{homotopywkbcal} which are needed in the proof of our main theorem in section \ref{proof}.

% !TEX root = paper.tex

\section{Conclusion}

From the semi-classical analysis of the Witten twisted Green's operator in section \ref{approximation}, we obtain our main theorem \ref{main_theorem} which can be viewed as an enhancement of the original Witten deformation of de Rham complex, concerning cohomology of the manifold $M$, to one concerning its rational homotopy type by incorporating wedge product structures. In \cite{fukaya05}, Fukaya proposed a differential geometric approach to the Strominger-Yau-Zaslow (SYZ) by relating A-model holomorphic disks instantons of a Calabi-Yau manifold equipped with Lagrangian torus fibration, to certain Witten twisted differential constructed from the symplectic structure. Proving theorem \ref{main_theorem} provides essential analytical technique for such an approach. For instance, the semi-classical analysis of Witten twisted Green's operator, can be applied to obtain a beautiful geometric interpretation of the complicated scattering diagram in \cite{kwchan-leung-ma}. 

{\bf Acknowledgements:} 
The authors thank Kwokwai Chan, Cheol-Hyun Cho, Yau Heng Tom Wan and Siye Wu for
useful discussions. Kaileung Chan would like to thank Siye Wu for his kind hospitality. Ziming Nikolas Ma would like to thank Prof. Shing-Tung Yau for the support during his visit to Yau Mathematical Sciences Center, Tsinghua University. Naichung Conan Leung were supported in part by grants from the Research Grants Council of the Hong Kong Council of the Hong Kong Special Administrative Region, China (Project No. CUHK 402012, Project No. CUHK 14302215 and Project No. CUHK 14303516).

\bibliographystyle{amsplain}
\bibliography{fukayawitten}

\providecommand{\bysame}{\leavevmode\hbox to3em{\hrulefill}\thinspace}
\providecommand{\MR}{\relax\ifhmode\unskip\space\fi MR }
\providecommand{\MRhref}[2]{%
  \href{http://www.ams.org/mathscinet-getitem?mr=#1}{#2}
}
\providecommand{\href}[2]{#2}
\begin{thebibliography}{10}

\bibitem{a1}
M.~Abouzaid, \emph{Morse homology, tropical geometry, and homological mirror
  symmetry for toric varieties}, Geom. Topol. \textbf{10} (2006), 1097--1157.
  \MR{2529936, Zbl 1204.14019}

\bibitem{La92}
J.-M. Bismut and W.~Zhang, \emph{{A}n {E}xtension of a {T}heorem of {C}heeger
  and {M}\"uller}, Ast\'erisque (1992), no.~No. 205. \MR{1185803, Zbl
  0781.58039}

\bibitem{kwchan-leung-ma}
K.-W. Chan, N.~C. Leung, and Z.~N. Ma, \emph{Scattering diagrams from
  asymptotic analysis on {M}aurer-{C}artan equations}, arXiv preprint
  arXiv:1807.08145 (2018).

\bibitem{DiSj}
M.~Dimassi and J.~Sj\"ostrand, \emph{Spectral asymptotics in the semi-classical
  limit}, London Mathematical Society Lecture Note Series, vol. 268, Cambridge
  University Press, 1999. \MR{1735654, Zbl 0926.35002}

\bibitem{fukayamorse}
K.~Fukaya, \emph{Morse homotopy, {$A^\infty$}-category, and {F}loer
  homologies}, Proceedings of {GARC} {W}orkshop on {G}eometry and {T}opology
  '93 ({S}eoul, 1993), Lecture Notes Ser., vol.~18, Seoul Nat. Univ., Seoul,
  1993, pp.~1--102. \MR{1270931, Zbl 0853.57030}

\bibitem{fukaya05}
\bysame, \emph{Multivalued {M}orse theory, asymptotic analysis and mirror
  symmetry}, Graphs and patterns in mathematics and theoretical physics, Proc.
  Sympos. Pure Math., vol.~73, Amer. Math. Soc., Providence, RI, 2005,
  pp.~205--278. \MR{2131017, Zbl 1085.53080}

\bibitem{fukaya-oh}
K.~Fukaya and Y.-G. Oh, \emph{Zero-loop open strings in the cotangent bundle
  and {M}orse homotopy}, Asian J. Math. \textbf{1} (1997), no.~1, 96--180.
  \MR{1480992, Zbl 0938.32009}

\bibitem{harvey01}
F.~R. Harvey and H.~B. Lawson~Jr, \emph{Finite volume flows and {M}orse
  theory}, Ann. of Math. (2) \textbf{153} (2001), no.~1, 1--25. \MR{1826410,
  Zbl 1001.58005}

\bibitem{helffer2006semi}
B.~Helffer, \emph{Semi-classical analysis for the {S}chr{\"o}dinger operator
  and applications}, vol. 1336, Springer, 2006. \MR{0960278, Zbl 0647.35002}

\bibitem{HelNi}
B.~Helffer and F.~Nier, \emph{Hypoelliptic estimates and spectral theory for
  {F}okker-{P}lanck operators and {W}itten {L}aplacians}, Lecture Notes in
  Mathematics, vol. 1862, Springer-Verlag, Berlin, 2005. \MR{2130405, Zbl
  1072.35006}

\bibitem{HelSj1}
B.~Helffer and J.~Sj\"ostrand, \emph{Multiple wells in the semi-classical limit
  {I}}, Comm. in PDE \textbf{9} (1984), no.~4, 337--408. \MR{740094, Zbl
  0546.35053}

\bibitem{HelSj2}
\bysame, \emph{Puits multiples en limite semi-classique {I}{I} - {I}nteraction
  mol\'eculaire-{S}ym\'etries-{P}erturbations}, Annales de l'IHP(section
  Physique th\'eorique) \textbf{42} (1985), no.~2, 127--212. \MR{0798695, Zbl
  0595.35031}

\bibitem{HelSj4}
\bysame, \emph{Puits multiples en limite semi-classique {I}{V} - {E}tdue du
  complexe de {W}itten}, Comm. in PDE \textbf{10} (1985), no.~3, 245--340.
  \MR{780068, Zbl 0597.35024}

\bibitem{kontsevich00}
M.~Kontsevich and Y.~Soibelman, \emph{Homological mirror symmetry and torus
  fibrations}, Symplectic geometry and mirror symmetry ({S}eoul, 2000), World
  Sci. Publ., River Edge, NJ, 2001, pp.~203--263. \MR{1882331, Zbl 1072.14046}

\bibitem{quillen1969}
D.~Quillen, \emph{Rational homotopy theory}, Ann. of Math. \textbf{90} (1969),
  205--295. \MR{258031, Zbl 0191.53702}

\bibitem{sullivan1977}
D.~Sullivan, \emph{Infinitesimal computations in topology}, Publications
  Math{\'e}matiques de l'IH{\'E}S \textbf{47} (1977), no.~1, 269--331.
  \MR{0646078, Zbl 0374.57002}

\bibitem{witten82}
E.~Witten, \emph{Supersymmetry and {M}orse theory}, J. Differential Geom.
  \textbf{17} (1982), no.~4, 661--692 (1983). \MR{683171, Zbl 0499.53056}

\bibitem{zhang}
W.~Zhang, \emph{Lectures on chern-weil theory and witten deformations}, Nankai
  Tracts in Mathematics, vol.~4, World Scientific, 2001. \MR{1864735, Zbl
  0993.58014}

\end{thebibliography}

\end{document}